\documentclass[10pt]{article}

\usepackage{latexsym, amssymb, amsmath, amsthm, amscd}
\usepackage[all,cmtip]{xy}
\usepackage{amsmath}
\usepackage{amsthm}
\usepackage{amssymb}
\usepackage{amsfonts}
\usepackage{amsrefs}
\usepackage{color,authblk}
\usepackage{tikz}
\usepackage{amscd} 
\usepackage{comment}
\usepackage{mathrsfs}
\usepackage{comment}
\usepackage[all]{xy}
\usepackage{graphicx}
\usepackage{authblk}
\usepackage{hyperref}
\usepackage{fullpage}
\usepackage[all,cmtip]{xy}
\usepackage{relsize}
\usepackage{amsmath,amscd}
\usepackage{tikz-cd}
\usepackage{tikz}

\usepackage[mathcal]{euscript}
\usepackage{txfonts}
\usetikzlibrary{matrix}
\usepackage[all]{xy}

\newcommand{\leftrarrows}{\mathrel{\raise.75ex\hbox{\oalign{%
  $\scriptstyle\leftarrow$\cr
  \vrule width0pt height.5ex$\hfil\scriptstyle\relbar$\cr}}}}
\newcommand{\lrightarrows}{\mathrel{\raise.75ex\hbox{\oalign{%
  $\scriptstyle\relbar$\hfil\cr
  $\scriptstyle\vrule width0pt height.5ex\smash\rightarrow$\cr}}}}
\newcommand{\Rrelbar}{\mathrel{\raise.75ex\hbox{\oalign{%
  $\scriptstyle\relbar$\cr
  \vrule width0pt height.5ex$\scriptstyle\relbar$}}}}

\makeatletter
\def\leftrightarrowsfill@{\arrowfill@\leftrarrows\Rrelbar\lrightarrows}
\newcommand{\xleftrightarrows}[2][]{\ext@arrow 3399\leftrightarrowsfill@{#1}{#2}}
\makeatother

\usepackage{hyperref}

\numberwithin{equation}{section} \DeclareMathSizes{2}{10}{12}{13}
\makeatletter
\newcommand*{\doublerightarrow}[2]{\mathrel{
  \settowidth{\@tempdima}{$\scriptstyle#1$}
  \settowidth{\@tempdimb}{$\scriptstyle#2$}
  \ifdim\@tempdimb>\@tempdima \@tempdima=\@tempdimb\fi
  \mathop{\vcenter{
    \offinterlineskip\ialign{\hbox to\dimexpr\@tempdima+1em{##}\cr
    \rightarrowfill\cr\noalign{\kern.5ex}
    \rightarrowfill\cr}}}\limits^{\!#1}_{\!#2}}}
\newcommand*{\triplerightarrow}[1]{\mathrel{
  \settowidth{\@tempdima}{$\scriptstyle#1$}
  \mathop{\vcenter{
    \offinterlineskip\ialign{\hbox to\dimexpr\@tempdima+1em{##}\cr
    \rightarrowfill\cr\noalign{\kern.5ex}
    \rightarrowfill\cr\noalign{\kern.5ex}
    \rightarrowfill\cr}}}\limits^{\!#1}}}
\makeatother

\newtheorem{thm}{Proposition}[section]
\newtheorem{Thm}[thm]{Theorem}
\newtheorem{rem}[thm]{Remark}

\newtheorem{cor}[thm]{Corollary}

\newtheorem{lem}[thm]{Lemma}
\newtheorem{defn}[thm]{Definition}

\title{On representation categories of $A_\infty$-algebras and $A_\infty$-coalgebras}

\author{Abhishek Banerjee \footnote{Department of Mathematics, Indian Institute of Science, Bangalore, India, Email: abhishekbanerjee1313@gmail.com} $\qquad\qquad\qquad$  Anita Naolekar \footnote{Stat-Math Unit, Indian Statistical Institute, Bangalore, India, Email: anita@isibang.ac.in}}

\date{}

\begin{document}

\maketitle 

\medskip

\begin{abstract} In this paper, we use the language of monads, comonads and Eilenberg-Moore categories  to describe a categorical framework for $A_\infty$-algebras and $A_\infty$-coalgebras, as well as $A_\infty$-modules and $A_\infty$-comodules over them respectively. The resulting formalism leads us to investigate relations between representation categories of $A_\infty$-algebras and $A_\infty$-coalgebras. In particular, we relate $A_\infty$-comodules and $A_\infty$-modules by considering a rational pairing between an $A_\infty$-coalgebra $C$ and an $A_\infty$-algebra $A$. The categorical framework also motivates us to introduce $A_\infty$-contramodules over an  $A_\infty$-coalgebra $C$. 
\end{abstract}

\medskip
MSC(2020) Subject Classification: 18C15, 18C20, 18G70


\medskip

\medskip

\small
\hypersetup{linktocpage} 
\tableofcontents

\normalsize
\medskip

\section{Introduction} 

A monad on a category $\mathcal C$ consists of an endofunctor $\mathbb U:\mathcal C\longrightarrow \mathcal C$ along with natural transformations
$\mathbb U\circ\mathbb U\longrightarrow \mathbb U$ and $1\longrightarrow \mathbb U$ satisfying conditions similar to an ordinary ring. As with a ring, one can study a monad by looking at its representation category, i.e., its category of modules. The role of modules over the monad $\mathbb U$ is played by the classical
Eilenberg-Moore category $EM_{\mathbb U}$ (see \cite{EM65}, \cite{EMoore}). For instance, Borceux, Caenepeel and Janelidze \cite{BCJ} have shown that 
Eilenberg-Moore categories can be used to study a categorical form of Galois descent, extending from the usual notion of Galois extension of commutative rings. The dual notion of comonad resembles that of a coalgebra, and we have Eilenberg-Moore categories of comodules over them. For instance, Eilenberg-Moore categories of comodules can be used to study notions such as relative Hopf modules, or Doi-Hopf modules (see  \cite{BC}).

\smallskip
In this paper, we consider $A_\infty$-versions of monads, comonads and their representation categories.  We develop Eilenberg-Moore categories for $A_\infty$-monads and $A_\infty$-comonads. We study a bar construction for $A_\infty$-monads. We also study versions of  distributive laws that enable us to lift $A_\infty$-monads to the Eilenberg-Moore category, or extend
$A_\infty$-monads to the Kleisli category, of an ordinary monad. This gives us a categorical framework for $A_\infty$-algebras and $A_\infty$-coalgebras, as well as $A_\infty$-modules and $A_\infty$-comodules over them respectively, using the language of monads, comonads and Eilenberg-Moore categories. We are then motivated by this formalism to consider rational pairings between $A_\infty$-coalgebras and $A_\infty$-algebras, and also to introduce $A_\infty$-contramodules. 

\smallskip
We recall that an $A_\infty$-algebra $A$ consists of a graded vector space $A=\underset{n\in\mathbb Z}{\bigoplus}A_n$ and a sequence $\{m_k:A^{\otimes k}\longrightarrow A\}_{k\geq 1}$ of homogeneous linear maps  satisfying higher order associativity conditions. The $A_\infty$-algebras were introduced by Stasheff \cite{Stasheff1}, \cite{Stasheff2}, and applied to the study of  homotopy associative structures on $H$-spaces. In later years, $A_\infty$-structures have found widespread applications in algebra, topology, geometry as well as in physics (see, for instance, \cite{Adams}, \cite{BV}, \cite{KF}, \cite{GJ}, \cite{TK},  \cite{BK}, \cite{KM}, \cite{KS}, \cite{May}, \cite{Smirnov}, \cite{St3}). 

\smallskip
We  investigate relations between three ``module like'' categories associated to $A_\infty$-(co)algebras: modules, comodules and contramodules. We define a rational pairing between an $A_\infty$-coalgebra $C$ and an $A_\infty$-algebra $A$. This gives us an adjunction between the category of $A_\infty$-modules over $A$ and the $A_\infty$-comodules over $C$. This extends the classical theory that embeds the comodules over a coalgebra as a coreflective subcategory of the modules over its dual algebra (see, for instance, \cite[Chapter 1]{BWis}). 
 
 \smallskip On the other hand, if $C'$ is an ordinary coassociative coalgebra, we know that there are two different representation categories of $C'$. One is the usual category of $C'$-comodules and the other is the category of $C'$-contramodules, introduced by Eilenberg and Moore in \cite[IV.5]{EMoore}. At a categorical level, both are dual to the notion of modules over an algebra, which makes the study of contramodules as fundamental as that of comodules. Accordingly, our monadic approach motivates us to introduce $A_\infty$-contramodules over an $A_\infty$-coalgebra $C$. We also introduce a contratensor product that gives us an adjunction between categories of $A_\infty$-comodules and $A_\infty$-contramodules.  For the contratensor product in the case of contramodules over a coassociative coalgebra, we refer the reader to \cite{semicontra}.

\smallskip We should mention that our approach  to $A_\infty$-structures in this paper using monads and comonads is inspired by B\"{o}hm, Brzezi\'{n}ski and Wisbauer, who provided a unifying framework for studying modules, comodules and contramodules over ordinary algebras and coalgebras in \cite{BBW}. As such, we hope that this is the first step towards studying notions such as mixed distributive laws, entwining structures and Galois properties over $A_\infty$-algebras and $A_\infty$-coalgebras, as well as relative Hopf modules, Doi-Hopf modules, and Yetter-Drinfeld modules. The literature on the latter is vast and has been developed widely by several authors (see, for instance, \cite{BBR1}, \cite{BBR2}, \cite{Brz-}, \cite{BC}, \cite{Cn0}, \cite{Cn1}, \cite{Cn2}, \cite{Cn3}). 

\smallskip
We now describe the paper in more detail. We begin in Section 2 with an abelian category $\mathcal C$ that is $K$-linear (for a field $K$) and satisfies (AB5). An $A_\infty$-monad $(\mathbb U,\Theta)$ consists of a sequence $\mathbb U=\{\mathbb U_n\}_{n\in \mathbb Z}$ of endofunctors on $\mathcal C$ along with a collection $\Theta$ of natural transformations
\begin{equation}\Theta=\{\mbox{$\theta_k(\bar{n}):\mathbb U_{\bar{n}}:=\mathbb U_{n_1}\circ ...\circ \mathbb U_{n_k}\longrightarrow \mathbb U_{n_1+...+n_k+(2-k)}=\mathbb U_{\Sigma \bar{n}+(2-k)}$ $\vert$ $k\geq 1$, $\bar{n}=(n_1,...,n_k)\in \mathbb Z^k$} \}
\end{equation} satisfying certain relations. We say that $(\mathbb U,\Theta)$ is even if $\theta_k(\bar{n})=0$ whenever $k$ is odd. Dually, an $A_\infty$-comonad $(\mathbb V,\Delta)$ consists of a sequence  $\mathbb V=\{\mathbb V_n\}_{n\in \mathbb Z}$ of endofunctors and a collection $\Delta$ of natural transformations \begin{equation}\Delta=\{\mbox{$\delta_k(\bar{n}):\mathbb V_{\Sigma \bar{n}+(2-k)}:= \mathbb V_{n_1+...+n_k+(2-k)}\longrightarrow \mathbb V_{n_1}\circ ...\circ \mathbb V_{n_k}=\mathbb V_{\bar{n}}$ $\vert$ $k\geq 1$, $\bar{n}=(n_1,...,n_k)\in \mathbb Z^k$} \}
\end{equation}  satisfying dual relations.  Our first result is that if $\{(\mathbb U_n,\mathbb V_n)\}_{n\in\mathbb Z}$ is a sequence of adjoint functors, then 
$\mathbb U=\{\mathbb U_n\}_{n\in\mathbb Z}$ can be equipped with the structure of an even $A_\infty$-monad (resp. an even $A_\infty$-comonad) if and only if $\mathbb V=\{\mathbb V_n\}_{n\in \mathbb Z}$ can be equipped with the structure of an even $A_\infty$-comonad (resp. an even $A_\infty$-monad). Additionally, if $\mathbb U=\{\mathbb U_n\}_{n\in\mathbb Z}$ is an even $A_\infty$-monad, we describe a process that gives a collection $\mathbb V^\oplus
=\{\mathbb V^\oplus_n\}_{n\in \mathbb Z}$ of subfunctors of $\mathbb V=\{\mathbb V_n\}_{n\in \mathbb Z}$ carrying the structure of an $A_\infty$-comonad $(\mathbb V^\oplus, \Delta^\oplus)$ that is locally finite. In other words, for each $k\geq 1$, $T\in \mathbb Z$, the induced morphism
\begin{equation}
\mathbb V_{T+(2-k)}^\oplus= \mathbb V_{n_1+...+n_k+(2-k)}^\oplus\longrightarrow \underset{\bar{n}\in\mathbb Z^k, \Sigma\bar{n}=T}{\prod}\mathbb V_{n_1}^\oplus\circ ...\circ \mathbb V_{n_k}^\oplus
\end{equation} factors through the direct sum $ \underset{\bar{n}\in\mathbb Z^k, \Sigma\bar{n}=T}{\bigoplus}\mathbb V^\oplus_{n_1}\circ ...\circ \mathbb V_{n_k}^\oplus$.

\smallskip
If $(\mathbb U,\Theta)$ is an $A_\infty$-monad, an $A_\infty$-module $(M,\pi)$ consists of a collection $M=\{M_n\}_{n\in\mathbb Z}$ of objects of $\mathcal C$ along with morphisms \begin{equation}
\pi=\{\mbox{$\pi_k(\bar{n},l):\mathbb U_{\bar{n}}M_{l}\longrightarrow M_{\Sigma\bar{n}+l+(2-k)}$ $\vert$ $k\geq 1$, $\bar{n}=(n_1,...,n_{k-1})
\in \mathbb Z^{k-1}$, $l\in \mathbb Z$ } \}
\end{equation}
 satisfying certain conditions described in Definition \ref{D2.4}. This gives the Eilenberg-Moore category $EM_{(\mathbb U,\Theta)}$ of $A_\infty$-modules over $(\mathbb U,\Theta)$. We also describe a canonical functor $\mathcal C\longrightarrow EM_{(\mathbb U,\Theta)}$. Further, if the functors
$\{\mathbb U_n\}_{n\in \mathbb Z}$ preserve direct sums, this extends to a functor $\hat{\mathbb U}:\mathcal C^{\mathbb Z}\longrightarrow EM_{(\mathbb U,\Theta)}$ from the category $\mathcal C^{\mathbb Z}$ of $\mathbb Z$-graded objects over $\mathcal C$. The $A_\infty$-comodules over an $A_\infty$-comonad $(\mathbb V,\Delta)$ are defined in a dual manner. If  $\{(\mathbb U_n,\mathbb V_n)\}_{n\in\mathbb Z}$ is a sequence of adjoint functors, we show (see Proposition \ref{P2.7}) that the category of even $A_\infty$-modules over $(\mathbb U,\Theta)$ is isomorphic to the category of even $A_\infty$-comodules over $(\mathbb V,\Delta)$. We also define $\infty$-morphisms $\alpha=\{\mbox{ $\alpha_k(\bar{n},l):\mathbb U_{\bar{n}}M^1_l\longrightarrow M^2_{\Sigma\bar{n}+l+(1-k)}$}\}$ between $A_\infty$-modules $(M^1,\pi^1)$, $(M^2,\pi^2)$ (see Definition \ref{D2.8}). Then, we show in Theorem \ref{T2.10} that even $A_\infty$-modules over $(\mathbb U,\Theta)$ with odd $\infty$-morphisms between them correspond to 
even $A_\infty$-comodules over $(\mathbb V,\Delta)$ with odd $\infty$-morphisms between them. 

\smallskip
We study locally finite $A_\infty$-comodules over $A_\infty$-comonads in Section 4. Let $(\mathbb V,\Delta)$ be an $A_\infty$-comonad such that each $\mathbb V_n$ is exact. If $P=\{P_n\}_{n\in \mathbb Z}$ is an $A_\infty$-comodule, we describe a process of restricting the structure maps $\rho_k(\bar{n},l):P_{\Sigma\bar{n}+l+(2-k)}\longrightarrow \mathbb V_{\bar{n}}P_{l} $  of $P$ to a subobject $P^\oplus$ such that for each $k\geq 1$,
$l\in \mathbb Z$, $T\in\mathbb Z$, the induced morphism
\begin{equation}
\underset{\bar{n}\in\mathbb Z^{k-1}, \Sigma\bar{n}+l=T}{\prod} \rho_k^\oplus(\bar{n},l):P^\oplus_{T+(2-k)}\longrightarrow \underset{\bar{n}\in\mathbb Z^{k-1}, \Sigma\bar{n}+l=T}{\prod}\mathbb V_{\bar{n}}P^\oplus_l
\end{equation}
factors through the direct sum $\underset{\bar{n}\in\mathbb Z^{k-1}, \Sigma\bar{n}+l=T}{\bigoplus}\mathbb V_{\bar{n}}P^\oplus_l$. Further, we show (see Theorem \ref{T2.3.6}) that this construction gives a right adjoint to the  inclusion of locally finite $A_\infty$-comodules into the Eilenberg-Moore category $EM^{(\mathbb V,\Delta)}$ of $A_\infty$-comodules over $(\mathbb V,\Delta)$.

\smallskip
In Section 5, we develop the bar construction for an $A_\infty$-monad $(\mathbb U,\Theta)$. This gives a comonad $Bar(\mathbb U,\Theta)$ which is graded and equipped with a coderivation. If $(\mathbb W,\delta_1,\delta_2)$ is a conilpotent dg-comonad (see Definition \ref{conil7}), we show that morphisms from $(\mathbb W,\delta_1,\delta_2)$ to  $Bar(\mathbb U,\Theta)$ correspond to families of natural transformations satisying conditions similar to the classical case of twisting morphisms. In Section 6, we study distributive laws between 
an $A_\infty$-monad $(\mathbb U,\Theta)$ and an ordinary monad $\mathbb S$. We show how these distributive laws can be used to lift $(\mathbb U,\Theta)$  to an $A_\infty$-monad
$(\widetilde{\mathbb U},\widetilde{\Theta})$ on the Eilenberg-Moore category $EM_{\mathbb S}$ of $\mathbb S$. Similarly, we show how distributive laws can be used to extend 
$(\mathbb U,\Theta)$  to an $A_\infty$-monad
$(\widehat{\mathbb U},\widehat{\Theta})$ on the Kleisli category $Kl(\mathbb S)$ of $\mathbb S$.

\smallskip
In Section 7, we relate this formalism of monads and comonads to $A_\infty$-algebras and $A_\infty$-coalgebras. Motivated by a classical construction of Popescu \cite{Pop} (see also Artin and Zhang \cite{AZ}) on the noncommutative base change of an algebra, we construct examples of $A_\infty$-monads and $A_\infty$-comonads starting with a $K$-linear Grothendieck category $\mathcal C$. If $A$ is an $A_\infty$-algebra and $\mathbb F$ is a monad on $\mathcal C$ that preserves direct sums,  the collection of functors $\mathbb U^{A,\mathbb F}_{\mathcal C}=\{\mathbb U_{\mathcal C,n}^{A,\mathbb F}:=A_n\otimes \mathbb F(\_\_):\mathcal C\longrightarrow \mathcal C\}_{n\in \mathbb Z}
$ is an $A_\infty$-monad on $\mathcal C$. Additionally, if $A$ is even (i.e., the operations $m_k:A^{\otimes k}\longrightarrow A$ vanish whenever $k$ is odd) and $\mathbb F$
has a right adjoint $\mathbb G$, the collection of functors $\mathbb V^{A,\mathbb G}_{\mathcal C}=\{\mathbb V_{\mathcal C,n}^{A,\mathbb G}:=\underline{Hom}(A_n, \mathbb G(\_\_)):\mathcal C\longrightarrow \mathcal C\}_{n\in \mathbb Z}
$ is an $A_\infty$-comonad on $\mathcal C$ (see Proposition \ref{scuP5.2}). We also mention classes of $A_\infty$-algebras that satisfy this evenness condition and have been studied extensively in the literature (see \cite{DoVa}, \cite{HeLu},  \cite{Kelx}, \cite{LPWZ}).

\smallskip Henceforth, we fix  $\mathcal C=Vect$, the category of $K$-vector spaces. If $A$ is a $\mathbb Z$-graded vector space, we show that the collection of functors $\mathbb U^A=\{\mathbb U_n^A:=A_n\otimes \_\_:Vect\longrightarrow Vect\}_{n\in \mathbb Z}
$ can be equipped with the structure of an $A_\infty$-monad if and only if $A$ can be equipped
with the structure of an $A_\infty$-algebra. In that case, the Eilenberg-Moore category $EM_{\mathbb U_A}$ of $A_\infty$-modules over $\mathbb U_A$ corresponds to $A_\infty$-modules over the $A_\infty$-algebra $A$.  On the other hand, if $C$ is an $A_\infty$-coalgebra that is even, it  follows from our earlier result in Theorem \ref{T2.3} that the right adjoints  $\{Hom(C_{-n},\_\_):Vect\longrightarrow Vect\}_{n\in \mathbb Z}
$ can be equipped with the structure of an $A_\infty$-monad. This is the key observation that motivates our definition of $A_\infty$-contramodules in Section 9. 

\smallskip
By a pairing of an $A_\infty$-algebra $A$ and an $A_\infty$-coalgebra $C$, we will mean an $\infty$-morphism of $A_\infty$-algebras $f= \{f_k: A^{\otimes k}\longrightarrow C^*\}_{k\geq 1}$ from $A$ to the graded dual $C^*$ of $C$. We show that such a pairing induces a functor $\iota(C,A):\mathbf M^C\longrightarrow {_A}\mathbf M$ from right $C$-comodules to left $A$-modules. Additionally, when the pairing is rational in the sense of Definition \ref{D6.5qh}, the functor $\iota(C,A)$ embeds $\mathbf M^C$ as a full subcategory of ${_A}\mathbf M$. Our main result in Section 8 is that a rational pairing  $f= \{f_k: A^{\otimes k}\longrightarrow C^*\}_{k\geq 1}$ of an $A_\infty$-algebra with an $A_\infty$-coalgebra gives rise to an adjunction (see Theorem \ref{T6.8nh})
\begin{equation} 
{_A}\mathbf M(\iota(C,A)(M'),M)\cong \mathbf M^C(M',R_f(M))
\end{equation} for any $M'\in \mathbf M^C$ and $M\in{_A}\mathbf M$.

\smallskip
We introduce $A_\infty$-contramodules in Section 9. If $A'$ is an ordinary algebra, we note that the structure map $A'\otimes M'\longrightarrow M'$ of a module $M'$ may be expressed equivalently as a morphism $M'\longrightarrow [A',M']:=Vect(A',M')$.  Accordingly, a contramodule over an ordinary coalgebra $C'$, as defined by Eilenberg and Moore in \cite[IV.5]{EMoore}, consists of a space $M''$ along with a structure map $Vect(C',M'')=[C',M'']\longrightarrow M''$ satisfying certain coassociativity conditions. In particular, if $V$ is any vector space, then $[C',V]$ is a $C'$-contramodule. In a categorical sense, it may be said therefore that  both comodules and contramodules dualize the notion of module over an algebra. Even though the study of contramodules in the literature is not as developed as that of comodules, the topic has seen a lot of renewed interest in recent years  (see, for instance, \cite{BBR}, \cite{BPS}, \cite{Pst2}, \cite{BBW}, \cite{semicontra}, \cite{Pmem}, \cite{P2},   \cite{Sha}, \cite{Wis}). 

\smallskip
Accordingly, our notion of an $A_\infty$-contramodule over an $A_\infty$-coalgebra $C$ consists of a graded vector space $M\in Vect^{\mathbb Z}$ along with a collection of structure
maps
\begin{equation}
t^M_k: Vect^{\mathbb Z}(C^{\otimes k-1},M)=[C^{\otimes k-1},M]\longrightarrow M\qquad k\geq 1
\end{equation} satisfying certain conditions set out in Definition \ref{D7.1wc}. We show that there is a canonical functor $[C,\_\_]:Vect^{\mathbb Z}\longrightarrow \mathbf M_{[C,\_\_]}$ from graded vector spaces to the category $\mathbf M_{[C,\_\_]}$ of $A_\infty$-contramodules over $C$. Further, we describe a canonical functor that relates $A_\infty$-contramodules to the Eilenberg-Moore category of $A_\infty$-modules over the $A_\infty$-monad  $\{Hom(C_{-n},\_\_):Vect\longrightarrow Vect\}_{n\in \mathbb Z}
$ described in Section 7. There is also a faithful functor $\kappa^C$ that takes $ \mathbf M_{[C,\_\_]}$ to $A_\infty$-modules over $C^*$. 

\smallskip
The final result of Section 9 relates $A_\infty$-contramodules to $A_\infty$-comodules, using contratensor products. For contramodules over coassociative coalgebras, the contratensor product was defined in \cite{semicontra}, along with its right adjoint. We let $C$, $D$ be $A_\infty$-coalgebras. Let $N$ be a space equipped with an $A_\infty$-left $C$-coaction as well as an $A_\infty$-right $D$-coaction (a $(C,D)$-bicomodule for instance) satisfying certain conditions. Using $N$, we define a contratensor product  (see \eqref{ctrtens})
\begin{equation}
\_\_\boxtimes_C N : \mathbf M_{[C,\_\_]}\longrightarrow \mathbf M^D
\end{equation}  We conclude with Theorem \ref{T8.4co}, which gives an adjunction of functors
  \begin{equation} 
  \mathbf M^D(M\boxtimes_CN, Q)\cong \mathbf M_{[C,\_\_]}(M,[N,Q]^D)
  \end{equation} for $M\in \mathbf M_{[C,\_\_]}$ and $Q\in \mathbf M^D$. We also mention that for the convenience of the reader and in order to fix notation, we have collected the basic definitions of $A_\infty$-algebras, $A_\infty$-coalgebras, $A_\infty$-modules and $A_\infty$-comodules in an appendix in the final section of this paper (see, for instance, \cite{BK}, \cite{Kenji}).

\smallskip

\medskip

{\bf Acknowledgements:} We are grateful to Jim Stasheff for valuable discussions on $A_\infty$-algebras and $A_\infty$-coalgebras.

\section{$A_\infty$-monads, $A_\infty$-comonads and adjoints}

Let $K$ be a field. Throughout this section and the rest of this paper, we let $\mathcal C$ be a $K$-linear abelian category. We will suppose that $\mathcal C$ satisfies the (AB5) axiom. This means  in particular that for any family $\{M_i\}_{i\in I}$ of objects in $\mathcal C$, we have an inclusion $\underset{i\in I}{\bigoplus}M_i\hookrightarrow \underset{i\in I}{\prod}M_i$. For any tuple $\bar{n}=(n_1,...,n_k)\in\mathbb Z^k$, we set $|\bar{n}|:=k$ and $\Sigma \bar{n}:=n_1+...+n_k$.  We also write  $\bar{n}^{op}$ for the opposite tuple
$(n_k,...,n_1)\in \mathbb Z^k$. On the other hand, for $T\in \mathbb Z$ and $k\geq 1$, we denote by $\mathbb Z(T,k)$ the set of all $\bar{n}\in \mathbb Z^k$ satisfying
$\Sigma\bar{n}=T$. 

\smallskip
 For any collection $\mathbb U=\{\mathbb U_n:\mathcal C\longrightarrow \mathcal C\}_{n\in \mathbb Z}$ of endofunctors on $\mathcal C$ indexed by $\mathbb Z$ and $\bar{n}=(n_1,...,n_k)\in \mathbb Z^k$, we will denote by $\mathbb U_{\bar{n}}$ the composition 
$\mathbb U_{\bar{n}}:=\mathbb U_{n_1}\circ ...\circ \mathbb U_{n_k}$. We start by considering an $A_\infty$-monad on $\mathcal C$. 

\begin{defn}\label{D2.1} An $A_\infty$-monad $(\mathbb U,\Theta)$ on $\mathcal C$ consists of the following data:

\smallskip
(a) A collection $\mathbb U=\{\mathbb U_n:\mathcal C\longrightarrow \mathcal C\}_{n\in \mathbb Z}$ of endofunctors on $\mathcal C$.

\smallskip
(b) A collection $\Theta$ of natural transformations 
\begin{equation}\label{2.1trf} \Theta=\{\mbox{$\theta_k(\bar{n}):\mathbb U_{\bar{n}}=\mathbb U_{n_1}\circ ...\circ \mathbb U_{n_k}\longrightarrow \mathbb U_{n_1+...+n_k+(2-k)}=\mathbb U_{\Sigma \bar{n}+(2-k)}$ $\vert$ $k\geq 1$, $\bar{n}=(n_1,...,n_k)\in \mathbb Z^k$} \}
\end{equation}
 satisfying, for each $\bar{z}\in \mathbb Z^N$,  $N\geq 1$:
\begin{equation}\label{2.1e}
0=\sum (-1)^{p+qr+q\Sigma\bar{n}}\theta_{p+1+r}(\bar{n},\Sigma\bar{n}'+(2-q),\bar{n}'')(\mathbb U_{\bar{n}}\ast \theta_q(\bar{n}')\ast \mathbb U_{\bar{n}''}):\mathbb U_{\bar{n}}\circ \mathbb U_{\bar{n}'}\circ \mathbb U_{\bar{n}''}\longrightarrow \mathbb U_{\Sigma \bar{n}+\Sigma\bar{n}'+\Sigma\bar{n}''+(3-N)}
\end{equation} where the sum runs over partitions $\bar{z}=(\bar{n}, \bar{n}',\bar{n}'')$ with $|\bar{n}|=p$, $|\bar{n}'|=q$, $|\bar{n}''|=r$. 
We will say that $(\mathbb U,\Theta)$  is even if each $\theta_k(\bar{n})=0$ whenever $k$ is odd.
\end{defn}

\begin{defn}\label{D2.2}
An $A_\infty$-comonad $(\mathbb V,\Delta)$ on $\mathcal C$ consists of the following data:

\smallskip
(a) A collection $\mathbb V=\{\mathbb V_n:\mathcal C\longrightarrow \mathcal C\}_{n\in \mathbb Z}$ of endofunctors on $\mathcal C$.

\smallskip
(b) A collection $\Delta$ of natural transformations 
\begin{equation}\Delta=\{\mbox{$\delta_k(\bar{n}):\mathbb V_{\Sigma \bar{n}+(2-k)}= \mathbb V_{n_1+...+n_k+(2-k)}\longrightarrow \mathbb V_{n_1}\circ ...\circ \mathbb V_{n_k}=\mathbb V_{\bar{n}}$ $\vert$ $k\geq 1$, $\bar{n}=(n_1,...,n_k)\in \mathbb Z^k$} \}
\end{equation}
 such that, for each $\bar{z}\in \mathbb Z^N$,  $N\geq 1$, we have
\begin{equation}\label{2.1ezr}
0=\sum (-1)^{pq+r+q\Sigma\bar{n}}(\mathbb V_{\bar{n}}\ast \delta_q(\bar{n}')\ast \mathbb V_{\bar{n}''})\delta_{p+1+r}(\bar{n},\Sigma\bar{n}'+(2-q),\bar{n}''): \mathbb V_{\Sigma \bar{n}+\Sigma\bar{n}'+\Sigma\bar{n}''+(3-N)}\longrightarrow \mathbb V_{\bar{n}}\circ \mathbb V_{\bar{n}'}\circ \mathbb V_{\bar{n}''}
\end{equation} where the sum runs over partitions $\bar{z}=(\bar{n}, \bar{n}',\bar{n}'')$ with $|\bar{n}|=p$, $|\bar{n}'|=q$, $|\bar{n}''|=r$.  

\smallskip
We will say that an $A_\infty$-comonad $(\mathbb V,\Delta)$ is locally finite if 
for each $k\geq 1$, $T\in \mathbb Z$, the induced morphism 
\begin{equation}\label{2.3lof}
\underset{\bar{n}\in\mathbb Z(T,k)}{\prod} \delta_k(\bar{n}):\mathbb V_{T+(2-k)}= \mathbb V_{n_1+...+n_k+(2-k)}\longrightarrow \underset{\bar{n}\in\mathbb Z(T,k)}{\prod}\mathbb V_{n_1}\circ ...\circ \mathbb V_{n_k}=\underset{\bar{n}\in\mathbb Z(T,k)}{\prod} \mathbb V_{\bar{n}}
\end{equation}
factors through the direct sum $ \underset{\bar{n}\in\mathbb Z(T,k)}{\bigoplus}\mathbb V_{\bar{n}}$.

\smallskip
We will say that $(\mathbb V,\Delta)$  is even if each $\delta_k(\bar{n})=0$ whenever $k$ is odd.
\end{defn}

In any category $\mathcal A$ containing products and direct sums, we will say that a collection  $\{\eta_i:X\longrightarrow X_i\}_{i\in I}$ of morphisms  is locally finite if the induced morphism $\underset{i\in I}{\prod}\eta_i:X\longrightarrow \underset{i\in I}{\prod}X_i$ factors through the direct sum $\underset{i\in I}{\bigoplus}X_i$.
We now make two conventions that we will use for notation throughout the paper.

\smallskip
(1) Suppose $(U,V)$ is a pair of adjoint functors between categories $\mathcal A$ and $\mathcal B$, so that we have natural isomorphisms
\begin{equation}
\mathcal B(UX,Y)\cong \mathcal A(X,VY)
\end{equation} for objects $X\in \mathcal A$, $Y\in \mathcal B$.  For any morphism $f\in \mathcal B(UX,Y)$, we denote by
$f^R$ the corresponding morphism in $\mathcal A(X,VY)$. Conversely, for any $g\in \mathcal A(X,VY)$, we denote by $g^L$ the corresponding morphism
in $\mathcal B(UX,Y)$.  

\smallskip
(2) Suppose $(U_1,V_1)$ and $(U_2,V_2)$ are pairs of adjoint functors between categories $\mathcal A$ and $\mathcal B$. Then, we know that there are isomorphisms
\begin{equation}
Nat(U_1,U_2)\cong Nat(V_2,V_1)
\end{equation} For any natural transformation $\eta \in Nat(U_1,U_2)$, we denote by $\eta^R$ the corresponding natural transformation in $Nat(V_2,V_1)$. Conversely, for any
$\zeta\in Nat(V_2,V_1)$, we denote by $\zeta^L$ the corresponding natural transformation in $Nat(U_1,U_2)$. 

\smallskip
We will say that $(U=\{U_n:\mathcal A\longrightarrow \mathcal B\}_{n\in \mathbb Z},V=\{V_n:\mathcal B\longrightarrow\mathcal A\}_{n\in 
\mathbb Z})$ is a $\mathbb Z$-system of adjoints between  categories $\mathcal A$ and $\mathcal B$ if $(U_n,V_n)$ is a pair of adjoint
functors for each $n\in \mathbb Z$.

\begin{Thm}\label{T2.3} Let $(\mathbb U=\{\mathbb U_n:\mathcal C\longrightarrow \mathcal C\}_{n\in\mathbb Z},\mathbb V=
\{\mathbb V_n:\mathcal C\longrightarrow\mathcal C\}_{n\in\mathbb Z})$ be a $\mathbb Z$-system of adjoints on $\mathcal C$.   Then, 

\smallskip
(a) $\mathbb U$ can be equipped with the structure of an  even $A_\infty$-monad if and only if $\mathbb V$ can be equipped with the structure of
an even  $A_\infty$-comonad.

\smallskip
(b) $\mathbb U$ can be equipped with the structure of an  even $A_\infty$-comonad if and only if $\mathbb V$ can be equipped with the structure of
an  even $A_\infty$-monad.

\end{Thm}

\begin{proof}
We only prove (a) because (b) is similar. Suppose that $(\mathbb U, \Theta=\{\theta_k(\bar{n})\})$ is an even $A_\infty$-monad. Accordingly, we have transformations 
\begin{equation}
(\theta_k(\bar{n}):\mathbb U_{\bar{n}}\longrightarrow \mathbb U_{\Sigma\bar{n}+(2-k)})\mapsto (\theta_k^R(\bar{n}):\mathbb V_{\Sigma\bar{n}+(2-k)}\longrightarrow \mathbb V_{\bar{n}^{op}})
\end{equation} Using  the notation of \eqref{2.1e}, we see that  the  $\theta^R_k(\bar{n})$ satisfy, for any $\bar{z}\in \mathbb Z^N$, $N\geq 1$, 
\begin{equation}\label{2.8yh}
\sum_{\bar{z}=(\bar{n},\bar{n}',\bar{n}'')} (-1)^{p}(\mathbb V_{\bar{n}''^{op}}\ast \theta_q^R(\bar{n}')\ast \mathbb V_{\bar{n}^{op}})\theta_{p+1+r}^R(\bar{n},\Sigma\bar{n}'+(2-q),\bar{n}'')=0
\end{equation}
 as a transformation $\mathbb V_{\Sigma \bar{n}+\Sigma\bar{n}'+\Sigma\bar{n}''+(3-N)}\longrightarrow \mathbb V_{\bar{n}''^{op}}\circ \mathbb V_{\bar{n}'^{op}}\circ \mathbb V_{\bar{n}^{op}}$ (note that $q$ must be even).  We now set $\delta_k(\bar{n}):=\theta_k^R(\bar{n}^{op})$.  Accordingly,   the condition in \eqref{2.8yh} now becomes 
 \begin{equation}\label{2.9yh}
\sum_{\bar{z}=(\bar{n},\bar{n}',\bar{n}'')}  (-1)^{p}(\mathbb V_{\bar{n}''^{op}}\ast \delta_q(\bar{n}'^{op})\ast \mathbb V_{\bar{n}^{op}})\delta_{p+1+r}(\bar{n}''^{op},\Sigma\bar{n}'^{op}+(2-q),\bar{n}^{op})=0
\end{equation} Comparing with \eqref{2.1ezr}, we see that $(\mathbb V,\Delta=\{\delta_k(\bar{n})\})$  satisfies the conditions for being an even $A_\infty$-comonad. Further, these arguments can be reversed, showing that if $\mathbb V$ carries the structure of an even $A_\infty$-comonad, then
$\mathbb U$ carries structure of an even $A_\infty$-monad. This proves the result.
\end{proof}

We will now describe the connection between $A_\infty$-monads and locally finite $A_\infty$-comonads. For this, we will first prove some simple lemmas on the collection $End(\mathcal C)$ of endofunctors in $\mathcal C$. If $F\in End(\mathcal C)$, a subfunctor $F'\hookrightarrow F$ is a natural transformation such that $F'(M)\longrightarrow F(M)$ is a monomorphism in $\mathcal C$ for each $M\in \mathcal C$.  In order to avoid set theoretic complications, we will assume wherever necessary that the category $End(\mathcal C)$ is well-powered, i.e., the collection of subfunctors of any $F\in End(\mathcal C)$ forms a set. 

\begin{lem}\label{FLem2.3}Let $F$, $G\in End(\mathcal C)$. Suppose that $G$ preserves monomorphisms. Then, if $F''\hookrightarrow F'\hookrightarrow F$ and $G''\hookrightarrow G'\hookrightarrow G$ are subfunctors, the composition $G''\circ F''$ is a subfunctor of $G'\circ F'$.
\end{lem}
\begin{proof} We take some $M\in \mathcal C$ and consider the following commutative diagram
\begin{equation}\label{2.6va}
\begin{CD}
@. G\circ F''(M)@>mono>> G\circ F'(M) @>mono>>G\circ F(M)\\
@. @AmonoAA @AmonoAA @.\\
G''\circ F''(M)@>mono>> G'\circ F''(M) @>mono>> G'\circ F'(M)@. \\
\end{CD}
\end{equation} Since $G$ preserves monomorphisms, the morphisms  in the top row of \eqref{2.6va} are all monomorphisms. Further, the commutativity of the square in \eqref{2.6va} shows that $G'\circ F''(M)\longrightarrow G'\circ F'(M)$ is a monomorphism. The composition in the bottom row now gives the desired result.

\end{proof}

\begin{lem}\label{FLem2.4} Let $\{\eta_i:F\longrightarrow F_i\}_{i\in I}$ be a family of morphisms in $End(\mathcal C)$. Fix an indexing set $A$.  For each $i\in I$, let $\{F_i^\alpha\}_{\alpha\in A}$ be a family of subfunctors of $F_i$ and let  $\{F^\alpha\}_{\alpha\in A}$ 
be a family of subfunctors of $F$ such that $\eta_i:F\longrightarrow F_i$ restricts to $\eta_i^\alpha:F^\alpha\longrightarrow F_i^\alpha$. For each $\alpha\in A$, suppose that  $\{\eta_i^\alpha:F^\alpha\longrightarrow F^\alpha_i\}_{i\in I}$ is locally finite in $End(\mathcal C)$.  Then, the family 
\begin{equation}
\left\{\eta_i^A=\sum_{\alpha\in A} \eta_i^\alpha : F^A=\sum_{\alpha\in A} F^\alpha \longrightarrow \sum_{\alpha\in A} F^\alpha_i= F_i^A \right\}_{i\in I}
\end{equation} is locally finite.

\end{lem}

\begin{proof} Since $\{\eta_i^\alpha:F^\alpha\longrightarrow F^\alpha_i\}_{i\in I}$ is locally finite for each $\alpha\in A$, we have morphisms  $F^\alpha
\longrightarrow \underset{i\in I}{\bigoplus}F^\alpha_i\longrightarrow \underset{i\in I}{\prod}F^\alpha_i$. Summing over all $\alpha\in A$, we can consider for each $i_0\in I$ the following composition 
\begin{equation}\label{2.8ydo} \sum_{\alpha\in A}F^\alpha
\longrightarrow\underset{\alpha\in A}{\sum}\underset{i\in I}{\bigoplus}F^\alpha_i  =\underset{i\in I}{\bigoplus}\underset{\alpha\in A}{\sum}F^\alpha_i \longrightarrow \underset{\alpha\in A}{\sum}F^\alpha_{i_0}
\end{equation} The second morphism in \eqref{2.8ydo} is given by the canonical projection corresponding to $i_0\in I$. The interchange in the middle term in \eqref{2.8ydo} is justified by the fact that $\mathcal C$ satisfies (AB5).  As we vary over all $i_0\in I$, the compositions in \eqref{2.8ydo} give us a factorization
\begin{equation}
\sum_{\alpha\in A}F^\alpha
\longrightarrow \underset{i\in I}{\bigoplus}\underset{\alpha\in A}{\sum}F^\alpha_i \longrightarrow \underset{i\in I}{\prod} \underset{\alpha\in A}{\sum}F^\alpha_{i}
\end{equation}
This proves the result. 
\end{proof}

\begin{lem}\label{FLem2.5} Let $F$, $G\in End(\mathcal C)$. Suppose that $G$ preserves monomorphisms. Let $\{F^\alpha\}_{\alpha\in A}$ (resp. $\{G^\alpha\}_{\alpha\in A}$)
be a family of subfunctors of $F$ (resp. $G$). Then, $\sum_{\alpha\in A}G^\alpha F^\alpha$ is a subfunctor of $(\sum_{\alpha\in A}G^\alpha)(\sum_{\alpha\in A}F^\alpha)$. 
\end{lem}

\begin{proof}
For each fixed $\alpha_0\in A$, we have subfunctors $F^{\alpha_0}\hookrightarrow (\sum_{\alpha\in A}F^\alpha)\hookrightarrow F$ and $G^{\alpha_0}\hookrightarrow (\sum_{\alpha\in A}G^\alpha)\hookrightarrow G$. Applying Lemma \ref{FLem2.3}, it follows that each $G^{\alpha_0}F^{\alpha_0}$ is a subfunctor of $(\sum_{\alpha\in A}G^\alpha)(\sum_{\alpha\in A}F^\alpha)$. The result is now clear by summing over all $\alpha_0\in A$. 
\end{proof}

We now consider a $\mathbb Z$-system $(\mathbb U=\{\mathbb U_n:\mathcal C\longrightarrow \mathcal C\}_{n\in \mathbb Z},\mathbb V=
\{\mathbb V_n:\mathcal C\longrightarrow\mathcal C\}_{n\in \mathbb Z})$   of adjoints on $\mathcal C$. Suppose that $(\mathbb U,\Theta)$ is an $A_\infty$-monad, given by a collection of natural transformations
\begin{equation}\Theta=\{\mbox{$\theta_k(\bar{n}):\mathbb U_{\bar{n}}=\mathbb U_{n_1}\circ ...\circ \mathbb U_{n_k}\longrightarrow \mathbb U_{n_1+...+n_k+(2-k)}=\mathbb U_{\Sigma \bar{n}+(2-k)}$ $\vert$ $k\geq 1$, $\bar{n}=(n_1,...,n_k)\in \mathbb Z^k$} \}
\end{equation} satisfying the condition in \eqref{2.1e}. If $(\mathbb U,\Theta)$ is even, setting $\delta_k(\bar{n})=\theta^R_k(\bar{n}^{op})$ as in the proof of Theorem \ref{T2.3}, we know that $(\mathbb V,\Delta)$ is an  $A_\infty$-comonad, where
\begin{equation}\label{2.16wf} \Delta=\{\mbox{$\delta_k(\bar{n}):\mathbb V_{\Sigma \bar{n}+(2-k)}= \mathbb V_{n_1+...+n_k+(2-k)}\longrightarrow \mathbb V_{n_1}\circ ...\circ \mathbb V_{n_k}=\mathbb V_{\bar{n}}$ $\vert$ $k\geq 1$, $\bar{n}=(n_1,...,n_k)\in \mathbb Z^k$} \}
\end{equation}
We now consider a family $\mathbb V^\alpha=\{\mathbb V^\alpha_n\hookrightarrow \mathbb V_n \}_{n\in \mathbb Z}$ of subfunctors satisfying the following conditions

\smallskip
(a) For any $k\geq 1$ and $\bar{n}=(n_1,...,n_k)\in \mathbb Z^k$, each of the  transformations $\delta_k(\bar{n})$ in \eqref{2.16wf} restricts to the corresponding subfunctors
\begin{equation}\label{2.17wf} \Delta^\alpha=\{\mbox{$\delta_k^\alpha(\bar{n}):\mathbb V_{\Sigma \bar{n}+(2-k)}^\alpha= \mathbb V_{n_1+...+n_k+(2-k)}^\alpha\longrightarrow \mathbb V_{n_1}^\alpha\circ ...\circ \mathbb V_{n_k}^\alpha=\mathbb V_{\bar{n}}^\alpha$ $\vert$ $k\geq 1$, $\bar{n}=(n_1,...,n_k)\in \mathbb Z^k$} \}
\end{equation} We note here that since each $\mathbb V_n$ is a right adjoint, it preserves monomorphisms. As such, it follows from Lemma \ref{FLem2.3} that each $\mathbb V_{\bar{n}}^\alpha=\mathbb V_{n_1}^\alpha\circ ...\circ \mathbb V_{n_k}^\alpha$ is a subfunctor of $\mathbb V_{\bar{n}}=\mathbb V_{n_1}\circ ...\circ \mathbb V_{n_k}$. 

\smallskip
(b) For each $k\geq 1$ and $T\in \mathbb Z$, the collection
\begin{equation}\label{2.18yj}
\left\{ \delta_k^\alpha(\bar{n}):\mathbb V^\alpha_{T+(2-k)}= \mathbb V^\alpha_{n_1+...+n_k+(2-k)}\longrightarrow  \mathbb V^\alpha_{n_1}\circ ...\circ \mathbb V^\alpha_{n_k}=  \mathbb V^\alpha_{\bar{n}}\right\}_{\bar{n}\in \mathbb Z(T,k)}
\end{equation}
is locally finite. 

\begin{lem}\label{FLem2.6}
$(\mathbb V^\alpha,\Delta^\alpha)$ is a locally finite $A_\infty$-comonad. 
\end{lem}

\begin{proof}
By assumption, the natural transformations $\delta_k(\bar{n})$ restrict to transformations $\delta^\alpha_k(\bar{n})$. Since each $\mathbb V_n$ preserves monomorphisms, we see that all the terms appearing in \eqref{2.1ezr} restrict to corresponding subfunctors. As such, these transformations  $\delta^\alpha_k(\bar{n})$ satisfy the condition in \eqref{2.1ezr}. Hence, $(\mathbb V^\alpha,\Delta^\alpha)$ is an $A_\infty$-comonad. The fact that $(\mathbb V^\alpha,\Delta^\alpha)$ is locally finite follows directly from the assumption in
\eqref{2.18yj}.
\end{proof}

We now let $\{(\mathbb V^\alpha,\Delta^\alpha)\}_{\alpha\in A}$ denote the collection of such families of subfunctors. We note that this collection is non-empty since it contains the family of zero subfunctors. We now define 
\begin{equation}\label{219jw}
\mathbb V^{\oplus}=\left\{\mathbb V^{\oplus}_n:=\sum_{\alpha\in A}\mathbb V^\alpha_n\right\}_{n\in \mathbb Z}
\end{equation}

\begin{thm}\label{P2.8gb} The collection $\mathbb V^{\oplus}=\{\mathbb V^{\oplus}_n\}_{n\in\mathbb Z}$ is equipped with the structure of a locally finite  $A_\infty$-comonad.
\end{thm}

\begin{proof} For any $k\geq 1$ and $\bar{n}=(n_1,...,n_k)\in \mathbb Z^k$, we consider
\begin{equation}\label{2.20kh}
\delta_k^{\oplus}(\bar{n})=\sum_{\alpha\in A} \delta_k^\alpha(\bar{n}):\sum_{\alpha\in A}\mathbb V_{\Sigma \bar{n}+(2-k)}^\alpha \longrightarrow \sum_{\alpha\in A}\mathbb V_{n_1}^\alpha\circ ...\circ \mathbb V_{n_k}^\alpha\hookrightarrow \left(\sum_{\alpha\in A}\mathbb V_{n_1}^\alpha\right)\circ ...\circ \left(\sum_{\alpha\in A}\mathbb V_{n_k}^\alpha\right)
\end{equation} where the inclusion in \eqref{2.20kh} follows from Lemma \ref{FLem2.5}. Again since each $\mathbb V_n$ preserves monomorphisms, we see that all the transformations appearing in \eqref{2.1ezr} restrict to corresponding subfunctors. It follows that the transformations $\delta^{\oplus}_k(\bar{n}):\mathbb V_{\Sigma \bar{n}+(2-k)}^{\oplus}
\longrightarrow \mathbb V_{n_1}^{\oplus}\circ ...\circ \mathbb V_{n_k}^{\oplus}$ satisfy the condition in \eqref{2.1ezr}, making $\left(\mathbb V^{\oplus},\Delta^{\oplus}=
\{\mbox{$\delta^{\oplus}_k(\bar{n})$ $\vert$ $k\geq 1$, $\bar{n}\in \mathbb Z^k$} \}\right)$ an   $A_\infty$-comonad. 

\smallskip
Additionally, using Lemma \ref{FLem2.4}, we see that for each $k\geq 1$ and $T\in \mathbb Z$, we have a factorization
\begin{equation}\label{2.20sq}
\begin{tikzcd}[row sep = huge, column sep = 30pt]
\sum_{\alpha\in A}\mathbb V^\alpha_{T+(2-k)} \arrow[r] \arrow[dr] & \underset{\bar{n}\in \mathbb Z(T,k)}{\prod}  \sum_{\alpha\in A}\mathbb V_{n_1}^\alpha\circ ...\circ \mathbb V_{n_k}^\alpha\arrow[r] &  \underset{\bar{n}\in \mathbb Z(T,k)}{\prod} \left(\sum_{\alpha\in A}\mathbb V_{n_1}^\alpha\right)\circ ...\circ \left(\sum_{\alpha\in A}\mathbb V_{n_k}^\alpha\right)  \\
 &  \underset{\bar{n}\in \mathbb Z(T,k)}{\bigoplus}  \sum_{\alpha\in A}\mathbb V_{n_1}^\alpha\circ ...\circ \mathbb V_{n_k}^\alpha \arrow[r] \arrow[u] &   \underset{\bar{n}\in \mathbb Z(T,k)}{\bigoplus} \left(\sum_{\alpha\in A}\mathbb V_{n_1}^\alpha\right)\circ ...\circ \left(\sum_{\alpha\in A}\mathbb V_{n_k}^\alpha\right)\arrow[u]\\
\end{tikzcd}
\end{equation} From \eqref{2.20sq}, it is clear that $(\mathbb V^{\oplus},\Delta^{\oplus})$ is locally finite.

\end{proof}

\begin{rem}\emph{On the other hand, suppose that  $(\mathbb U=\{\mathbb U_n:\mathcal C\longrightarrow \mathcal C\}_{n\in \mathbb Z},\mathbb V=
\{\mathbb V_n:\mathcal C\longrightarrow\mathcal C\}_{n\in \mathbb Z})$ is a $\mathbb Z$-system of adjoints such that $\mathbb V$ carries the structure of an  
even $A_\infty$-monad. By Theorem \ref{T2.3}, we know that $\mathbb U=\{\mathbb U_n\}_{n\in \mathbb Z}$ is canonically equipped with an  
$A_\infty$-comonad structure. Additionally, if we suppose that each $\mathbb U_n$ preserves monomorphisms, we can construct by a process similar to  \eqref{219jw} a family
$\mathbb U^{\oplus}=\{\mathbb U^{\oplus}_n\hookrightarrow \mathbb U_n\}_{n\in \mathbb Z}$ of subfunctors equipped with a locally finite $A_\infty$-comonad structure.}
\end{rem}

\section{Eilenberg-Moore categories for $A_\infty$-monads and $A_\infty$-comonads}

We  now consider modules over $A_\infty$-monads and comodules over $A_\infty$-comonads.

\begin{defn}\label{D2.4}
Let $(\mathbb U,\Theta)$ be an $A_\infty$-monad over $\mathcal C$. A $(\mathbb U,\Theta)$-module $(M,\pi)$ consists of the following data:

\smallskip
(a) A collection $M=\{M_n\}_{n\in \mathbb Z}$ of objects of $\mathcal C$.

\smallskip
(b) A collection of morphisms
\begin{equation}
\pi=\{\mbox{$\pi_k(\bar{n},l):\mathbb U_{\bar{n}}M_{l}\longrightarrow M_{\Sigma\bar{n}+l+(2-k)}$ $\vert$ $k\geq 1$, $\bar{n}=(n_1,...,n_{k-1})
\in \mathbb Z^{k-1}$, $l\in \mathbb Z$ } \}
\end{equation}
 satisfying, for each $\bar{z}\in \mathbb Z^N$,  $N\geq 0$ and $l\in\mathbb Z$,
\begin{equation}\label{l2.11}
\begin{array}{l}
\underset{\bar{z}=(\bar{n},\bar{n}',\bar{n}'')}{\sum} (-1)^{p+qr+q\Sigma\bar{n}}\pi_{p+1+r}(\bar{n},\Sigma \bar{n}'+(2-q), \bar{n}'', l)\circ (\mathbb U_{\bar{n}}\ast \theta_q(\bar{n}'))(\mathbb U_{\bar{n}''}M_{l})\\
\qquad\qquad =-\underset{\bar{z}=(\bar{m},\bar{m}')}{\sum} (-1)^{a+b\Sigma\bar{m}}\pi_{a+1}(\bar{m},\Sigma\bar{m}'+l+(2-b))\circ \mathbb U_{\bar{m}} (\pi_b(\bar{m}',l))\\
\end{array}
\end{equation} On the left hand side of \eqref{l2.11}, we have $|\bar{n}|=p$, $|\bar{n}'|=q$ and $|\bar{n}''|=r-1$, while on the right hand side, we have
$|\bar{m}|=a$ and $|\bar{m}'|=b-1$. Accordingly, both sides of \eqref{l2.11} represent morphisms 
\begin{equation}
\mathbb U_{\bar{n}}\mathbb U_{\bar{n}'}\mathbb U_{\bar{n}''}M_{l}=\mathbb U_{\bar{m}}\mathbb U_{\bar{m}'}M_{l}\longrightarrow M_{\Sigma \bar{m}+\Sigma \bar{m}'+l+(2-N)}=M_{\Sigma \bar{n}+\Sigma \bar{n}'+\Sigma \bar{n}''+l+(2-N)}
\end{equation} A morphism $f:(M,\pi)\longrightarrow (M',\pi')$ of $\mathbb U$-modules consists of a collection $f=\{f_l:M_l\longrightarrow M'_l\}_{l\in \mathbb Z}$ such that 
$\pi_k'(\bar{n},l)\circ \mathbb U_{\bar{n}}(f_l)=f_{\Sigma\bar{n}+l+(2-k)}\circ \pi_k(\bar{n},l)$ for every $k\geq 1$, $\bar{n}=(n_1,...,n_{k-1})
\in \mathbb Z^{k-1}$, $l\in \mathbb Z$ . We denote the category of  $(\mathbb U,\Theta)$-modules by $EM_{(\mathbb U,\Theta)}$. 

\smallskip
We will say that a $\mathbb U$-module $(M,\pi)$ is even if $\pi_k(\bar{n},l)=0$ whenever $k$ is odd. The full subcategory of even $\mathbb U$-modules will be denoted by 
$EM^e_{(\mathbb U,\Theta)}$.
\end{defn}

\begin{defn}\label{D2.5}
Let $(\mathbb V,\Delta)$ be an $A_\infty$-comonad over $\mathcal C$. A $(\mathbb V,\Delta)$-comodule $(P,\rho)$ consists of the following data:

\smallskip
(a) A collection $P=\{P_n\}_{n\in \mathbb Z}$ of objects of $\mathcal C$.

\smallskip
(b) A collection of morphisms
\begin{equation}\label{2.25juk}
\rho=\{\mbox{$\rho_k(\bar{n},l):P_{\Sigma\bar{n}+l+(2-k)}\longrightarrow \mathbb V_{\bar{n}}P_{l} $ $\vert$ $k\geq 1$, $\bar{n}=(n_1,...,n_{k-1})
\in \mathbb Z^{k-1}$, $l\in\mathbb Z$ } \}
\end{equation}
 satisfying, for each $\bar{z}\in \mathbb Z^N$,  $N\geq 0$ and $l\in\mathbb Z$,
\begin{equation}\label{l2.14}
\begin{array}{l}
\underset{\bar{z}=(\bar{n},\bar{n}',\bar{n}'')}{\sum} (-1)^{pq+r+q\Sigma\bar{n}} (\mathbb V_{\bar{n}}\ast \delta_q(\bar{n}'))(\mathbb V_{\bar{n}''}P_{l})\circ \rho_{p+1+r}(\bar{n},\Sigma \bar{n}'+(2-q), \bar{n}'', l)\\
\qquad\qquad =-\underset{\bar{z}=(\bar{m},\bar{m}')}{\sum} (-1)^{ab+b\Sigma\bar{m}}\mathbb V_{\bar{m}} (\rho_b(\bar{m}',l))\circ \rho_{a+1}(\bar{m},\Sigma\bar{m}'+l+(2-b)) \\
\end{array}
\end{equation} On the left hand side of \eqref{l2.14}, we have $|\bar{n}|=p$, $|\bar{n}'|=q$ and $|\bar{n}''|=r-1$, while on the right hand side, we have
$|\bar{m}|=a$ and $|\bar{m}'|=b-1$. Accordingly, both sides of \eqref{l2.14} represent morphisms 
\begin{equation}
P_{\Sigma \bar{n}+\Sigma \bar{n}'+\Sigma \bar{n}''+l+(2-N)}= P_{\Sigma \bar{m}+\Sigma \bar{m}'+l+(2-N)}\longrightarrow \mathbb V_{\bar{m}}\mathbb V_{\bar{m}'}P_{l}=\mathbb V_{\bar{n}}\mathbb V_{\bar{n}'}\mathbb V_{\bar{n}''}P_{l}
\end{equation}  A morphism $g:(P,\rho)\longrightarrow (P',\rho')$ of $\mathbb V$-comodules consists of a collection $g=\{g_l:P_l\longrightarrow P'_l\}_{l\in \mathbb Z}$ such that 
$\rho_k'(\bar{n},l)\circ g_{\Sigma\bar{n}+l+(2-k)}=\mathbb V_{\bar{n}}(g_l)\circ \rho_k(\bar{n},l)$ for every $k\geq 1$, $\bar{n}=(n_1,...,n_{k-1})
\in \mathbb Z^{k-1}$, $l\in \mathbb Z$ . We denote the category of  $(\mathbb V,\Delta)$-comodules by $EM^{(\mathbb V,\Delta)}$. 

\smallskip
We will say that a $\mathbb V$-comodule $(P,\rho)$ is even if $\rho_k(\bar{n},l)=0$ whenever $k$ is odd. The full subcategory of even $\mathbb V$-comodules will be denoted by
$EM_e^{(\mathbb V,\Delta)}$. 
\end{defn}

\begin{thm}\label{P2.6} (a) Let $(\mathbb U,\Theta)$ be an $A_\infty$-monad over $\mathcal C$. Then, for any object $M\in \mathcal C$, the collection $\{\mathbb U_nM\}_{n\in
\mathbb Z}$ can be equipped with a canonical $(\mathbb U,\Theta)$-module structure.

\smallskip
(b) Let $(\mathbb V,\Delta)$ be an $A_\infty$-comonad over $\mathcal C$. Then, for any object $M\in \mathcal C$, the collection $\{\mathbb V_nM\}_{n\in
\mathbb Z}$ can be equipped with a canonical $(\mathbb V,\Delta)$-comodule structure.

\end{thm}

\begin{proof} We prove only (a) because (b) is similar. For the sake of convenience, we set $Q_n:=\mathbb U_nM$ for each $n\in\mathbb Z$.  For $\bar{n}=(n_1,...,n_{k-1})
\in \mathbb Z^{k-1}$, $k\geq 1$, $l\in\mathbb Z$, we set
\begin{equation}\label{216cf}
\begin{CD}
\pi_k(\bar{n},l):\mathbb U_{\bar{n}}Q_{l}=\mathbb U_{\bar{n}}\mathbb U_lM@>\theta_k(\bar{n},l)(M)>>\mathbb U_{\Sigma\bar{n}+l+(2-k)}M=Q_{\Sigma\bar{n}+l+(2-k)}
\end{CD}
\end{equation} We note that the condition in \eqref{2.1e} for $(\mathbb U,\Theta)$ to be an $A_\infty$-monad implies in particular that for each $\bar{z}\in \mathbb Z^N$,  $N\geq 0$ and $l\in\mathbb Z$, we must have
\begin{equation}\label{l2.11dt}
\begin{array}{l}
\underset{\bar{z}=(\bar{n},\bar{n}',\bar{n}'')}{\sum} (-1)^{p+qr+q\Sigma\bar{n}}\theta_{p+1+r}(\bar{n},\Sigma \bar{n}'+(2-q), \bar{n}'', l)(M)\circ (\mathbb U_{\bar{n}}\ast \theta_q(\bar{n}'))(\mathbb U_{\bar{n}''}\mathbb U_{l}M)\\
\qquad\qquad =-\underset{\bar{z}=(\bar{m},\bar{m}')}{\sum} (-1)^{a+b\Sigma\bar{m}}\theta_{a+1}(\bar{m},\Sigma\bar{m}'+l+(2-b))(M)\circ \mathbb U_{\bar{m}} (\theta_b(\bar{m}',l)(M))\\
\end{array}
\end{equation} On the left hand side of \eqref{l2.11dt}, we have $|\bar{n}|=p$, $|\bar{n}'|=q$ and $|\bar{n}''|=r-1$, while on the right hand side, we have
$|\bar{m}|=a$ and $|\bar{m}'|=b-1$. Putting $\pi_k(\bar{n},l)=\theta_k(\bar{n},l)(M)$ and $Q_n=\mathbb U_nM$, it follows from \eqref{l2.11dt} that
\begin{equation}\label{l2.18dt}
\begin{array}{l}
\underset{\bar{z}=(\bar{n},\bar{n}',\bar{n}'')}{\sum} (-1)^{p+qr+q\Sigma\bar{n}}\pi_{p+1+r}(\bar{n},\Sigma \bar{n}'+(2-q), \bar{n}'', l)\circ (\mathbb U_{\bar{n}}\ast \theta_q(\bar{n}'))(\mathbb U_{\bar{n}''}Q_l)\\
\qquad\qquad =-\underset{\bar{z}=(\bar{m},\bar{m}')}{\sum} (-1)^{a+b\Sigma\bar{m}}\pi_{a+1}(\bar{m},\Sigma\bar{m}'+l+(2-b))\circ \mathbb U_{\bar{m}} (\pi_b(\bar{m}',l))\\
\end{array}
\end{equation} This proves the result.

\end{proof}

Let $(\mathbb U,\Theta)$ be an $A_\infty$-monad over $\mathcal C$.  We now let $\mathcal C^{\mathbb Z}$ denote the category of $\mathbb Z$-graded objects over $\mathcal C$. Then, for any $M=\{M_l\}_{l\in \mathbb Z}$ in $\mathcal C^{\mathbb Z}$, we set
\begin{equation}
\hat{\mathbb U}M=\left\{(\hat{\mathbb U}M)_l:=\underset{x+y=l}{\bigoplus}\textrm{ }\mathbb U_xM_y\right\}_{l\in \mathbb Z}\qquad 
\end{equation} In particular,  for $M\in \mathcal C$, we denote by $\hat{\mathbb U}M$ the object $\{\mathbb U_nM\}_{n\in \mathbb Z} \in EM_{(\mathbb U,\Theta)}$ determined by Proposition \ref{P2.6}. 

\begin{thm}\label{L3.4nu} Let $(\mathbb U,\Theta)$ be an $A_\infty$-monad over $\mathcal C$ such that each of the functors $\{\mathbb U_n\}_{n\in \mathbb Z}$ preserves direct sums. Then, $\hat{\mathbb U}$ determines a functor
\begin{equation}
\hat{\mathbb U}:\mathcal C^{\mathbb Z}\longrightarrow EM_{(\mathbb U,\Theta)}
\end{equation}
\end{thm}
\begin{proof}
For the sake of convenience, we set $Q_l:=(\hat{\mathbb U}M)_l$ for each $l\in \mathbb Z$. For  $\bar{n}=(n_1,...,n_{k-1})
\in \mathbb Z^{k-1}$, $k\geq 1$, $l\in\mathbb Z$, we define $\pi_k(\bar{n},l):\mathbb U_{\bar{n}}Q_l\longrightarrow Q_{\Sigma \bar{n}+l+(2-k)}$ by setting 
\begin{equation}\label{312cfr}
\begin{CD}
\mathbb U_{\bar{n}}Q_{l}=\mathbb U_{\bar{n}}\left(\underset{x+y=l}{\bigoplus}\mathbb U_xM_y\right)=\left(\underset{x+y=l}{\bigoplus}\mathbb U_{\bar{n}}\mathbb U_xM_y\right)@>\underset{x+y=l}{\bigoplus}\theta_k(\bar{n},x)(M_y)>=\pi_k(\bar{n},l)>\left(\underset{x+y=l}{\bigoplus}\mathbb U_{\Sigma\bar{n}+x+(2-k)}M_y\right)=Q_{\Sigma\bar{n}+l+(2-k)}
\end{CD}
\end{equation} In \eqref{312cfr} we have used the fact that each $\mathbb U_n$, and hence any $\mathbb U_{\bar{n}}$, preserves direct sums. As in \eqref{l2.11dt}, we note that the condition in \eqref{2.1e} for $(\mathbb U,\Theta)$ to be an $A_\infty$-monad implies in particular that for each $\bar{z}\in \mathbb Z^N$,  $N\geq 0$ and $x,y\in\mathbb Z$, we must have
\begin{equation}\label{l2.11dtc}
\begin{array}{l}
\underset{\bar{z}=(\bar{n},\bar{n}',\bar{n}'')}{\sum} (-1)^{p+qr+q\Sigma\bar{n}}\theta_{p+1+r}(\bar{n},\Sigma \bar{n}'+(2-q), \bar{n}'', x)(M_y)\circ (\mathbb U_{\bar{n}}\ast \theta_q(\bar{n}'))(\mathbb U_{\bar{n}''}\mathbb U_{x}M_y)\\
\qquad\qquad =-\underset{\bar{z}=(\bar{m},\bar{m}')}{\sum} (-1)^{a+b\Sigma\bar{m}}\theta_{a+1}(\bar{m},\Sigma\bar{m}'+x+(2-b))(M_y)\circ \mathbb U_{\bar{m}} (\theta_b(\bar{m}',x)(M_y))\\
\end{array}
\end{equation} Applying the definition of $\pi_k(\bar{n},l)$ in \eqref{312cfr} and taking the direct sum of the equalities in \eqref{l2.11dtc} over all $x$, $y\in \mathbb Z$ such that $l=x
+y$, we obtain
\begin{equation}\label{l2.11brt}
\begin{array}{l}
\underset{\bar{z}=(\bar{n},\bar{n}',\bar{n}'')}{\sum} (-1)^{p+qr+q\Sigma\bar{n}}\pi_{p+1+r}(\bar{n},\Sigma \bar{n}'+(2-q), \bar{n}'', l)\circ (\mathbb U_{\bar{n}}\ast \theta_q(\bar{n}'))(\mathbb U_{\bar{n}''}Q_l)\\
\qquad\qquad =-\underset{\bar{z}=(\bar{m},\bar{m}')}{\sum} (-1)^{a+b\Sigma\bar{m}}\pi_{a+1}(\bar{m},\Sigma\bar{m}'+l+(2-b))\circ \mathbb U_{\bar{m}} (\pi_b(\bar{m}',l))\\
\end{array}
\end{equation} It is now clear that we have a functor $\hat{\mathbb U}:\mathcal C^{\mathbb Z}\longrightarrow EM_{(\mathbb U,\Theta)}$. 
\end{proof}

We mention here the following two simple facts on the formalism of adjoint functors, which we will use in the proof of the next result.

\smallskip
(1) Let $(U_1,V_1)$ and $(U_2,V_2)$ be pairs of adjoint endofunctors on a category $\mathcal A$. We consider a transformation $\theta:U_1\longrightarrow U_2$ and the corresponding transformation $\theta^R:V_2\longrightarrow V_1$ between right adjoints. Let $X$, $Y\in \mathcal A$ and let $f:U_2X\longrightarrow Y$ be a morphism.  Then, we have commutative diagrams
\begin{equation}\label{rule1}
\begin{tikzcd}[row sep=2.5em, column sep = 4em]
 & U_2X \arrow[dr,"f"] \\
U_1X \arrow[ur,"\theta X"] \arrow[rr,"(f\circ \theta X)"] && Y
\end{tikzcd} \qquad \Rightarrow \qquad \begin{tikzcd}[row sep=2.5em, column sep = 4em]
 & V_2Y\arrow[dr,"\theta^RY"] \\
X \arrow[ur,"f^R"] \arrow[rr,"(f\circ \theta X)^R"] && V_1Y
\end{tikzcd}
\end{equation}

\smallskip
(2) Let $(U_1,V_1)$ and $(U_2,V_2)$ be pairs of adjoint endofunctors on a category $\mathcal A$. We consider objects $X$, $Y$, $Z\in \mathcal A$ as well as morphisms $f:U_1X\longrightarrow Y$ and $g:U_2Y\longrightarrow Z$. Then, we have commutative diagrams
\begin{equation}\label{rule2}
\begin{tikzcd}[row sep=2.5em, column sep = 4em]
 & U_2Y \arrow[dr,"g"] \\
U_2U_1X \arrow[ur,"U_2f"] \arrow[rr,"(g\circ U_2f)"] && Z
\end{tikzcd} \qquad \Rightarrow \qquad \begin{tikzcd}[row sep=2.5em, column sep = 4em]
 & V_1Y\arrow[dr,"V_1g^R"] \\
X \arrow[ur,"f^R"] \arrow[rr,"(g\circ U_2f)^R"] && V_1V_2Z
\end{tikzcd}
\end{equation}

\begin{thm}\label{P2.7} Let $(\mathbb U=\{\mathbb U_n:\mathcal C\longrightarrow \mathcal C\}_{n\geq 0},\mathbb V=
\{\mathbb V_n:\mathcal C\longrightarrow\mathcal C\}_{n\geq 0})$ be a $\mathbb Z$-system of adjoints on $\mathcal C$. Let
$(\mathbb U,\Theta)$ be an even  $A_\infty$-monad and $(\mathbb V,\Delta)$ be the corresponding $A_\infty$-comonad. Then, the categories $EM_{(\mathbb U,\Theta)}^e$ and $EM^{(\mathbb V,\Delta)}_e $ are isomorphic. 

\end{thm}

\begin{proof}
Let $(M,\pi)$ be an even $(\mathbb U,\Theta)$-module as in Definition \ref{D2.4}. Accordingly, for $\bar{n}\in \mathbb Z^{k-1}$, $k\geq 1$, $l\in\mathbb Z$, we have morphisms 
\begin{equation}
(\pi_k(\bar{n},l):\mathbb U_{\bar{n}}M_l\longrightarrow M_{\Sigma\bar{n}+l+(2-k)})\Rightarrow (\pi_k^R(\bar{n},l):M_l\longrightarrow \mathbb V_{\bar{n}^{op}}M_{\Sigma\bar{n}+l+(2-k)})
\end{equation} We now define $P=\{P_l\}_{l\in \mathbb Z}$ by setting $P_l=M_{-l}$ for each $l\in \mathbb Z$ as well as 
\begin{equation}\label{2.22edx}
\begin{CD}
\rho_k(\bar{n},l):P_{\Sigma \bar{n}+l+(2-k)}=M_{-\Sigma \bar{n}-l-(2-k)}@>\pi_k^R(\bar{n}^{op},-\Sigma \bar{n}-l-(2-k))>> \mathbb V_{\bar{n}}M_{-l} =\mathbb V_{\bar{n}}P_l \\
\end{CD}
\end{equation} for $l\in \mathbb Z$ and $\bar{n}\in \mathbb Z^{k-1}$.  From the proof of Theorem \ref{T2.3}, we know that the collection $\Delta=\{\delta_k(\bar{n})\}$ is given by setting $\delta_k(\bar{n})=\theta_k^R(\bar{n}^{op})$.
In the notation of \eqref{l2.11}, we now see that for each $\bar{z}\in \mathbb Z^N$, $N\geq 0$, $l\in\mathbb Z$, we have \begin{equation}\label{e2.22}
\begin{array}{l}
\underset{\bar{z}=(\bar{n},\bar{n}',\bar{n}'')}{\sum} (-1)^{p+qr+q\Sigma\bar{n}}\pi_{p+1+r}(\bar{n},\Sigma \bar{n}'+(2-q), \bar{n}'', l)\circ (\mathbb U_{\bar{n}}\ast \theta_q(\bar{n}'))(\mathbb U_{\bar{n}''}M_{l})\\
\qquad\qquad =-\underset{\bar{z}=(\bar{m},\bar{m}')}{\sum} (-1)^{a+b\Sigma\bar{m}}\pi_{a+1}(\bar{m},\Sigma\bar{m}'+l+(2-b))\circ \mathbb U_{\bar{m}} (\pi_b(\bar{m}',l))\\
\end{array}
\end{equation} Taking adjoints on both sides of \eqref{e2.22}, we obtain
\begin{equation}\label{e2.23}
\begin{array}{l}
\underset{\bar{z}=(\bar{n},\bar{n}',\bar{n}'')}{\sum} (-1)^{p+qr+q\Sigma\bar{n}}(\pi_{p+1+r}(\bar{n},\Sigma \bar{n}'+(2-q), \bar{n}'', l)\circ (\mathbb U_{\bar{n}}\ast \theta_q(\bar{n}'))(\mathbb U_{\bar{n}''}M_{l}))^R\\
\qquad\qquad =-\underset{\bar{z}=(\bar{m},\bar{m}')}{\sum} (-1)^{a+b\Sigma\bar{m}}(\pi_{a+1}(\bar{m},\Sigma\bar{m}'+l+(2-b))\circ \mathbb U_{\bar{m}} (\pi_b(\bar{m}',l)))^R\\
\end{array}
\end{equation} We now consider one by one the terms in \eqref{e2.23}. For any term on the left hand side of \eqref{e2.23}, we apply \eqref{rule1} with
\begin{equation} \begin{array}{c} X=M_l\qquad Y=M_{\Sigma \bar{n}+\Sigma\bar{n}'+\Sigma \bar{n}''+l+(2-N)} \\ \theta=\mathbb U_{\bar{n}}\ast \theta_q(\bar{n}') \ast \mathbb U_{\bar{n}''}: U_1=\mathbb U_{\bar{n}}\mathbb U_{\bar{n}'}\mathbb U_{\bar{n}''}\longrightarrow \mathbb U_{\bar{n}}\mathbb U_{\Sigma\bar{n}'+(2-q)}\mathbb U_{\bar{n}''}=U_2\\
f=\pi_{p+1+r}(\bar{n},\Sigma \bar{n}'+(2-q), \bar{n}'', l):U_2X= \mathbb U_{\bar{n}}\mathbb U_{\Sigma\bar{n}'+(2-q)}\mathbb U_{\bar{n}''}M_l\longrightarrow M_{\Sigma \bar{n}+\Sigma\bar{n}'+\Sigma \bar{n}''+l+(2-N)} =Y \\
\end{array}\end{equation} It follows therefore from \eqref{rule1} that the left hand side of \eqref{e2.23} gives
\begin{equation}\label{e2.25}
\begin{array}{l}
\underset{\bar{z}=(\bar{n},\bar{n}',\bar{n}'')}{\sum} (-1)^{p+qr+q\Sigma\bar{n}}(\pi_{p+1+r}(\bar{n},\Sigma \bar{n}'+(2-q), \bar{n}'', l)\circ (\mathbb U_{\bar{n}}\ast \theta_q(\bar{n}'))(\mathbb U_{\bar{n}''}M_{l}))^R\\
= \underset{\bar{z}=(\bar{n},\bar{n}',\bar{n}'')}{\sum} (-1)^{p+qr+q\Sigma\bar{n}}(\mathbb V_{\bar{n}''^{op}}\ast \theta_q^R(\bar{n}'))(  \mathbb V_{\bar{n}^{op}}(M_{
\bar{z}+l+(2-N)}))\circ\pi_{p+1+r}^R(\bar{n},\Sigma \bar{n}'+(2-q), \bar{n}'', l)\\
\end{array}
\end{equation}
If we shift degrees by setting $l'=-\Sigma\bar{z}-l-(2-N)$ and use \eqref{2.22edx}, then \eqref{e2.25} becomes
\begin{equation}\label{e2.26xw}
\begin{array}{l}
 - \underset{\bar{z}=(\bar{n},\bar{n}',\bar{n}'')}{\sum} (-1)^{p+1} (\mathbb V_{\bar{n}''^{op}}\ast \delta_q(\bar{n}'^{op})) ( \mathbb V_{\bar{n}^{op}}(P_{l'}))\circ \rho_{p+1+r}(\bar{n}''^{op},\Sigma \bar{n}'+(2-q), \bar{n}^{op},l')\\
\end{array} \end{equation} where the sign is obtained from the fact that $q$ and $p+1+r$ are both even. 
  For any term on the right hand side of \eqref{e2.23}, we now apply \eqref{rule2} with
\begin{equation}
\begin{array}{c}
X=M_l\qquad Y=M_{\Sigma\bar{m}'+l+(2-b)}\qquad Z=M_{\Sigma\bar{m}+\Sigma\bar{m}'+l+(2-N)} \qquad U_1=\mathbb U_{\bar{m}'}\qquad U_2=\mathbb U_{\bar{m}}\\
f=\pi_b(\bar{m}',l) :U_1X=\mathbb U_{\bar{m}'}M_l\longrightarrow M_{\Sigma\bar{m}'+l+(2-b)}=Y\\
g=\pi_{a+1}(\bar{m},\Sigma\bar{m}'+l+(2-b)):U_2Y=\mathbb U_{\bar{m}}M_{\Sigma\bar{m}'+l+(2-b)}\longrightarrow M_{\Sigma\bar{m}+\Sigma\bar{m}'+l+(2-N)}=Z\\
\end{array}
\end{equation} It follows therefore from \eqref{rule2} that the right hand side of \eqref{e2.23} gives
\begin{equation}\label{e2.27}
\begin{array}{l}
-\underset{\bar{z}=(\bar{m},\bar{m}')}{\sum} (-1)^{a+b\Sigma\bar{m}}(\pi_{a+1}(\bar{m},\Sigma\bar{m}'+l+(2-b))\circ \mathbb U_{\bar{m}} (\pi_b(\bar{m}',l)))^R\\
=-\underset{\bar{z}=(\bar{m},\bar{m}')}{\sum} (-1)^{a+b\Sigma\bar{m}}\mathbb V_{\bar{m}'^{op}}\pi_{a+1}^R(\bar{m},\Sigma\bar{m}'+l+(2-b))\circ \pi_b^R(\bar{m}',l)\\
=-\underset{\bar{z}=(\bar{m},\bar{m}')}{\sum} (-1)^{a+b\Sigma\bar{m}}\mathbb V_{\bar{m}'^{op}}\rho_{a+1}(\bar{m}^{op},-\Sigma\bar{z}-l-(2-N))\circ \rho_b(\bar{m}'^{op},-\Sigma\bar{m}'-l-(2-b))\\
\end{array}
\end{equation} 
Again, if we shift degrees by setting $l'=-\Sigma\bar{z}-l-(2-N)$, then \eqref{e2.27} becomes
\begin{equation}\label{e2.27cqo}
\begin{array}{l}
-\underset{\bar{z}=(\bar{m},\bar{m}')}{\sum} (-1)^{a+b\Sigma\bar{m}}\mathbb V_{\bar{m}'^{op}}\rho_{a+1}(\bar{m}^{op},l')\circ \rho_b(\bar{m}'^{op},\Sigma\bar{m}+l'+(1-a))\\ 
=\underset{\bar{z}=(\bar{m},\bar{m}')}{\sum} \mathbb V_{\bar{m}'^{op}}\rho_{a+1}(\bar{m}^{op},l')\circ \rho_b(\bar{m}'^{op},\Sigma\bar{m}+l'+(1-a))\\ 
\end{array} 
\end{equation} where the last equality in \eqref{e2.27} follows from the fact that $a+1$ and $b$ must both be even. By inspecting \eqref{e2.26xw} and
\eqref{e2.27cqo}, we see that the morphisms $\rho_k(\bar{n},l)$ together make $P=\{P_n=M_{-n}\}_{n\in \mathbb Z}$ into a 
$(\mathbb V,\Delta)$-comodule. These arguments can be reversed, showing that even $(\mathbb U,\Theta)$-modules correspond to even $(\mathbb V,\Delta)$-comodules and vice-versa.

\end{proof}

\begin{defn}\label{D2.8}
Let $(\mathbb U,\Theta)$ be an $A_\infty$-monad over $\mathcal C$ and let $(M^1,\pi^1)$, $(M^2,\pi^2)$ be $(\mathbb U,\Theta)$-modules. An $\infty$-morphism $\alpha :(M^1,\pi^1)\longrightarrow (M^2,\pi^2)$ 
consists of a collection
\begin{equation}\label{e2.28}
\alpha=\{\mbox{ $\alpha_k(\bar{n},l):\mathbb U_{\bar{n}}M^1_l\longrightarrow M^2_{\Sigma\bar{n}+l+(1-k)}$ $\vert$ $k\geq 1$, $\bar{n}=(n_1,...,n_{k-1})\in 
\mathbb Z^{k-1}$,  $l\in\mathbb Z$ }\}
\end{equation} satisfying for each $\bar{z}\in \mathbb Z^N$, $N\geq 0$, $l\in \mathbb Z$,
\begin{equation}\label{e2.29}
\begin{array}{c}
\underset{\bar{z}=(\bar{m},\bar{m}')}{\sum}(-1)^{a+b\Sigma\bar{m}}\alpha_{a+1}(\bar{m},\Sigma\bar{m}'+l+(2-b))\circ \mathbb U_{\bar{m}}(\pi_b^1(\bar{m}',l))\\
\qquad\qquad \qquad\qquad+ \underset{\bar{z}=(\bar{n},\bar{n}',\bar{n}'')}{\sum}(-1)^{a+bc+b\Sigma\bar{n}}\alpha_{a+1+c}(\bar{n},\Sigma\bar{n}'+(2-b),\bar{n}'',l)\circ (\mathbb U_{\bar{n}}\ast\theta_b(\bar{n}')\ast \mathbb U_{\bar{n}''})(M^1_l)  \\
\qquad\qquad \qquad\qquad \qquad \qquad = \underset{\bar{z}=(\bar{p},\bar{p}')}{\sum}\pi^2_{a+1}(\bar{p},\Sigma\bar{p}'+l+(1-b))\circ \mathbb U_{\bar{p}}(\alpha_b(\bar{p}',l))\\
\end{array}
\end{equation} where:

\smallskip
(1) In the first term on left hand side of \eqref{e2.29}, we have $|\bar{m}|=a$ and $|\bar{m}'|=b-1$.

\smallskip
(2) In the second term on left hand side of \eqref{e2.29}, we have $|\bar{n}|=a$, $|\bar{n}'|=b$ and $|\bar{n}''|=c-1$.

\smallskip
(3) On the right hand side of \eqref{e2.29}, we have $|\bar{p}|=a$ and $|\bar{p}'|=b-1$.
\end{defn}

We observe that both sides of \eqref{e2.29} represent morphisms 
$ \mathbb U_{\bar{z}}M_l^1\longrightarrow M^2_{\Sigma \bar{z}+l+(1-N)}
$.  We will   denote by $\widetilde{EM}_{(\mathbb U,\Theta)}$ the category of $(\mathbb U,\Theta)$-modules equipped with $\infty$-morphisms. 
For any $(M,\pi)$ in  $\widetilde{EM}_{(\mathbb U,\Theta)}$, the identity $\iota:(M,\pi)\longrightarrow (M,\pi)$ is given by taking $\iota_1(l)=id:M_l\longrightarrow M_l$ for each $l\in 
\mathbb Z$ and $\iota_k(\bar{z},l)=0$ otherwise.  Given morphisms $\alpha:(M^1,\pi^1)\longrightarrow (M^2,\pi^2)$ and $\beta:(M^2,\pi^2)\longrightarrow (M^3,\pi^3)$, we define their composition $\gamma=\beta\alpha$ as follows
\begin{equation} \label{e2.30}
\begin{array}{c}
\gamma=\{\mbox{ $\gamma_k(\bar{z},l):\mathbb U_{\bar{z}}M^1_l\longrightarrow M^3_{\Sigma\bar{z}+l+(1-k)}$ $\vert$ $k\geq 1$, $\bar{z}=(z_1,...,z_{k-1})\in 
\mathbb Z^{k-1}$,  $l\in\mathbb Z$ }\} \\ \\
\gamma_k(\bar{z},l):=\underset{\bar{z}=(\bar{n},\bar{n}')}{\sum}\beta_{a+1}(\bar{n},\Sigma\bar{n}'+l+(1-b))\circ \mathbb U_{\bar{n}}(\alpha_b(\bar{n}',l))\\
\end{array}
\end{equation} where $|\bar{n}|=a$ and $|\bar{n}'|=b-1$ in \eqref{e2.30}. We will say that a morphism $\alpha :(M^1,\pi^1)\longrightarrow (M^2,\pi^2)$  is odd if 
$\alpha_k(\bar{n},l)=0$ for any even $k$. In particular, any identity morphism is odd. From \eqref{e2.30}, we also  notice that the composition of odd morphisms in $\widetilde{EM}_{(\mathbb U,\Theta)}$ must be odd. 

\begin{defn}\label{D2.9}
Let $(\mathbb V,\Delta)$ be an $A_\infty$-comonad over $\mathcal C$ and let $(P^1,\rho^1)$, $(P^2,\rho^2)$ be $(\mathbb V,\Delta)$-comodules. An $\infty$-morphism $\alpha :(P^1,\rho^1)\longrightarrow (P^2,\rho^2)$ 
consists of a collection
\begin{equation}\label{e2.31}
\alpha=\{\mbox{ $\alpha_k(\bar{n},l):P^1_{\Sigma\bar{n}+l+(1-k)}\longrightarrow \mathbb V_{\bar{n}}P^2_l$ $\vert$ $k\geq 1$, $\bar{n}=(n_1,...,n_{k-1})\in 
\mathbb Z^{k-1}$,  $l\in\mathbb Z$ }\}
\end{equation} satisfying for each $\bar{z}\in \mathbb Z^N$, $N\geq 0$, $l\in \mathbb Z$,
\begin{equation}\label{e2.32}
\begin{array}{c}
\underset{\bar{z}=(\bar{m},\bar{m}')}{\sum}(-1)^{ab+b\Sigma\bar{m}}\mathbb V_{\bar{m}}(\rho_b^2(\bar{m}',l))\circ \alpha_{a+1}(\bar{m},\Sigma\bar{m}'+l+(2-b))\\
\qquad\qquad \qquad\qquad+ \underset{\bar{z}=(\bar{n},\bar{n}',\bar{n}'')}{\sum}(-1)^{c+ab+b\Sigma\bar{n}}(\mathbb V_{\bar{n}}\ast\delta_b(\bar{n}')\ast \mathbb V_{\bar{n}''})(P^2_l)  \circ \alpha_{a+1+c}(\bar{n},\Sigma\bar{n}'+(2-b),\bar{n}'',l)\\
\qquad\qquad \qquad\qquad \qquad \qquad = \underset{\bar{z}=(\bar{p},\bar{p}')}{\sum}\mathbb V_{\bar{p}}(\alpha_b(\bar{p}',l)) \circ \rho^1_{a+1}(\bar{p},\Sigma\bar{p}'+l+(1-b))\\
\end{array}
\end{equation} where:

\smallskip
(1) In the first term on left hand side of \eqref{e2.32}, we have $|\bar{m}|=a$ and $|\bar{m}'|=b-1$.

\smallskip
(2) In the second term on left hand side of \eqref{e2.32}, we have $|\bar{n}|=a$, $|\bar{n}'|=b$ and $|\bar{n}''|=c-1$.

\smallskip
(3) On the right hand side of \eqref{e2.32}, we have $|\bar{p}|=a$ and $|\bar{p}'|=b-1$.

\end{defn}

We observe that both sides of \eqref{e2.32} represent morphisms 
$ P^1_{\Sigma \bar{z}+l+(1-N)}\longrightarrow \mathbb V_{\bar{z}}P_l^2
$. We will  denote by $\widetilde{EM}^{(\mathbb V,\Delta)}$ the category of $(\mathbb V,\Delta)$-comodules equipped with $\infty$-morphisms.
The composition of morphisms in $\widetilde{EM}^{(\mathbb V,\Delta)}$ is defined in a manner dual to \eqref{e2.30}. We will say that a morphism $\alpha$ in $EM^{(\mathbb V,\Delta)}$ is odd if $\alpha_k(\bar{n},l)=0$ whenever $k$ is even.  We may also verify that the collection of odd morphisms in $\widetilde{EM}^{(\mathbb V,\Delta)}$ is closed under composition.

\smallskip
For an $A_\infty$-monad $(\mathbb U,\Theta)$ we denote by $\widetilde{EM}_{(\mathbb U,\Theta)}^{eo}$ the subcategory of $\widetilde{EM}_{(\mathbb U,\Theta)}$ whose objects are even $(\mathbb U,\Theta)$-modules with odd $\infty$-morphisms between them. Similarly, for an $A_\infty$-comonad $(\mathbb V,\Delta)$, we denote by $\widetilde{EM}^{(\mathbb V,\Delta)}_{eo}$ the subcategory of
$\widetilde{EM}^{(\mathbb V,\Delta)}$ whose objects are even $(\mathbb V,\Delta)$-comodules with odd $\infty$-morphisms between them. We now have the final result of this section.

\begin{Thm}\label{T2.10}  Let $(\mathbb U=\{\mathbb U_n:\mathcal C\longrightarrow \mathcal C\}_{n\geq 0},\mathbb V=
\{\mathbb V_n:\mathcal C\longrightarrow\mathcal C\}_{n\geq 0})$ be a $\mathbb Z$-system of adjoints on $\mathcal C$. Let
$(\mathbb U,\Theta)$ be an even $A_\infty$-monad and $(\mathbb V,\Delta)$ be the corresponding $A_\infty$-comonad. Then, the categories $\widetilde{EM}_{(\mathbb U,\Theta)}^{eo}$  and
$\widetilde{EM}^{(\mathbb V,\Delta)}_{eo}$ are isomorphic.

\end{Thm}

\begin{proof} From Proposition \ref{P2.7}, we already know that there is a one-one correspondence between the objects of  $EM_{(\mathbb U,\Theta)}^{eo}$  and
$EM^{(\mathbb V,\Delta)}_{eo}$. We consider therefore a morphism $\alpha:(M^1,\pi^1)\longrightarrow (M^2,\pi^2)$ in  $EM_{(\mathbb U,\Theta)}^{eo}$. Accordingly for any
$k\geq 1$, $\bar{n}\in \mathbb Z^{k-1}$, $l\in \mathbb Z$ we have morphisms
\begin{equation}
(\alpha_k(\bar{n},l):\mathbb U_{\bar{n}}M^1_l\longrightarrow M^2_{\Sigma\bar{n}+l+(1-k)})\Rightarrow (\alpha_k^R(\bar{n},l):M^1_l\longrightarrow \mathbb V_{\bar{n}^{op}}M^2_{\Sigma\bar{n}+l+(1-k)})
\end{equation} From the proof of Proposition \ref{P2.7}, we know that the object $(P^1,\rho^1)$ (resp. $(P^2,\rho^2)$) in $EM^{(\mathbb V,\Delta)}_{eo}$ corresponding
to $(M^1,\pi^1)$ (resp. $(M^2,\pi^2)$) is given by setting
\begin{equation}\label{e2.37}
\begin{CD}
P^1_{l}:=M^1_{-l}\qquad \rho^1_k(\bar{n},l):P^1_{\Sigma \bar{n}+l+(2-k)}=M^1_{-\Sigma \bar{n}-l-(2-k)}@>\pi_k^{1R}(\bar{n}^{op},-\Sigma \bar{n}-l-(2-k))>> \mathbb V_{\bar{n}}M^1_{-l} =\mathbb V_{\bar{n}}P^1_l \\
P^2_{l}:=M^2_{-l}\qquad \rho^2_k(\bar{n},l):P^2_{\Sigma \bar{n}+l+(2-k)}=M^2_{-\Sigma \bar{n}-l-(2-k)}@>\pi_k^{2R}(\bar{n}^{op},-\Sigma \bar{n}-l-(2-k))>> \mathbb V_{\bar{n}}M^2_{-l} =\mathbb V_{\bar{n}}P^2_l \\
\end{CD}
\end{equation} For $k\geq 1$, $\bar{n}\in \mathbb Z^{k-1}$, $l\in \mathbb Z$, we now define
\begin{equation}\label{e2.38}
\begin{CD}
\beta_k(\bar{n},l):P^1_{\Sigma\bar{n}+l+(1-k)}=M^1_{-\Sigma\bar{n}-l-(1-k)}@>\alpha_k^R(\bar{n}^{op},-\Sigma\bar{n}-l-(1-k))>> \mathbb V_{\bar{n}}M^2_{-l}=\mathbb V_{\bar{n}}P^2_l
\end{CD}
\end{equation} We claim that \eqref{e2.38} determines a morphism $(P^1,\rho^1)\longrightarrow (P^2,\rho^2)$ in $\widetilde{EM}^{(\mathbb V,\Delta)}_{eo}$. From the proof of Theorem \ref{T2.3}, we know that the collection $\Delta=\{\delta_k(\bar{n})\}$ is given by setting $\delta_k(\bar{n})=\theta_k^R(\bar{n}^{op})$. 
In the notation of \eqref{e2.29}, we now see that for each $\bar{z}\in \mathbb Z^N$, $N\geq 0$, $l\in\mathbb Z$, we have 
\begin{equation}\label{e2.39}
\begin{array}{c}
\underset{\bar{z}=(\bar{m},\bar{m}')}{\sum}(-1)^{a+b\Sigma\bar{m}}\alpha_{a+1}(\bar{m},\Sigma\bar{m}'+l+(2-b))\circ \mathbb U_{\bar{m}}(\pi_b^1(\bar{m}',l))\\
\qquad\qquad \qquad\qquad+ \underset{\bar{z}=(\bar{n},\bar{n}',\bar{n}'')}{\sum}(-1)^{a+bc+b\Sigma\bar{n}}\alpha_{a+1+c}(\bar{n},\Sigma\bar{n}'+(2-b),\bar{n}'',l)\circ (\mathbb U_{\bar{n}}\ast\theta_b(\bar{n}')\ast \mathbb U_{\bar{n}''})(M^1_l)  \\
\qquad\qquad \qquad\qquad \qquad \qquad = \underset{\bar{z}=(\bar{p},\bar{p}')}{\sum}\pi^2_{a+1}(\bar{p},\Sigma\bar{p}'+l+(1-b))\circ \mathbb U_{\bar{p}}(\alpha_b(\bar{p}',l))\\
\end{array}
\end{equation} We now consider one by one the terms in \eqref{e2.39}. For those in the first term on the left hand side of \eqref{e2.39}, we apply \eqref{rule2} with
\begin{equation}\label{e2.40}
\begin{array}{c}
X=M^1_l \qquad Y=M^1_{\Sigma\bar{m}'+l+(2-b)}\qquad Z=M^2_{\Sigma\bar{z}+l+(1-N)}\qquad U_1=\mathbb U_{\bar{m}'}\qquad U_2=\mathbb U_{\bar{m}}\\
f=\pi^1_b(\bar{m}',l): U_1X=\mathbb U_{\bar{m}'}M^1_l\longrightarrow M^1_{\Sigma\bar{m}'+l+(2-b)}=Y\\
g=\alpha_{a+1}(\bar{m},\Sigma\bar{m}'+l+(2-b)):U_2Y=\mathbb U_{\bar{m}}M^1_{\Sigma\bar{m}'+l+(2-b)}\longrightarrow M^2_{\Sigma\bar{z}+l+(1-N)}=Z\\
\end{array}
\end{equation} It follows from \eqref{rule2} that the first term on the left hand side of \eqref{e2.39} gives
\begin{equation}\label{e2.41}
\begin{array}{l}
\underset{\bar{z}=(\bar{m},\bar{m}')}{\sum}(-1)^{a+b\Sigma\bar{m}}(\alpha_{a+1}(\bar{m},\Sigma\bar{m}'+l+(2-b))\circ \mathbb U_{\bar{m}}(\pi_b^1(\bar{m}',l)))^R\\
=\underset{\bar{z}=(\bar{m},\bar{m}')}{\sum}(-1)^{a+b\Sigma\bar{m}}\mathbb V_{\bar{m}'^{op}}(\alpha^R_{a+1}(\bar{m},\Sigma\bar{m}'+l+(2-b)))\circ \pi^{1R}_b(\bar{m}',l)\\
=\underset{\bar{z}=(\bar{m},\bar{m}')}{\sum}(-1)^{a+b\Sigma\bar{m}}\mathbb V_{\bar{m}'^{op}}(\beta_{a+1}(\bar{m}^{op},-\Sigma\bar{z}-l-(1-N)))\circ \rho^1_b(\bar{m}'^{op},
-\Sigma\bar{m}'-l-(2-b))\\
=\underset{\bar{z}=(\bar{m},\bar{m}')}{\sum} \mathbb V_{\bar{m}'^{op}}(\beta_{a+1}(\bar{m}^{op},l'))\circ \rho^1_b(\bar{m}'^{op},
\Sigma\bar{m}
+l'-a)\\
\end{array}
\end{equation} where we have set $l'=-\Sigma\bar{z}-l-(1-N)$, used the replacements in \eqref{e2.37}, \eqref{e2.38} as well as the fact that $a$, $b$ are both even.  For those in the second term on the left hand side of \eqref{e2.39}, we apply \eqref{rule1} with
\begin{equation}\label{e2.42}
\begin{array}{c}
X=M^1_l\qquad Y=M^2_{\bar{z}+l+(1-N)}\\
\theta = \mathbb U_{\bar{n}}\ast \theta_b(\bar{n}')\ast \mathbb U_{\bar{n}''}: U_1=\mathbb U_{\bar{n}}\mathbb U_{\bar{n}'}\mathbb U_{\bar{n}''} \longrightarrow \mathbb U_{\bar{n}}\mathbb U_{\Sigma\bar{n}'+(2-b)}\mathbb U_{\bar{n}''} =U_2\\
f=\alpha_{a+1+c}(\bar{n},\Sigma\bar{n}'+(2-b),\bar{n}'',l):U_2X= \mathbb U_{\bar{n}}\mathbb U_{\Sigma\bar{n}'+(2-b)}\mathbb U_{\bar{n}''} (M_l^1) \longrightarrow M^2_{\bar{z}+l+(1-N)}=Y\\
\end{array}
\end{equation}  It follows from \eqref{rule1} that the second term on the left hand side of \eqref{e2.39} gives
\begin{equation}\label{e2.43}
\begin{array}{l}
\underset{\bar{z}=(\bar{n},\bar{n}',\bar{n}'')}{\sum}(-1)^{a+bc+b\Sigma\bar{n}}(\alpha_{a+1+c}(\bar{n},\Sigma\bar{n}'+(2-b),\bar{n}'',l)\circ (\mathbb U_{\bar{n}}\ast\theta_b(\bar{n}')\ast \mathbb U_{\bar{n}''})(M^1_l))^R \\
=\underset{\bar{z}=(\bar{n},\bar{n}',\bar{n}'')}{\sum}(-1)^{a+bc+b\Sigma\bar{n}} (\mathbb V_{\bar{n}''^{op}}\ast \theta_b^R(\bar{n}')\ast \mathbb V_{\bar{n}^{op}})(M^2_{\bar{z}+l+(1-N)})\circ \alpha^R_{a+1+c}(\bar{n},\Sigma\bar{n}'+(2-b),\bar{n}'',l)\\
=\underset{\bar{z}=(\bar{n},\bar{n}',\bar{n}'')}{\sum}(-1)^{a} (\mathbb V_{\bar{n}''^{op}}\ast \delta_b(\bar{n}'^{op})\ast \mathbb V_{\bar{n}^{op}})(P^2_{l'})\circ \beta_{a+1+c}(\bar{n}''^{op},\Sigma\bar{n}'+(2-b),\bar{n}^{op},l')\\
\end{array}
\end{equation} where we have set $l'=-\Sigma\bar{z}-l-(1-N)$, used the replacements in \eqref{e2.37}, \eqref{e2.38} as well as the fact that $b$, $a+c$ are both even. For those on the right  hand side of \eqref{e2.39}, we apply \eqref{rule2} with
\begin{equation}
\begin{array}{c}
X=M^1_l \qquad Y= M^2_{\Sigma \bar{p}'+l+(1-b)}\qquad Z=M^2_{\Sigma\bar{z}+l+(1-N)}\qquad  U_1=\mathbb U_{\bar{p}'}\qquad
U_2=\mathbb U_{\bar{p}}\\
f=\alpha_b(\bar{p}',l) :U_1X=\mathbb U_{\bar{p}'}M_l^1\longrightarrow M^2_{\Sigma \bar{p}'+l+(1-b)}=Y \\
g=\pi^2_{a+1}(\bar{p},\Sigma\bar{p}'+l+(1-b)):U_2Y=\mathbb U_{\bar{p}}M^2_{\Sigma \bar{p}'+l+(1-b)}\longrightarrow M^2_{\Sigma\bar{z}+l+(1-N)}=Z\\
\end{array}
\end{equation}  It follows from \eqref{rule2} that the  the right hand side of \eqref{e2.39} gives
\begin{equation}\label{e2.45}
\begin{array}{l}
\underset{\bar{z}=(\bar{p},\bar{p}')}{\sum}(\pi^2_{a+1}(\bar{p},\Sigma\bar{p}'+l+(1-b))\circ \mathbb U_{\bar{p}}(\alpha_b(\bar{p}',l)))^R\\
=\underset{\bar{z}=(\bar{p},\bar{p}')}{\sum}\mathbb V_{\bar{p}'^{op}}(\pi^{2R}_{a+1}(\bar{p},\Sigma\bar{p}'+l+(1-b)))\circ \alpha^R_b(\bar{p}',l)\\
=\underset{\bar{z}=(\bar{p},\bar{p}')}{\sum}\mathbb V_{\bar{p}'^{op}}(\rho^{2}_{a+1}(\bar{p}^{op},l'))\circ \beta_b(\bar{p}'^{op},\Sigma\bar{p} +l'+(1-a))\\
\end{array}
\end{equation} where we have set $l'=-\Sigma\bar{z}-l-(1-N)$ and used the replacements in \eqref{e2.37}, \eqref{e2.38}. By inspecting \eqref{e2.41}, \eqref{e2.43} and \eqref{e2.45} and comparing with \eqref{e2.32}, we conclude that $\{\beta_k(\bar{n},l)\}$ together determine a morphism
$(P^1,\rho^1)\longrightarrow (P^2,\rho^2)$ in $\widetilde{EM}^{(\mathbb V,\Delta)}_{eo}$. Finally, these arguments can be reversed, which shows that there is an isomorphism between categories
$\widetilde{EM}_{(\mathbb U,\Theta)}^{eo}$  and
$\widetilde{EM}^{(\mathbb V,\Delta)}_{eo}$.
\end{proof}

\section{Locally finite comodules over $A_\infty$-comonads}

We suppose throughout that $(\mathbb V,\Delta)$ is an $A_\infty$-comonad. We now define what it means for a $(\mathbb V,\Delta)$-comodule to be locally finite.

\begin{defn}\label{D231} Let $(P,\rho)$ be a $(\mathbb V,\Delta)$-comodule. In the notation of Definition \ref{D2.5}, we will say that $(P,\rho)$ is locally finite if for any $T\in \mathbb Z$ and $k\geq 1$, the family of morphisms 
\begin{equation}
\left\{\rho_k(\bar{n},l):P_{T+(2-k)}\longrightarrow \mathbb V_{\bar{n}}P_l\right\}_{(\bar{n},l)\in\mathbb Z(T,k)}
\end{equation}
in $\mathcal C$ is locally finite.

\end{defn}

In the rest of this section, we suppose that each $\{\mathbb V_n\}_{n\in \mathbb Z}$ preserves monomorphisms. Given a $(\mathbb V,\Delta)$-comodule $(P,\rho)$, we now consider, in the notation of Definition \ref{D2.5}, a family of subobjects 
$\{P_n^\alpha\hookrightarrow P_n\}_{n\in \mathbb Z}$ satisfying the following conditions

\smallskip
(a) For any $\bar{n}\in \mathbb Z^{k-1}$, $l\in \mathbb Z$, the morphisms in \eqref{2.25juk} restrict to the corresponding subobjects
\begin{equation}\label{2.59juk}
\rho^\alpha=\{\mbox{$\rho_k^\alpha(\bar{n},l):P^\alpha_{\Sigma\bar{n}+l+(2-k)}\longrightarrow \mathbb V_{\bar{n}}P^\alpha_{l} $ $\vert$ $k\geq 1$, $\bar{n}=(n_1,...,n_{k-1})
\in \mathbb Z^{k-1}$, $l\in\mathbb Z$ } \}
\end{equation}

(b) For  any $T\in \mathbb Z$ and $k\geq 1$, the family of morphisms
\begin{equation}\label{2.60vk}
\left\{\rho_k^\alpha(\bar{n},l):P^\alpha_{T+(2-k)}\longrightarrow \mathbb V_{\bar{n}}P^\alpha_l\right\}_{(\bar{n},l)\in\mathbb Z(T,k)}
\end{equation}  is locally finite in $\mathcal C$. 

\begin{lem}\label{L2.3.2} $(P^\alpha,\rho^\alpha)$ is a locally finite $(\mathbb V,\Delta)$-comodule.

\end{lem}

\begin{proof}
Because the morphisms $\rho=\{\mbox{$\rho_k(\bar{n},l)$ $\vert$ $k\geq 1$, $\bar{n}=(n_1,...,n_{k-1})
\in \mathbb Z^{k-1}$, $l\in\mathbb Z$ } \}$ in \eqref{2.25juk} restrict to 
\begin{equation} \rho^\alpha=\{\mbox{$\rho_k^\alpha(\bar{n},l)$ $\vert$ $k\geq 1$, $\bar{n}=(n_1,...,n_{k-1})
\in \mathbb Z^{k-1}$, $l\in\mathbb Z$ } \}
\end{equation} the morphisms $\rho^\alpha(\bar{n},l)$ also satisfy the relations in \eqref{l2.14}. The local finiteness condition is clear from the assumption in 
\eqref{2.60vk}. 
\end{proof}

We now let $\{(P^\alpha,\rho^\alpha)\}_{\alpha\in A}$ denote the collection of such families of subobjects. We note that this collection is non-empty since it contains the zero family. We now define
\begin{equation}\label{2.61ja}
P^{\oplus}=\left\{P^{\oplus}_n:=\underset{\alpha\in A}{\sum} P^\alpha_n\right\}_{n\in \mathbb Z}
\end{equation}

\begin{thm}
The collection $P^{\oplus}=\left\{P^{\oplus}_n\right\}_{n\in \mathbb Z}$ is a locally finite $A_\infty$-comodule over $(\mathbb V,\Delta)$. 
\end{thm}
\begin{proof}
Since each $\mathbb V_n$ preserves monomorphisms, for $\bar{n}=(n_1,...,n_{k-1})
\in \mathbb Z^{k-1}$ and $l\in\mathbb Z$, we have  induced maps
\begin{equation}\label{e2.62k}
\begin{CD}
\rho^{\oplus}_k(\bar{n},l):P^{\oplus}_{\Sigma\bar{n}+l+(2-k)}=\underset{\alpha\in A}{\sum} P^\alpha_{\Sigma\bar{n}+l+(2-k)}@>\underset{\alpha\in A}{\sum} \rho_k^\alpha(\bar{n},l)>> \underset{\alpha\in A}{\sum} \mathbb V_{\bar{n}}P^\alpha_{l}\hookrightarrow \mathbb V_{\bar{n}}\left(\underset{\alpha\in A}{\sum} P^\alpha_l\right)=\mathbb V_{\bar{n}}(P^{\oplus}_l)
\end{CD}
\end{equation} Again since the morphisms in \eqref{e2.62k} are restrictions of $\rho_k(\bar{n},l)$ to subobjects, it follows that they also satisfy the relations in \eqref{l2.14}. This makes
$(P^\oplus,\rho^\oplus)$ a $(\mathbb V,\Delta)$-comodule. Since each $(P^\alpha,\rho^\alpha)$ is locally finite, we have for each $\alpha\in A$, $T\in \mathbb Z$ and $k\geq 1$, the following factorisation
\begin{equation}\label{2.63sq}
\begin{tikzcd}[row sep = 20pt, column sep = 30pt]
P^\alpha_{T+(2-k)} \arrow[r] \arrow[dr] & \left(\underset{(\bar{n},l)\in \mathbb Z(T,k)}{\prod}  \mathbb V_{\bar{n}}P^\alpha_l\right)\arrow[r] &  \left(\underset{(\bar{n},l)\in \mathbb Z(T,k)}{\prod} \mathbb V_{\bar{n}}P^\oplus_l \right) \\
 & \left( \underset{(\bar{n},l)\in \mathbb Z(T,k)}{\bigoplus} \mathbb V_{\bar{n}}P^\alpha_l\right)\arrow[r] \arrow[u] &   \left(\underset{(\bar{n},l)\in \mathbb Z(T,k)}{\bigoplus} \mathbb V_{\bar{n}}P^\oplus_l\right) \arrow[u]\\
\end{tikzcd}
\end{equation} From \eqref{2.63sq}, it follows that each of the compositions $\{P^\alpha_{T+(2-k)}\longrightarrow P^\oplus_{T+(2-k)}
\longrightarrow \underset{(\bar{n},l)\in \mathbb Z(T,k)}{\prod} \mathbb V_{\bar{n}}P^\oplus_l\}_{\alpha\in A}$ factors through the subobject $\underset{(\bar{n},l)\in \mathbb Z(T,k)}{\bigoplus} \mathbb V_{\bar{n}}P^\oplus_l$. Since $\mathcal C$ is abelian and $ P^\oplus_{T+(2-k)}=\sum_{\alpha\in A} P^\alpha_{T+(2-k)}$, the result follows. 
\end{proof}

Let $(P,\rho^P)$ and $(Q,\rho^Q)$ be $(\mathbb V,\Delta)$-comodules. By Definition \ref{D2.5}, a morphism $g:(P,\rho^P)\longrightarrow (Q,\rho^Q)$ in
$EM^{(\mathbb V,\Delta)}$ consists of a collection of morphisms $g=\{g_l:P_l\longrightarrow Q_l\}_{l\in \mathbb Z}$ such that 
\begin{equation}\label{2.64er}
\begin{CD}
P_{\Sigma \bar{n}+l+(2-k)} @>g_{\Sigma \bar{n}+l+(2-k)}>> Q_{\Sigma \bar{n}+l+(2-k)} \\
@V\rho^P_{k}(\bar{n},l)VV @VV\rho^Q_{k}(\bar{n},l)V \\
\mathbb V_{\bar{n}}P_l @>\mathbb V_{\bar{n}}(g_l)>> \mathbb V_{\bar{n}}Q_l\\
\end{CD}
\end{equation} is commutative for any $\bar{n}\in \mathbb Z^{k-1}$, $k\geq 1$, $l\in \mathbb Z$. 

\begin{lem}\label{L2.3.4} 
Let $(\mathbb V,\Delta)$ be an $A_\infty$-comonad such that each $\{\mathbb V_n\}_{n\in\mathbb Z}$ is exact. Suppose that $g:(P,\rho^P)\longrightarrow (Q,\rho^Q)$ is a morphism in  $EM^{(\mathbb V,\Delta)}$. Then, 
\begin{equation}
R=\{R_l:=Ker(g_l:P_l\longrightarrow Q_l)\}_{l\in \mathbb Z} \qquad S=\{S_l:=Cok(g_l:P_l\longrightarrow Q_l)\}_{l\in \mathbb Z}
\end{equation} are canonically equipped with the structure of $(\mathbb V,\Delta)$-comodules. 
\end{lem}

\begin{proof} Since the functors $\mathbb V_{\bar{n}}$ are all exact, it is clear from \eqref{2.64er} that we have induced morphisms
\begin{equation}\label{2.66et}
\rho_k^R(\bar{n},l):R_{\Sigma \bar{n}+l+(2-k)} \longrightarrow \mathbb V_{\bar{n}}R_l\qquad \rho_k^S(\bar{n},l):S_{\Sigma \bar{n}+l+(2-k)} \longrightarrow \mathbb V_{\bar{n}}S_l
\end{equation} Again since the functors $\mathbb V_{\bar{n}}$ are all exact, it follows that the morphisms in \eqref{2.66et} satisfy the conditions in \eqref{l2.14} for $R$ and $S$ to 
be $(\mathbb V,\Delta)$-comodules. 

\end{proof}

\begin{lem}\label{L2.3.5} 
Let $(\mathbb V,\Delta)$ be an $A_\infty$-comonad such that each $\{\mathbb V_n\}_{n\in\mathbb Z}$ is exact. Suppose that $g:(P,\rho^P)\longrightarrow (Q,\rho^Q)$ is a morphism in  $EM^{(\mathbb V,\Delta)}$ such that each $g_l:P_l\longrightarrow Q_l$ is an epimorphism. Then if $P$ is locally finite, so is $Q$. 
\end{lem}

\begin{proof}
For any $T\in \mathbb Z$  and $k\geq 1$, we consider the families
\begin{equation}
\{\rho_k^P(\bar{n},l):P_{T+(2-k)}\longrightarrow \mathbb V_{\bar{n}}P_l\}_{(\bar{n},l)\in\mathbb Z(T,k)}\qquad 
\{\rho_k^Q(\bar{n},l):Q_{T+(2-k)}\longrightarrow \mathbb V_{\bar{n}}Q_l\}_{(\bar{n},l)\in\mathbb Z(T,k)}
\end{equation} It is immediate that $\left(\underset{(\bar{n},l)\in\mathbb Z(T,k)}{\prod}\rho_k^Q(\bar{n},l)\right)\circ g_{T+(2-k)}=\left(\underset{(\bar{n},l)\in\mathbb Z(T,k)}{\prod}\mathbb V_{\bar{n}}(g_l)\right)\circ\left(\underset{(\bar{n},l)\in\mathbb Z(T,k)}{\prod}\rho_k^P(\bar{n},l)\right)$.
\end{proof} We now consider the commutative diagram
\begin{equation}\label{2.68cd}
\begin{CD}
R_{T+(2-k)}@>\upsilon_{T+(2-k)}>> P_{T+(2-k)} @>\bigoplus \rho^P_k(\bar{n},l)>>\underset{(\bar{n},l)\in\mathbb Z(T,k)}{\bigoplus} \mathbb V_{\bar{n}}P_l@>\iota^P_{T,k}>>\underset{(\bar{n},l)\in\mathbb Z(T,k)}{\prod}  \mathbb V_{\bar{n}}P_l\\
@. @Vg_{T+(2-k)}VV @VV{ {\bigoplus}\mathbb V_{\bar{n}}(g_l)}V @VV{ {\prod}\mathbb V_{\bar{n}}(g_l)}V\\
@. Q_{T+(2-k)} @. \underset{(\bar{n},l)\in\mathbb Z(T,k)}{\bigoplus} \mathbb V_{\bar{n}}Q_l @>\iota^Q_{T,k}>> \underset{(\bar{n},l)\in\mathbb Z(T,k)}{\prod}  \mathbb V_{\bar{n}}Q_l\\
\end{CD}
\end{equation} where $R_{T+(2-k)}$ is the kernel of $g_{T+(2-k)}:P_{T+(2-k)}\longrightarrow Q_{T+(2-k)}$ and the factorization $\iota^P_{T,k}\circ \left(\bigoplus \rho^P_k(\bar{n},l)\right)=\prod\rho^P_k(\bar{n},l)$ is due to the fact that $(P,\rho^P)$ is locally finite. From \eqref{2.68cd} we see that
\begin{equation}\label{e2.69}
\begin{array}{ll}
\iota^Q_{T,k}\circ \left({ {\bigoplus}\mathbb V_{\bar{n}}(g_l)}\right) \circ \left(\bigoplus \rho^P_k(\bar{n},l)\right)\circ \upsilon_{T+(2-k)}&= \left({ {\prod}\mathbb V_{\bar{n}}(g_l)}\right)\circ \iota^P_{T,k}\circ \left(\bigoplus \rho^P_k(\bar{n},l)\right)\circ \upsilon_{T+(2-k)}\\ &=\left({ {\prod}\mathbb V_{\bar{n}}(g_l)}\right)\circ \left(\prod \rho_k^P(\bar{n},l)\right)\circ \upsilon_{T+(2-k)}\\
&=\left( {\prod}\rho_k^Q(\bar{n},l)\right)\circ g_{T+(2-k)}\circ \upsilon_{T+(2-k)}=0\\
\end{array}
\end{equation} Since $\iota^Q_{T,k}$ is a monomorphism, it follows from \eqref{e2.69} that $ \left({ {\bigoplus}\mathbb V_{\bar{n}}(g_l)}\right) \circ \left(\bigoplus \rho^P_k(\bar{n},l)\right)\circ \upsilon_{T+(2-k)}=0$. Since $R_{T+(2-k)}$ is the kernel of the epimorphism $g_{T+(2-k)}:P_{T+(2-k)}\longrightarrow Q_{T+(2-k)}$, we have an induced map $Q_{T+(2-k)}\longrightarrow \underset{(\bar{n},l)\in\mathbb Z(T,k)}{\bigoplus} \mathbb V_{\bar{n}}Q_l $ which fits into the commutative diagram \eqref{2.68cd}. Again since $g_{T+(2-k)}$ is an epimorphism, it is easily seen that this  gives a factorization of $ {\prod}\rho_k^Q(\bar{n},l)$.

\begin{Thm}\label{T2.3.6} Let $(\mathbb V,\Delta)$ be an $A_\infty$-comonad such that each $\{\mathbb V_n\}_{n\in\mathbb Z}$ is exact. Then, the association $Q\mapsto Q^\oplus$ determines a functor from $EM^{(\mathbb V,\Delta)}$ to the full subcategory $EM^{(\mathbb V,\Delta)}_{<\infty}$ of comodules  in $EM^{(\mathbb V,\Delta)}$ that are locally finite. Further, this functor is right adjoint to the inclusion $EM^{(\mathbb V,\Delta)}_{<\infty}\hookrightarrow EM^{(\mathbb V,\Delta)}$, i.e., there are natural isomorphisms
\begin{equation}
EM^{(\mathbb V,\Delta)}(P,Q)\cong EM^{(\mathbb V,\Delta)}_{<\infty}(P,Q^\oplus)
\end{equation}
for $P\in EM^{(\mathbb V,\Delta)}_{<\infty}$ and $Q\in EM^{(\mathbb V,\Delta)}$.

\end{Thm} 

\begin{proof}
We consider some $P\in EM^{(\mathbb V,\Delta)}_{<\infty}$,  $Q\in EM^{(\mathbb V,\Delta)}$ and a   morphism $g:P\longrightarrow Q$ in $EM^{(\mathbb V,\Delta)}$. Applying Lemma \ref{L2.3.4}, we can consider the $(\mathbb V,\Delta)$-comodule $R$ determined by setting
\begin{equation}
P_n\twoheadrightarrow R_n:=Im(g_n:P_n\longrightarrow Q_n)\hookrightarrow Q_n\qquad n\in \mathbb Z
\end{equation}  Since each $P_n\twoheadrightarrow R_n$ is an epimorphism and $P$ is locally finite, it follows from Lemma \ref{L2.3.5} that $R$ is also locally finite. As such, the family $\{R_n\hookrightarrow Q_n\}_{n\in \mathbb Z}$ of subobjects determines a locally finite $(\mathbb V,\Delta)$-comodule. From the definition in \eqref{2.61ja}, it is now clear that each $R_n\subseteq Q^\oplus_n$. 

\smallskip
It remains to show that the association $Q\mapsto Q^\oplus$ is a functor. For this, we consider a  morphism $Q'\longrightarrow Q$ in $EM^{(\mathbb V,\Delta)}$. Then $Q'^\oplus \in EM^{(\mathbb V,\Delta)}_{<\infty}$ and if we apply the above reasoning to the composition $Q'^\oplus\longrightarrow Q'\longrightarrow Q$ in $EM^{(\mathbb V,\Delta)}$, we see that there is an induced morphism $Q'^\oplus\longrightarrow Q^\oplus$. This proves the result.
\end{proof}

\section{Bar construction for $A_\infty$-monads}

In this section, we will obtain the bar construction $Bar(\mathbb U,\Theta)$ of an $A_\infty$-monad $(\mathbb U,\Theta)$, which gives a differential graded comonad on $\mathcal C$. Further, we show that morphisms to  $Bar(\mathbb U,\Theta)$  from a conilpotent dg-comonad correspond to families of natural transformations that behave in a manner similar to the classical case of twisting morphisms (see, for instance, \cite[$\S$ 3]{BK}, \cite[$\S$ 4]{Kelx}, \cite[$\S$ 2]{LodV}). We mention here that the cobar construction of an $A_\infty$-comonad can be obtained in an analogous manner. We begin with the following definition.

\begin{defn}\label{D6.1cz}
A graded comonad $(\mathbb W,\delta_2)$ on $\mathcal C$ is given by the following

\smallskip
(a) A collection $\{\mathbb W_n:\mathcal C\longrightarrow \mathcal C\}_{n\in \mathbb Z}$ of endofunctors on $\mathcal C$.

\smallskip
(b) A collection of natural transformations
\begin{equation}\label{b6.1hp}
\delta_2=\{\mbox{$\delta_2(n_1,n_2):\mathbb W_{n_1+n_2}\longrightarrow \mathbb W_{n_1}\circ \mathbb W_{n_2}$ $\vert$ $(n_1,n_2)\in \mathbb Z^2$}\}
\end{equation} such that for each $T\in \mathbb Z$, the induced morphism
\begin{equation}\label{b6.2hp}
\underset{n_1+n_2=T}{\prod} \delta_2(n_1,n_2): \mathbb W_T\longrightarrow \underset{n_1+n_2=T}{\prod}\mathbb W_{n_1}\circ \mathbb W_{n_2}
\end{equation} factors through the direct sum $\underset{n_1+n_2=T}{\bigoplus}\mathbb W_{n_1}\circ \mathbb W_{n_2}$ and for every $n_1$, $n_2$, $n_3\in \mathbb Z$, we have
\begin{equation}\label{53ed}
(\mathbb W_{n_1}\ast \delta_2(n_2,n_3))\delta_2(n_1,n_2+n_3)=(\delta_2(n_1,n_2)\ast \mathbb W_{n_3})\delta_2(n_1+n_2,n_3):\mathbb W_{n_1+n_2+n_3}\longrightarrow \mathbb W_{n_1}\circ \mathbb W_{n_2}\circ \mathbb W_{n_3}
\end{equation} A morphism $\beta:(\mathbb W,\delta_2)\longrightarrow (\mathbb W',\delta'_2)$ of graded comonads on $\mathcal C$ is a collection of natural transformations 
$\beta=\{\beta_n:\mathbb W_n\longrightarrow \mathbb W'_n\}_{n\in \mathbb Z}$ such that $ (\beta_{n_1}\ast \beta_{n_2})\delta_2(n_1,n_2)=\delta'_2(n_1,n_2)\beta_{n_1+n_2} $
for $n_1$, $n_2\in \mathbb Z$. We will denote by $gr-Comon(\mathcal C)$ the category of graded comonads over $\mathcal C$. 

\smallskip
A differential graded comonad (or dg-comonad) $(\mathbb W,\delta_1,\delta_2)$ on  $\mathcal C$ consists of a graded comonad $(\mathbb W,\delta_2)$ as well
as a collection   of natural transformations
$
\delta_1=\{\mbox{$\delta_1(n):\mathbb W_n\longrightarrow \mathbb W_{n+1}$ $\vert$ $n\in \mathbb Z$}\}
$ satisfying the following conditions (for $n$, $n_1$, $n_2 \in \mathbb Z$)
\begin{equation}\small 
\begin{array}{c}
\delta_1(n+1)\delta_1(n)=0\\
\delta_2(n_1,n_2)\delta_1(n_1+n_2-1)=(\delta_1(n_1-1)\ast \mathbb W_{n_2})\delta_2(n_1-1,n_2)+(-1)^{n_1}(\mathbb W_{n_1}\ast \delta_1(n_2-1))\delta_2(n_1,n_2-1):\mathbb W_{n_1+n_2-1}\longrightarrow \mathbb W_{n_1}\circ \mathbb W_{n_2}\\
\end{array}
\end{equation}

\smallskip
A morphism $\beta:(\mathbb W,\delta_1,\delta_2)\longrightarrow (\mathbb W',\delta'_1,\delta'_2)$ of dg-comonads on $\mathcal C$ is a collection of natural transformations 
$\beta=\{\beta_n:\mathbb W_n\longrightarrow \mathbb W'_n\}_{n\in \mathbb Z}$ such that
\begin{equation}\label{b6.7hp}
\delta'_1(n)  \beta_n=\beta_{n+1}\delta_1(n)\qquad (\beta_{n_1}\ast \beta_{n_2})\delta_2(n_1,n_2)=\delta'_2(n_1,n_2)\beta_{n_1+n_2} 
\end{equation} for $n$, $n_1$, $n_2\in \mathbb Z$. We will denote by $dg-Comon(\mathcal C)$ the category of dg-comonads over $\mathcal C$. 
\end{defn}

If $(\mathbb W,\delta_2)$ is a graded comonad on $\mathcal C$, we set for any $\bar{n}=(n_1,...,n_k)\in \mathbb Z^k$, $k\geq 1$
\begin{equation}\label{del26}
\delta_2(\bar{n}):=(\mathbb W_{n_1}\ast ... \ast \mathbb W_{n_{k-2}}\ast \delta_2(n_{k-1},n_k))...(\mathbb W_{n_1}\ast \delta_2(n_3+...+n_k))\delta_2(n_1,n_2+n_3+....+n_k):
\mathbb W_{\Sigma \bar{n}}\longrightarrow \mathbb W_{n_1}\circ ...\circ \mathbb W_{n_k}=\mathbb W_{\bar{n}}
\end{equation} 
We note that due to the coassociativity condition in \eqref{53ed}, the natural transformation  $\delta_2(\bar{n})$ appearing in 
\eqref{del26} can be expressed in a number of equivalent ways, one for each way the sum $(n_1+...+n_k)$ can be ``broken up'' into the partition $(n_1,...,n_k)$. 

  \smallskip
 
 From the conditions in Definition \ref{D6.1cz}, we know that for each  $T\in \mathbb Z$ and $k\geq 1$, the morphism $\underset{\bar{n}\in \mathbb Z(T,k) }{\prod}
 \delta_2(\bar{n}):\mathbb W_T\longrightarrow \underset{\bar{n}\in \mathbb Z(T,k) }{\prod} \mathbb W_{\bar{n}}$ factors through the direct sum $\underset{\bar{n}\in \mathbb Z(T,k) }{\bigoplus} \mathbb W_{\bar{n}}$. We now have the following definition.
 
 \begin{defn}\label{conil7}
 Let $(\mathbb W,\delta_2)$ be a graded comonad over $\mathcal C$. We will say that  $(\mathbb W,\delta_2)$  is conilpotent if for each object $M\in \mathcal C$ and $L\in \mathbb Z$, we have
 \begin{equation}\label{conil74}
 \mathbb W_L(M)=\underset{k\geq 1}{\bigcup} Ker\left(\mathbb W_L(M)\xrightarrow{\qquad \underset{\bar{n}\in \mathbb Z(T,k) }{\bigoplus}\delta_2(\bar{n})(M)\qquad } \underset{\bar{n}\in \mathbb Z(L,k) }{\bigoplus} \mathbb W_{\bar{n}}(M) \right)
 \end{equation}
 \end{defn}

\smallskip
We now let $(\mathbb U,\Theta)$ be an  $A_\infty$-monad, and let $(\mathbb W,\delta_1,\delta_2)$ be a conilpotent dg-comonad over $\mathcal C$. For each $n\in \mathbb Z$, we set 
\begin{equation}\label{b6.71hp}
A_n:=\underset{k\in \mathbb Z}{\prod} \textrm{ }Nat(\mathbb W_k,\mathbb U_{k+n})=\{\mbox{$\zeta=(\zeta^{(k)})_{k\in \mathbb Z}$ $\vert$ $\zeta^{(k)}\in Nat(\mathbb W_k,
\mathbb U_{k+n})$} \}\end{equation} where $Nat(\mathbb W_k,\mathbb U_{k+n})$ is the collection of natural transformations from $\mathbb W_k$ to 
$\mathbb U_{k+n}$. We set $A:=\underset{n\in \mathbb Z}{\bigoplus}A_n$. We now have maps
$\{m_k:A^{\otimes k}\longrightarrow A\}_{k\geq 1}$ with $m_k$ of degree $(2-k)$, given by  components 
\begin{equation}\label{7x.1}
\begin{array}{c}
m_k(\bar{n}):A_{n_1}\otimes A_{n_2}\otimes ...\otimes A_{n_k} \longrightarrow A_{\Sigma \bar{n}+(2-k)} \qquad
(\zeta_{n_1}\otimes ... \otimes \zeta_{n_k})\mapsto \left(\underset{\bar{t}\in \mathbb Z(l,k)}{\sum}\theta_k(\bar{t}+\bar{n})(\zeta_{n_1}^{(t_1)}\ast ... \ast \zeta_{n_k}^{(t_k)})\delta_2(\bar{t})\right)_{l\in \mathbb Z}\\
\end{array}
\end{equation}
 for $\bar{n}=(n_1,n_2,..., n_k)\in \mathbb Z^k$. In \eqref{7x.1}, we have suppressed the sign of the  summand $\theta_k(\bar{t}+\bar{n})(\zeta_{n_1}^{(t_1)}\ast ... \ast \zeta_{n_k}^{(t_k)})\delta_2(\bar{t})$, which is 
$
 (-1)^{t_1(n_2+...+n_k)+t_2(n_3+...+n_k)+....+t_{k-1}n_k}
$ as given by Koszul sign rule.  For ease of notation, we set $A_{\bar{n}}:=A_{n_1}\otimes A_{n_2}\otimes ...\otimes A_{n_k}$. Accordingly, an element 
 $(\zeta_{n_1}\otimes ... \otimes \zeta_{n_k})\in A_{n_1}\otimes A_{n_2}\otimes ...\otimes A_{n_k}=A_{\bar{n}}$ will often be denoted by $\zeta_{\bar{n}}$. Similarly, we write
 $\zeta_{\bar{n}}^{(\bar{t})}=\zeta_{n_1}^{(t_1)}\ast ... \ast \zeta_{n_k}^{(t_k)}$ for $\bar{n}$, $\bar{t}\in \mathbb Z^k$.  We now have the following result.
 
 \begin{thm}\label{P7x.2w}
 The collection $(A,\{m_k:A^{\otimes k}\longrightarrow A\}_{k\geq 1})$ is an $A_\infty$-algebra.
 \end{thm}
 
 \begin{proof}
 We choose $\bar{z}\in \mathbb Z^N$ for some $N\geq 1$ and consider $\zeta_{\bar{z}}\in A_{\bar{z}}$. In the following, we have suppressed the signs for sake of convenience. With the sum running over partitions $\bar{z}=(\bar{n}, \bar{n}',\bar{n}'')$ with $|\bar{n}|=p$, $|\bar{n}'|=q$, $|\bar{n}''|=r$, we  see that we have for any $l\in \mathbb Z$ 
 \begin{equation*}\small
 \begin{array}{l}
 \left(\underset{\bar{z}=(\bar{n}, \bar{n}',\bar{n}'')}{\sum} m_{p+1+r}(\bar{n},\Sigma\bar{n}'+(2-q),\bar{n}'')(id \otimes m_q(\bar{n}')\otimes id)(\zeta_{\bar{z}})\right)^{(l)}\\
 =\left(\underset{\bar{z}=(\bar{n}, \bar{n}',\bar{n}'')}{\sum}m_{p+1+r}(\bar{n},\Sigma\bar{n}'+(2-q),\bar{n}'')(\zeta_{\bar{n}} \otimes m_q(\bar{n}')(\zeta_{\bar{n}'})\otimes \zeta_{\bar{n}''})\right)^{(l)}\\
 =\underset{\bar{t}\in \mathbb Z(l,p+1+r)}{\sum}\textrm{ }\underset{\bar{z}=(\bar{n}, \bar{n}',\bar{n}'')}{\sum} \theta_{p+1+r}(\bar{t}+(\bar{n},\Sigma\bar{n}'+(2-q),\bar{n}''))(\zeta_{n_1}^{(t_1)} 
 \ast ...\ast \zeta_{n_p}^{(t_p)}\ast m_q(\bar{n}')(\zeta_{\bar{n}'})^{(t_{p+1})}\ast \zeta_{\bar{n}''}^{(t_{p+1+r})})\delta_2(\bar{t})\\
  =\underset{\bar{t}\in \mathbb Z(l,N)}{\sum}\textrm{ }\underset{\bar{z}=(\bar{n}, \bar{n}',\bar{n}'')}{\sum} \theta_{p+1+r}( (t_1,...,t_p,\sum_{y=1}^qt_{p+y},t_{p+q+1},...
  ,t_{p+q+r})+(\bar{n},\Sigma\bar{n}'+(2-q),\bar{n}''))  (\mathbb U_{\bar{n}} \ast \theta_q((t_{p+1},...,t_{p+q})+\bar{n}')\ast \mathbb U_{\bar{n}''})(\zeta_{\bar{z}}^{(\bar{t})})\delta_2(\bar{t})\\
   =\underset{\bar{t}\in \mathbb Z(l,N)}{\sum}\textrm{ }\underset{\bar{z}+\bar{t}=(\bar{n}, \bar{n}',\bar{n}'')}{\sum} \theta_{p+1+r}(\bar{n},\Sigma\bar{n}'+(2-q),\bar{n}'')(\mathbb U_{\bar{n}} \ast \theta_q(\bar{n}')\ast \mathbb U_{\bar{n}''})(\zeta_{\bar{z}}^{(\bar{t})})\delta_2(\bar{t})=0\\
 \end{array}
 \end{equation*}
 \end{proof}
 
 As in \cite[$\S$ 3.6]{BK}, we recall here that any $A_\infty$-algebra $(A,\{m_k:A^{\otimes k}\longrightarrow A\}_{k\geq 1})$ can be equivalently described using structure maps
 $(sA,\{b_k:(sA)^{\otimes k}\longrightarrow (sA)\}_{k\geq 1})$ given by
 \begin{equation}\label{7x.1sw}
b_k(\bar{n}):=(-1)^{n_1(k-1)+...+n_{k-1}}m_k(n_1+1,...,n_k+1):(sA)_{n_1}\otimes ...\otimes (sA)_{n_k}\longrightarrow (sA)_{\Sigma\bar{n}+1}
 \end{equation}
  for $\bar{n}=(n_1,n_2,..., n_k)\in \mathbb Z^k$, where $sA=\{(sA)_j:=A_{j+1}\}_{j\in \mathbb Z}$ is the suspension of $A$.
  
\smallskip
 
We will need the  following  basic example of a conilpotent graded comonad for describing the bar construction. Let $F=\{F_n:\mathcal C\longrightarrow \mathcal C\}_{n\in 
\mathbb Z}$  be a family of endofunctors. As always, for any $k\geq 1$ and $\bar{n}=(n_1,...,n_k)\in \mathbb Z^k$, we set $F_{\bar{n}}:= F_{n_1}\circ 
... \circ F_{n_k}$. We now consider the graded comonad $\mathbb T^c(F)$ given by setting 
 \begin{equation}\label{tccot6}
 \mathbb T^c(F)_L:=\underset{\bar{n}\in \mathbb Z(L,k),k\geq 1}{\bigoplus} F_{\bar{n}}\qquad L\in \mathbb Z
 \end{equation} The ``comultiplication'' $\delta^F_2:\mathbb T^c(F)\longrightarrow \mathbb T^c(F)\circ \mathbb T^c(F)$ is given on components by means of identity maps
 $F_{\bar{n}}\longrightarrow F_{\bar{n}_1}\circ F_{\bar{n}_2}$
for each partition $\bar{n}=(\bar{n_1},\bar{n_2})$ of $\bar{n}$. If we consider some $M\in \mathcal C$, we see that $\delta^F_2(\bar{t})(M)|_{F_{\bar{n}}(M)}=0$ 
for any $\bar{t}$ such that $|\bar{t}|>|\bar{n}|$. From the condition in \eqref{conil74}, it  is now clear that $(\mathbb T^c(F),\delta^F_2)$ is a conilpotent graded comonad.

\begin{thm}\label{P5.4aj} Let $(\mathbb W,\delta_2)$ be a conilpotent graded comonad over $\mathcal C$ and $F=\{F_n:\mathcal C\longrightarrow \mathcal C\}_{n\in 
\mathbb Z}$  be a family of endofunctors. Then, there is a one-one correspondence
\begin{equation}\label{5.11ck}
\underset{L\in \mathbb Z}{\prod} Nat(\mathbb W_L, F_L)\xleftrightarrows[\text{$\qquad h\qquad$}]{\text{$\qquad g\qquad$}}  gr-Comon(\mathcal C)((\mathbb W,\delta_2), ( \mathbb T^c(F),\delta_2^F))
\end{equation}

\end{thm}

\begin{proof}
For  $\zeta=\{\zeta^{(L)}\}_{L\in \mathbb Z}\in \underset{L\in \mathbb Z}{\prod} Nat(\mathbb W_L, F_L)$, we get $\chi:=h(\zeta)$ by setting for each $M\in \mathcal C$
\begin{equation}\label{conil64}
\chi_{L}(M)=\underset{\bar{n}\in \mathbb Z(L,k),k\geq 1}{\bigoplus}\chi_{\bar{n}}(M)=\underset{\bar{n}\in \mathbb Z(L,k),k\geq 1}{\bigoplus}\zeta^{(\bar{n})}\delta_2(\bar{n})(M):\mathbb W_L(M)\longrightarrow  \mathbb T^c(F)_L(M)=\underset{\bar{n}\in \mathbb Z(L,k),k\geq 1}{\bigoplus} F_{\bar{n}}(M)
\end{equation} where $\zeta^{(\bar{n})}=\zeta^{(n_1)}\ast ... \ast \zeta^{(n_k)}$ in \eqref{conil64} for $\bar{n}=(n_1,...,n_k)\in \mathbb Z^k$. It is clear from the conilpotence
condition in \eqref{conil74} that the morphism in \eqref{conil64} is well defined. On the other hand, for $\chi \in   gr-Comon(\mathcal C)((\mathbb W,\delta_2), ( \mathbb T^c(F),\delta_2^F))$, we get $\zeta:=g(\chi)$ by setting $\zeta^{(L)}$ to be the composition of  $\chi_L(M):\mathbb W_L(M)\longrightarrow \mathbb T^c(F)_L(M)$ with the projection to $F_L(M)$ for each $M\in 
\mathcal C$. It may be verified that these two associations are inverse to each other.
\end{proof}

\begin{lem}\label{L5.5tfc}
Let $F=\{F_n:\mathcal C\longrightarrow \mathcal C\}_{n\in 
\mathbb Z}$  be a family of endofunctors. Then, there is a one-one correspondence between the following

\smallskip
(a) Families of natural transformations
\begin{equation}\label{gam5c}
\Gamma=\{\mbox{$\gamma_k(\bar{n}):F_{\bar{n}}=F_{n_1}\circ ... \circ F_{n_k}\longrightarrow F_{\Sigma\bar{n}+1}$ $\vert$ $k\geq 1$ $\bar{n}=(n_1,...,n_k)\in 
\mathbb Z^k$}\}
\end{equation}

\smallskip
(b) Families of natural transformations $\delta_1^F=\{\mbox{$\delta^F_1(L):\mathbb T^c(F)_L\longrightarrow \mathbb T^c(F)_{L+1}$ $\vert$ $L\in \mathbb Z$}\}$  which satisfy
\begin{equation}\label{del1cF}
\delta_2^F(L_1,L_2)\delta_1^F(L_1+L_2-1)=(\delta_1^F(L_1-1)\ast \mathbb T^c(F)_{L_2})\delta_2^F(L_1-1,L_2)+(-1)^{L_1}(\mathbb T^c(F)_{L_1}\ast \delta_1^F(L_2-1))\delta_2^F(L_1,L_2-1) 
\end{equation} for $L_1$, $L_2\in \mathbb Z$. 
\end{lem}

\begin{proof}
Given a family $\Gamma$ as in (a), we set for each $\bar{z}\in \mathbb Z^N$, $N\geq 1$:
\begin{equation}\label{gadel5}
\delta_1^F(\Sigma \bar{z})|_{F_{\bar{z}}}:=\underset{\bar{z}=(\bar{n},\bar{n}',\bar{n}'')}{\sum} (-1)^{\Sigma \bar{n}} (F_{\bar{n}}\ast \gamma_{|\bar{n}'|}(\bar{n}')\ast F_{\bar{n}''}):F_{\bar{z}}\longrightarrow \underset{\bar{z}=(\bar{n},\bar{n}',\bar{n}'')}{\bigoplus}  F_{\bar{n}}\circ F_{\Sigma\bar{n}'+1}\circ F_{\bar{n}''}\hookrightarrow 
\mathbb T^c(\mathbb F)_{\Sigma\bar{z}+1}
\end{equation} where the sum in \eqref{gadel5} is taken over all partitions $\bar{z}=(\bar{n},\bar{n}',\bar{n}'')$. It may be verified that
the family $\delta_1^F$ satisfies the condition in \eqref{del1cF}. Conversely, given a family $\delta_1^F$ as in (b), we set for each
$\bar{z}\in \mathbb Z^N$, $N\geq 1$
\begin{equation}\label{delga5}
\gamma_{|\bar{z}|}(\bar{z}): F_{\bar{z}} \xrightarrow{\delta_1^F(\Sigma \bar{z})|_{F_{\bar{z}}}}\mathbb T^c(\mathbb F)_{\Sigma\bar{z}+1}\longrightarrow F_{\Sigma\bar{z}+1}
\end{equation} where the second arrow in \eqref{delga5} is the canonical projection. It may be checked that these two associations are inverse to each other.
\end{proof}

Given a collection $F=\{F_n:\mathcal C\longrightarrow \mathcal C\}_{n\in 
\mathbb Z}$ of endofunctors, we now set $sF=\{(sF)_n:=F_{n+1}\}_{n\in \mathbb Z}$. If $(\mathbb U,\Theta)$ is an $A_\infty$-monad, we can now consider the family $\Gamma^{s\mathbb U}$ 
of natural transformations
\begin{equation}\label{sFgam6}
\gamma^{s\mathbb U}_k(\bar{n}):(s\mathbb U)_{\bar{n}}=(s\mathbb U)_{n_1}\circ ... \circ (s\mathbb U)_{n_k}\longrightarrow (s\mathbb U)_{\Sigma\bar{n}+1}=\mathbb U_{\Sigma\bar{n}+2}\qquad \gamma^{s\mathbb U}_k(\bar{n}):=(-1)^{n_1(k-1)+...+n_{k-1}}\theta_k(n_1+1,...,n_k+1)
\end{equation} for $\bar{n}=(n_1,...,n_k)$, $k\geq 1$. Applying Lemma \ref{L5.5tfc}, the family of natural transformations in \eqref{sFgam6} corresponds to 
$\delta_1^{s\mathbb U}=\{\mbox{$\delta^{s\mathbb U}_1(L):\mathbb T^c(s\mathbb U)_L\longrightarrow \mathbb T^c(s\mathbb U)_{L+1}$ $\vert$ $L\in \mathbb Z$}\}$. Finally, using the fact that $(\mathbb U,\Theta)$ is an $A_\infty$-monad, it can be seen from the relations in \eqref{2.1e} that each $\delta_1^{s\mathbb U}(L+1)\delta_1^{s\mathbb U}(L)=0$. We will now say that the dg-cononad given by
\begin{equation}\label{barcon5r}
Bar(\mathbb U,\Theta):=(\mathbb T^c(s\mathbb U),\delta_1^{s\mathbb U},\delta_2^{s\mathbb U})
\end{equation} is the bar construction on the $A_\infty$-monad $(\mathbb U,\Theta)$. 

\begin{Thm}\label{T5.7kq}
Let $(\mathbb W,\delta_1,\delta_2)$ be a dg-comonad over $\mathcal C$ that is conilpotent. Let $(\mathbb U,\Theta)$ be an $A_\infty$-monad over $\mathcal C$ and let
$(A,\{m_l:A^{\otimes l}\longrightarrow A\}_{l\geq 1})$ be the $A_\infty$-algebra given by setting $A_l:=\underset{l'\in \mathbb Z}{\prod} \textrm{ }Nat(\mathbb W_{l'},\mathbb U_{l+l'})$. 
Then, there is a one-one correspondence between the following:

\smallskip
(a) Morphisms $\chi:(\mathbb W,\delta_1,\delta_2)\longrightarrow Bar(\mathbb U,\Theta)=(\mathbb T^c(s\mathbb U),\delta_1^{s\mathbb U},\delta_2^{s\mathbb U})$ of dg-comonads over $\mathcal C$

\smallskip
(b) Elements $\zeta =\{\zeta^{(L)}\}_{L\in \mathbb Z}\in \underset{L\in \mathbb Z}{\prod} Nat(\mathbb W_L, (s
\mathbb U)_L)=A_1$ satisfying
\begin{equation}\label{twco5r}\small
\begin{array}{l}
(\zeta^{(n_1-1)}\ast ...\ast\zeta^{(n_k-1)})\delta_2(n_1-1,...,n_k-1)\delta_1(M)\\
=\underset{i=0}{\overset{k-1}{\sum}}\underset{j=1}{\overset{\infty}{\sum}} (-1)^{n_1+...+n_i-i}(
\zeta^{(n_1-1)}\ast ...\ast \zeta^{(n_i-1)}\ast b_j(\zeta^{\otimes j})^{(n_{i+1}-2)}\ast \zeta^{(n_{i+2}-1)}\ast ...\ast \zeta^{(n_k-1)})\delta_2(n_1-1,...,n_i-1,n_{i+1}-2,n_{i+2}-1,...,n_{k}-1)(M)\\
\end{array}
\end{equation} for each fixed $k\geq 1$, $(n_1,...,n_k)\in \mathbb Z^k$ and $M\in \mathcal C$.
\end{Thm}

\begin{proof}
We consider a morphism $\chi:(\mathbb W,\delta_1,\delta_2)\longrightarrow (\mathbb T^c(s\mathbb U),\delta_1^{s\mathbb U},\delta_2^{s\mathbb U})$ of dg-comonads over $\mathcal C$. In particular, we see  that $\chi\in  gr-Comon(\mathcal C)((\mathbb W,\delta_2),(\mathbb T^c(s\mathbb U),\delta_2^{s\mathbb U}))$. Since $(\mathbb W,\delta_2)$ is conilpotent, it follows from Proposition \ref{P5.4aj} that the morphism $\chi$ corresponds to an element $\zeta =\{\zeta^{(L)}\}_{L\in \mathbb Z}\in \underset{L\in \mathbb Z}{\prod} Nat(\mathbb W_L, (s
\mathbb U)_L)=A_1$. Additionally,  the morphism $\chi$ of graded comonads is well behaved with respect to the coderivations $\delta_1$ and $\delta_1^{s\mathbb U}$. We now fix   $k\geq 1$, $\bar{n}=(n_1,...,n_k)\in \mathbb Z^k$ and $M\in \mathcal C$.  Composing   $\delta_1^{s\mathbb U}(M)$ with the morphisms induced by $\chi$, we have the sum of the following compositions
over $0\leq i\leq k-1$, $j\geq 1$
\begin{align}\label{cd520xl}
\begin{array}{c}
\xymatrix{
\mathbb W_{\Sigma\bar{n}-k-1}(M) \ar[dd]^{\underset{\bar{n}'\in \mathbb Z(n_{i+1}-2,j)}{\bigoplus}\zeta^{(n_1-1,...,n_i-1,\bar{n}',n_{i+2}-1,...,n_{k}-1)}\delta_2(n_1-1,...,n_i-1,\bar{n}',n_{i+2}-1,...,n_{k}-1)(M)} \\ \\
\underset{\bar{n}'\in \mathbb Z(n_{i+1}-2,j)}{\bigoplus} \mathbb U_{n_1}...\mathbb U_{n_{i}}
\mathbb U_{n'_{1}+1}....\mathbb U_{n'_{j}+1}\mathbb U_{n_{i+2}}...\mathbb U_{n_k}(M) 
 \ar[dd]^{(-1)^{n_1+...+n_i-i}(-1)^{n'_{1}(j-1)+...+n'_{j-1}}(\mathbb U_{n_1}\ast ... \ast \mathbb U_{n_i}\ast \theta_j(n'_{1}+1,...,n'_{j}+1)\ast \mathbb U_{n_{i+2}}\ast ...\ast \mathbb U_{n_{k}})(M)} \\ \\
 \mathbb U_{\bar{n}}=\mathbb U_{n_1}  ...  \mathbb U_{n_k}(M)\\
}
\end{array}
\end{align} Using the coassociativity condition on $\delta_2$ in \eqref{53ed} and the description of the structure maps of the $A_\infty$-algebra $\{b_l:(sA)^{\otimes l}\longrightarrow sA\}_{l\geq 1}$ in \eqref{7x.1sw}, we see that the sum of the compositions in \eqref{cd520xl} equates to
\begin{equation}\label{xlrijm}
\small
\underset{i=0}{\overset{k-1}{\sum}}\underset{j=1}{\overset{\infty}{\sum}} (-1)^{n_1+...+n_i-i}(
\zeta^{(n_1-1)}\ast ...\ast \zeta^{(n_i-1)}\ast b_j(\zeta^{\otimes j})^{(n_{i+1}-2)}\ast \zeta^{(n_{i+2}-1)}\ast ...\ast \zeta^{(n_k-1)})\delta_2(n_1-1,...,n_i-1,n_{i+1}-2,n_{i+2}-1,...,n_{k}-1)(M)
\end{equation} On the other hand, composing $\delta_1$ with the morphisms induced by $\chi$, we have the composition
\begin{equation}\label{xlrirn}
(\zeta^{(n_1-1)}\ast ...\ast\zeta^{(n_k-1)})\delta_2(n_1-1,...,n_k-1)\delta_1(M)
\end{equation} The result is now clear from \eqref{xlrijm} and \eqref{xlrirn}. 
\end{proof}

 \section{Distributive laws and lifting of $A_\infty$-monads to Eilenberg-Moore categories}
 
 Let $(\mathbb S,\theta_{\mathbb S},\iota_{\mathbb S})$ be a monad on $\mathcal C$, with ``multiplication'' $\theta_{\mathbb S}:\mathbb S\circ \mathbb S\longrightarrow
 \mathbb S$ and ``unit'' $\iota_{\mathbb S}:id\longrightarrow \mathbb S$. The standard Eilenberg-Moore category $EM_{\mathbb S}$ for the monad $\mathbb S$ consists of pairs $(M,\pi_M:
 \mathbb S(M)\longrightarrow M)$ satisfying $\pi_M\circ \mathbb S(\pi_M)=\pi_M\circ \theta_{\mathbb S}(M)$ and $id_M=\pi_M\circ \iota_{\mathbb S}(M)$ (see, for instance, 
 \cite{BBW}). The category $EM_{\mathbb S}$ is equipped with a pair $(L_{\mathbb S},R_{\mathbb S})$ of adjoint functors, where $R_{\mathbb S}:EM_{\mathbb S}
 \longrightarrow \mathcal C$ is the forgetful functor and its left adjoint $L_{\mathbb S}$ is given by
 \begin{equation}
 L_{\mathbb S}:\mathcal C\longrightarrow EM_{\mathbb S}\qquad M\mapsto (\mathbb S(M),\theta_{\mathbb S}(M))
 \end{equation} If $\mathbb T$ is another monad on $\mathcal C$, it is a classical fact that liftings of $\mathbb T$ to a monad $\widetilde{\mathbb T}$ on the Eilenberg-Moore category 
 of $\mathbb S$ correspond to monad distributive laws of $\mathbb S$ over $\mathbb T$ (see Beck \cite{Beck1969}).  In this section, we will prove similar results, showing how an $A_\infty$-monad on $\mathcal C$ can be lifted to an $A_\infty$-monad on $EM_{\mathbb S}$. We note that dual results may be proved for lifting of $A_\infty$-comonads to
 Eilenberg-Moore categories of comonads on $\mathcal C$.
 
 \smallskip
More explicitly, let $\mathbb F:\mathcal C\longrightarrow \mathcal C$ be an endofunctor. A lifting of $\mathbb F$ to the Eilenberg-Moore category $EM_{\mathbb S}$ is given by a functor
$\widetilde{\mathbb F}:EM_{\mathbb S}\longrightarrow EM_{\mathbb S}$ that fits into the following commutative diagram
\begin{equation}\label{c62cdj}
\begin{CD}
EM_{\mathbb S} @>\widetilde{\mathbb F}>> EM_{\mathbb S}\\
@VR_{\mathbb S}VV @VVR_{\mathbb S}V\\
\mathcal C @>\mathbb F>> \mathcal C\\
\end{CD}
\end{equation} We know (see \cite{Beck1969}) that such liftings are in one-one correspondence with distributive laws  of the monad $\mathbb S$ over the endofunctor $\mathbb F$. Such a distributive law consists of a natural transformation $\lambda: \mathbb S\mathbb F\longrightarrow \mathbb F\mathbb S$ such that the following two   diagrams commute
\begin{equation}
\label{dist6tf1}
\begin{array}{ccc}
\xymatrix{& \mathbb F \ar[dl]_{\iota_{\mathbb S}\mathbb F} \ar[dr]^{\mathbb F\iota_{\mathbb S}}&\\
\mathbb S\mathbb F\ar[rr]^{\lambda} && \mathbb F\mathbb S\\
}& \qquad &\xymatrix{
\mathbb S^2\mathbb F \ar[r]^{\mathbb S\lambda} \ar[d]_{\theta_{\mathbb S}\mathbb F}& \mathbb S\mathbb F\mathbb S \ar[r]^{\lambda\mathbb S} & \mathbb F\mathbb S^2
\ar[d]^{\mathbb F\theta_{\mathbb S}}\\
\mathbb S\mathbb F\ar[rr]^{\lambda}&&\mathbb F\mathbb S\\} \\
\end{array}
\end{equation} Given a distributive law as in \eqref{dist6tf1}, the lift of $\mathbb F$ to $EM_{\mathbb S}$ is defined by setting 
$\widetilde{\mathbb F}(M,\pi_M):=(\mathbb F(M),\mathbb F(\pi_M)\circ \lambda(M))$. Conversely, given a lifting 
$\widetilde{\mathbb F}$ of $\mathbb F$ as in \eqref{c62cdj}, the corresponding distributive law is defined by setting
\begin{equation}\label{dist64}
\lambda: \mathbb S\mathbb F\xrightarrow{\qquad\mathbb S\mathbb F\iota_{\mathbb S}\qquad}\mathbb S\mathbb F\mathbb S=\mathbb S\mathbb F R_{\mathbb S}L_{\mathbb S}=R_{\mathbb S}L_{\mathbb S} R_{\mathbb S}\widetilde{\mathbb F}L_{\mathbb S}\longrightarrow R_{\mathbb S}\widetilde{\mathbb F}L_{\mathbb S}=\mathbb FR_{\mathbb S}L_{\mathbb S}=\mathbb F\mathbb S
\end{equation} where the morphism $R_{\mathbb S}L_{\mathbb S} R_{\mathbb S}\widetilde{\mathbb F}L_{\mathbb S}\longrightarrow R_{\mathbb S}\widetilde{\mathbb F}L_{\mathbb S}$ in 
\eqref{dist64} is induced by the counit $L_{\mathbb S}R_{\mathbb S}\longrightarrow id$ of the adjunction $(L_{\mathbb S},R_{\mathbb S})$. 

\smallskip
Now suppose that $\tau:\mathbb F\longrightarrow \mathbb F'$ is a natural transformation between endofunctors on $\mathcal C$. Suppose that $\mathbb F$ (resp.  $\mathbb F'$)  lifts to
an endofunctor $\widetilde{\mathbb F}$ (resp. $\widetilde{\mathbb F}'$) on $EM_{\mathbb S}$, corresponding to a distributive law $\lambda: \mathbb S\mathbb F\longrightarrow \mathbb F\mathbb S$ (resp. $\lambda': \mathbb S\mathbb F'\longrightarrow \mathbb F'\mathbb S$). The natural transformation $\tau$ is said to lift to $EM_{\mathbb S}$ (see 
\cite[Definition 3.7]{Tan}) if for each 
$(M,\pi_M)\in EM_{\mathbb S}$, the morphism $\tau(M):\mathbb F(M)\longrightarrow \mathbb F'(M)$ gives a morphism $\widetilde{\tau}(M,\pi_M):\widetilde{\mathbb F}(M,\pi_M)\longrightarrow \widetilde{\mathbb F}'(M,\pi_M)$ 
in $EM_{\mathbb S}$. In other words, the following diagram must commute
\begin{equation}\label{dst65c}
\begin{CD}
\mathbb S\mathbb F(M) @>\lambda(M)>> \mathbb F\mathbb S(M) @>\mathbb F(\pi_M)>> \mathbb F(M)\\
@V\mathbb S(\tau(M))VV @. @VV\tau(M)V \\
\mathbb S\mathbb F'(M) @>\lambda'(M)>>  \mathbb F'\mathbb S(M)  @>\mathbb F'(\pi_M)>> \mathbb F'(M)\\
\end{CD}
\end{equation} Additionally, we know from \cite[Proposition 3.13]{Tan} that such a lifting of $\tau$ exists if and only if the following diagram commutes 
\begin{equation}\label{dst66c}
\begin{CD}
\mathbb S\mathbb F  @>\lambda >> \mathbb F\mathbb S  \\
@V\mathbb S\tau VV @VV\tau\mathbb SV \\
\mathbb S\mathbb F' @>\lambda'>>  \mathbb F'\mathbb S \\
\end{CD}
\end{equation} We need one more convention. If $\mathbb F_1$, ..., $\mathbb F_k$ is a family of endofunctors, along with distributive laws
$\{\lambda_i:\mathbb S\mathbb F_i\longrightarrow \mathbb F_i\mathbb S\}_{1\leq i\leq k}$, we write 
\begin{equation}\label{distcomb7}
(\lambda_1,...,\lambda_k):\mathbb S(\mathbb F_1...\mathbb F_k) \xrightarrow{\lambda_1\mathbb F_2...\mathbb F_k}\mathbb F_1\mathbb S\mathbb F_2....\mathbb F_k\xrightarrow{ \mathbb F_1\lambda_2
\mathbb F_3...\mathbb F_k}\dots \xrightarrow{\mathbb F_1...\mathbb F_{k-1}\lambda_k}(\mathbb F_1...\mathbb F_k)\mathbb S
\end{equation}  for the composition of the distributive laws $\lambda_1$, ... ,$\lambda_k$.

\begin{defn}
\label{D6.1dist}  Let $(\mathbb S,\theta_{\mathbb S},\iota_{\mathbb S})$ be a monad and let $(\mathbb U,\Theta)$ be an $A_\infty$-monad on $\mathcal C$. A distributive law of the monad $(\mathbb S,\theta_{\mathbb S},\iota_{\mathbb S})$ over the $A_\infty$-monad $(\mathbb U,\Theta)$ is given by a family $\Lambda=\{\lambda_n:
\mathbb S\mathbb U_n\longrightarrow \mathbb U_n\mathbb S\}_{n\in \mathbb Z}$ such that

\smallskip
(a) Each $\lambda_n:
\mathbb S\mathbb U_n\longrightarrow \mathbb U_n\mathbb S$ is a distributive law for the monad $\mathbb S$ over the endofunctor $\mathbb U_n$.

\smallskip
(b) For each $\bar{n}=(n_1,...,n_k)\in \mathbb Z^k$, $k\geq 1$, the following diagram commutes
\begin{equation}\label{cd65dsw}
\begin{CD}
\mathbb S\mathbb U_{\bar{n}}=\mathbb S\mathbb U_{n_1}...\mathbb U_{n_k} @>\lambda_{\bar{n}}=(\lambda_{n_1},...,\lambda_{n_k})>> \mathbb U_{n_1}...\mathbb U_{n_k} \mathbb S=\mathbb U_{\bar{n}}\mathbb S\\
@V\mathbb S\theta_k(\bar{n})VV @VV\theta_k(\bar{n})\mathbb SV\\
\mathbb S\mathbb U_{\Sigma\bar{n}+(2-k)} @>\lambda_{\Sigma\bar{n}+(2-k)}>> \mathbb U_{\Sigma\bar{n}+(2-k)}\mathbb S\\
\end{CD}
\end{equation}
\end{defn}

\begin{lem}\label{L6.2wk}
Let $\Lambda=\{\lambda_n\}_{n\in \mathbb Z}$ be a distributive law of the monad $\mathbb S$ over the $A_\infty$-monad $(\mathbb U,\Theta)$. Then, $(\mathbb U,\Theta)$ lifts to an
$A_\infty$-monad $(\widetilde{\mathbb U},\widetilde{\Theta})$ on the category $EM_{\mathbb S}$.
\end{lem}

\begin{proof}
Since each  $\lambda_n:
\mathbb S\mathbb U_n\longrightarrow \mathbb U_n\mathbb S$ is a distributive law, it follows from \eqref{c62cdj} and \eqref{dist6tf1} that $\mathbb U_n$ lifts to
$\widetilde{\mathbb U}_n:EM_{\mathbb S}\longrightarrow EM_{\mathbb S}$. Further, since \eqref{cd65dsw} commutes, it follows from the criterion described in 
\eqref{dst66c} that the natural transformations $\theta_k(\bar{n})$ lift to $\widetilde{\theta}_k(\bar{n}): \widetilde{\mathbb U}_{\bar{n}}=\widetilde{\mathbb U}_{n_1}...\widetilde{\mathbb U}_{n_k}
\longrightarrow \widetilde{\mathbb U}_{\Sigma\bar{n}+(2-k)} $. Accordingly, for any $(M,\pi_M)\in EM_{\mathbb S}$ and $\bar{z}\in 
\mathbb Z^N$, $N\geq 1$, it follows from \eqref{2.1e} that the morphism in 
$\mathcal C$ underlying the sum
\begin{equation}
\label{2.1ex6jq}
\sum (-1)^{p+qr+q\Sigma\bar{n}}\widetilde{\theta}_{p+1+r}(\bar{n},\Sigma\bar{n}'+(2-q),\bar{n}'')(\widetilde{\mathbb U}_{\bar{n}}\ast 
\widetilde{\theta}_q(\bar{n}')\ast \widetilde{\mathbb U}_{\bar{n}''}):\widetilde{\mathbb U}_{\bar{n}}\circ \widetilde{\mathbb U}_{\bar{n}'}\circ 
\widetilde{\mathbb U}_{\bar{n}''}(M,\pi_M)\longrightarrow \widetilde{\mathbb U}_{\Sigma \bar{n}+\Sigma\bar{n}'+\Sigma\bar{n}''+(3-N)}(M,\pi_M)
\end{equation} is zero,  where the sum runs over partitions $\bar{z}=(\bar{n}, \bar{n}',\bar{n}'')$ with $|\bar{n}|=p$, $|\bar{n}'|=q$, $|\bar{n}''|=r$. 
\end{proof}

\begin{Thm}\label{T6.3qz}
Let  $(\mathbb U,\Theta)$ be an $A_\infty$-monad and let $(\mathbb S,\theta_{\mathbb S},\iota_{\mathbb S})$ be a monad on $\mathcal C$. Then, there is a one one correspondence between the following

\smallskip
(a) Liftings of the $A_\infty$-monad $(\mathbb U,\Theta)$ to an $A_\infty$-monad $(\widetilde{\mathbb U},\widetilde{\Theta})$ on the category $EM_{\mathbb S}$.

\smallskip
(b) Distributive laws  of the monad $\mathbb S$ over the $A_\infty$-monad $(\mathbb U,\Theta)$.
\end{Thm} 

\begin{proof}
We know from Lemma \ref{L6.2wk} that a distributive law  of the monad $\mathbb S$ over the $A_\infty$-monad $(\mathbb U,\Theta)$ gives rise to a lifting of the $A_\infty$-monad $(
\mathbb U,\Theta)$ to $EM_{\mathbb S}$. On the other hand, let $(\widetilde{\mathbb U},\widetilde{\Theta})$  be a lifting of the $A_\infty$-monad 
$(\mathbb U,\Theta)$ to $EM_{\mathbb S}$. In particular, the lifting $\widetilde{\mathbb U}_n$ of the endofunctor 
$\mathbb U_n$ leads to a distributive law $\lambda_n:\mathbb S\mathbb U_n\longrightarrow \mathbb U_n\mathbb S$. Also, for any 
$k\geq 1$ and $\bar{n}\in \mathbb Z^k$, we know that the natural transformation $\theta_k(n):\mathbb U_{\bar{n}}\longrightarrow \mathbb U_{\Sigma\bar{n}+(2-k)}$ lifts 
to $EM_{\mathbb S}$. It now follows from \eqref{dst66c} that we must have $(\theta_k(\bar{n})\mathbb S)\circ (\lambda_{n_1},...,\lambda_{n_k})=\lambda_{\Sigma\bar{n}+(2-k)}
\circ (\mathbb S\theta_k(\bar{n}))$. Hence, the collection $\Lambda=\{\lambda_n:
\mathbb S\mathbb U_n\longrightarrow \mathbb U_n\mathbb S\}_{n\in \mathbb Z}$ gives a distributive law of the monad $(\mathbb S,\theta_{\mathbb S},\iota_{\mathbb S})$ over the $A_\infty$-monad $(\mathbb U,\Theta)$ in the sense of Definition \ref{D6.1dist}. Since there is a one one correspondence between distributive laws and liftings of an endofunctor
to $EM_{\mathbb S}$, it follows that these two associations are inverse to each other. 
\end{proof}

We  recall (see, for instance,  \cite[Proposition 3.30]{Tan}) that a distributive law of $\mathbb S$ over a monad $\mathbb T$ gives rise to a new monad $\mathbb T\mathbb S$. We will now study something similar for
$A_\infty$-monads. 

\begin{thm}\label{P6.4rc}
Let $\Lambda=\{\lambda_n:
\mathbb S\mathbb U_n\longrightarrow \mathbb U_n\mathbb S\}_{n\in \mathbb Z}$ be a distributive law of the monad $(\mathbb S,\theta_{\mathbb S},\iota_{\mathbb S})$ over the $A_\infty$-monad $(\mathbb U,\Theta)$. Then, the pair $(\mathbb U^{(\mathbb S)},\Theta^{(\mathbb S)})$ defined by setting $\mathbb U^{(\mathbb S)}_n:=\mathbb U_n\mathbb S$ and
$\theta_k^{(\mathbb S)}(\bar{n}): \mathbb U^{(\mathbb S)}_{\bar{n}}\longrightarrow \mathbb U^{(\mathbb S)}_{\Sigma\bar{n}+(2-k)}$ as follows
\begin{equation}\label{e6.10the}
\begin{array}{c}
\xymatrix{
 \mathbb U^{(\mathbb S)}_{\bar{n}}=(\mathbb U_{n_1}\mathbb S)...(\mathbb U_{n_k}\mathbb S)\ar[d]^{(\mathbb U_{n_1}\mathbb S)...(\mathbb U_{n_{k-2}}\mathbb S)\mathbb U_{n_{k-1}}\lambda_{n_k}\mathbb S} \\ \vdots \ar[d]^{\mathbb U_{n_1}(\lambda_{n_2},...,\lambda_{n_k})\mathbb S^{k-1}}\\ \mathbb U_{\bar{n}}\mathbb S^k =\mathbb U_{n_1}...\mathbb U_{n_k}\mathbb S^k
\ar[d]^{\theta_k(\bar{n})\theta_{\mathbb S}^k}\\ \mathbb U_{\Sigma
\bar{n}+(2-k)}\mathbb S=\mathbb U^{(\mathbb S)}_{\Sigma\bar{n}+(2-k)}\\
}
\end{array}
\end{equation}gives an $A_\infty$-monad on $\mathcal C$, where $\theta_{\mathbb S}^k:\mathbb S^k\longrightarrow \mathbb S$ is obtained by iterating the multiplication $\theta_{\mathbb S}:
\mathbb S^2\longrightarrow \mathbb S$ on the monad
$S$.
\end{thm}

\begin{proof}
We fix any $\bar{z}\in \mathbb Z^N$ and any partition $\bar{z}=(\bar{n},\bar{n}',\bar{n}'')$ with  $|\bar{n}|=p$, $|\bar{n}'|=q$, $|\bar{n}''|=r$. Since the distributive laws are well behaved both with respect to the multiplication $\theta_{\mathbb S}$ on the monad $\mathbb S$ as in \eqref{dist6tf1} and the structure maps $\theta_k(\bar{n})$ of the $A_\infty$-monad $
(\mathbb U,\Theta)$  as in \eqref{cd65dsw}, we note that the following diagram commutes
\begin{equation}\label{611mp}
\begin{array}{c}
\xymatrix{
\mathbb U^{(\mathbb S)}_{\bar{n}}\circ \mathbb U^{(\mathbb S)}_{\bar{n}'}\circ \mathbb U^{(\mathbb S)}_{\bar{n}''} \ar[rrrrr]^{} \ar[dd]_{\theta^{(\mathbb S)}_{p+1+r}(\bar{n},\Sigma\bar{n}'+(2-q),\bar{n}'')(\mathbb U^{(\mathbb S)}_{\bar{n}}\ast \theta_q^{(\mathbb S)}(\bar{n}')\ast \mathbb U^{(\mathbb S)}_{\bar{n}''})} & && && \mathbb U_{\bar{n}}\circ \mathbb U_{\bar{n}'}\circ \mathbb U_{\bar{n}''}\circ \mathbb S^N \ar[dd]^{\mathbb U_{\bar{n}}\mathbb U_{\bar{n}'}\mathbb U_{\bar{n}''}\theta^N_{\mathbb S}}\\ \\
 \mathbb U^{(\mathbb S)}_{\Sigma \bar{n}+\Sigma\bar{n}'+\Sigma\bar{n}''+(3-N)}= \mathbb U_{\Sigma \bar{n}+\Sigma\bar{n}'+\Sigma\bar{n}''+(3-N)}\mathbb S& &&& &\ar[lllll]_{\qquad\qquad\theta_{p+1+r}(\bar{n},\Sigma\bar{n}'+(2-q),\bar{n}'')(\mathbb U_{\bar{n}}\ast \theta_q(\bar{n}')\ast \mathbb U_{\bar{n}''})\mathbb S}  \mathbb U_{\bar{n}}\circ \mathbb U_{\bar{n}'}\circ \mathbb U_{\bar{n}''}\circ \mathbb S \\
}
\end{array}
\end{equation} Here the top horizontal arrow in \eqref{611mp} is obtained by repeatedly applying the distributive laws. It now follows from \eqref{2.1e} that
\begin{equation}
\sum (-1)^{p+qr+q\Sigma\bar{n}}\theta^{(\mathbb S)}_{p+1+r}(\bar{n},\Sigma\bar{n}'+(2-q),\bar{n}'')(\mathbb U^{(\mathbb S)}_{\bar{n}}\ast \theta^{(\mathbb S)}_q(\bar{n}')\ast \mathbb U^{(\mathbb S)}_{\bar{n}''}) =0
\end{equation} where the sum runs over partitions $\bar{z}=(\bar{n}, \bar{n}',\bar{n}'')$ with $|\bar{n}|=p$, $|\bar{n}'|=q$, $|\bar{n}''|=r$. 
\end{proof}

Let $\Lambda=\{\lambda_n:
\mathbb S\mathbb U_n\longrightarrow \mathbb U_n\mathbb S\}_{n\in \mathbb Z}$ be a distributive law of the monad $(\mathbb S,\theta_{\mathbb S},\iota_{\mathbb S})$ over the $A_\infty$-monad $(\mathbb U,\Theta)$. By Theorem \ref{T6.3qz}, $(\mathbb U,\Theta)$ lifts to an $A_\infty$-monad $(\widetilde{\mathbb U},\widetilde{\Theta})$ on the category $EM_{\mathbb S}$. Accordingly, we have the category $EM_{(\widetilde{\mathbb U},\widetilde{\Theta})}$ of $(\widetilde{\mathbb U},\widetilde{\Theta})$-modules in the sense of
Definition \ref{D2.4}. On the other hand, we have  $(\mathbb U^{(\mathbb S)},\Theta^{(\mathbb S)})$  from Proposition \ref{P6.4rc}, which is an $A_\infty$-monad on $\mathcal C$. 
The category of  $(\mathbb U^{(\mathbb S)},\Theta^{(\mathbb S)})$-modules is denoted by $EM_{ (\mathbb U^{(\mathbb S)},\Theta^{(\mathbb S)}) }$.

\begin{thm}\label{P6.5xk}
Let $\Lambda=\{\lambda_n:
\mathbb S\mathbb U_n\longrightarrow \mathbb U_n\mathbb S\}_{n\in \mathbb Z}$ be a distributive law of the monad $(\mathbb S,\theta_{\mathbb S},\iota_{\mathbb S})$ over the $A_\infty$-monad $(\mathbb U,\Theta)$. Then, there is a canonical functor $EM_{(\widetilde{\mathbb U},\widetilde{\Theta})}\longrightarrow EM_{ (\mathbb U^{(\mathbb S)},\Theta^{(\mathbb S)}) }$.
\end{thm}

\begin{proof} By definition, a  $(\widetilde{\mathbb U},\widetilde{\Theta})$-module consists of a collection $\{(M_n,\pi_n:\mathbb SM_n\longrightarrow M_n)\}_{n\in \mathbb Z}$ of objects of $EM_{\mathbb S}$ 
and a  collection of morphisms in  $EM_{\mathbb S}$:
\begin{equation}\label{612qef}
\{\mbox{$\widetilde{\pi}_k(\bar{n},l):\widetilde{\mathbb U}_{\bar{n}}(M_{l},\pi_l)\longrightarrow (M_{\Sigma\bar{n}+l+(2-k)},\pi_{\Sigma\bar{n}+l+(2-k)})$ $\vert$ $k\geq 1$, $\bar{n}=(n_1,...,n_{k-1})
\in \mathbb Z^{k-1}$, $l\in \mathbb Z$ } \}
\end{equation} Since $\widetilde{\mathbb U}_{\bar{n}}$ is a lifting of $\mathbb U_{\bar{n}}$ to $EM_{\mathbb S}$, there is a morphism $\mathbb U_{\bar{n}}M_l\longrightarrow M_{\Sigma\bar{n}+l+(2-k)}$ in $\mathcal C$ underlying $\widetilde{\pi}_k(\bar{n},l)$. We continue to denote this morphism by $\widetilde{\pi}_k(\bar{n},l):\mathbb U_{\bar{n}}M_l\longrightarrow M_{\Sigma\bar{n}+l+(2-k)}$. We also note the morphism
\begin{equation}
\pi_l\circ ...\circ \mathbb S^{k-3}(\pi_l)\circ \mathbb S^{k-2}(\pi_l):\mathbb S^{k-1}(M_l)\xrightarrow{ \mathbb S^{k-2}(\pi_l)} \mathbb S^{k-2}(M_l)\xrightarrow{ \mathbb S^{k-3}(\pi_l)}
\mathbb S^{k-3}(M_l)\xrightarrow{\mathbb S^{k-4}(\pi_l)} ...\xrightarrow{\pi_l} M_l
\end{equation} For $k\geq 1$, $\bar{n}=(n_1,...,n_{k-1})
\in \mathbb Z^{k-1}$, $l\in \mathbb Z$, we now consider the composition
\begin{equation}\label{614rty}
{\pi}^{(\mathbb S)}_k(\bar{n},l):\mathbb U^{(\mathbb S)}_{\bar{n}}M_l\xrightarrow{\qquad\qquad\qquad} \mathbb U_{\bar{n}}\mathbb S^{k-1}M_l\xrightarrow{\pi_l\circ ...\circ \mathbb S^{k-3}(\pi_l)\circ \mathbb S^{k-2}(\pi_l)}\mathbb U_{\bar{n}}M_l\xrightarrow{\widetilde{\pi}_k(\bar{n},l)} M_{\Sigma\bar{n}+l+(2-k)}
\end{equation} where the first morphism in  \eqref{614rty} is obtained by repeatedly applying the distributive law $\Lambda=\{\lambda_n:
\mathbb S\mathbb U_n\longrightarrow \mathbb U_n\mathbb S\}_{n\in \mathbb Z}$ as in top horizontal arrow in \eqref{611mp}. As in the proof of Proposition \ref{P6.4rc}, it may now be verified that the collection $\{M_n\}_{n\in \mathbb Z}$ along with the morphisms ${\pi}^{(\mathbb S)}_k(\bar{n},l)$ satisfies the relations for being a $(\mathbb U^{(\mathbb S)},\Theta^{(\mathbb S)})$-module in the sense of Definition \ref{D2.4}. 

\end{proof}

We conclude this section with the following result. By a distributive law of the $A_\infty$-monad $(\mathbb U,\Theta)$ over the monad $(\mathbb S,\theta_{\mathbb S},\iota_{\mathbb S})$, we mean a collection of natural transformations $\Lambda'=\{\lambda'_n:
\mathbb U_n\mathbb S\longrightarrow \mathbb S \mathbb U_n\}_{n\in \mathbb Z}$ such that

\smallskip
(a) The following two  diagrams commute for each $n\in \mathbb Z$
\begin{equation}
\label{dist6tf2x}
\begin{array}{ccc}
\xymatrix{& \mathbb U_n \ar[dl]_{\mathbb U_n\iota_{\mathbb S}} \ar[dr]^{\iota_{\mathbb S}\mathbb U_n}&\\
\mathbb U_n \mathbb S \ar[rr]_{\lambda'_n} &&\mathbb S \mathbb U_n\\
}& \qquad &\xymatrix{
\mathbb U_n\mathbb S^2 \ar[d]_{\mathbb U_n\theta_{\mathbb S}}\ar[r]^{\lambda'_n\mathbb S}&   \mathbb S\mathbb U_n\mathbb S  \ar[r]^{\mathbb S\lambda'_n}&\mathbb S^2 \mathbb U_n
\ar[d]^{\theta_{\mathbb S}\mathbb U_n}\\
\mathbb U_n\mathbb S\ar[rr]_{\lambda'_n}&&\mathbb S\mathbb U_n\\} \\
\end{array}
\end{equation}
\smallskip
(b)  For each $\bar{n}=(n_1,...,n_k)\in \mathbb Z^k$, $k\geq 1$, the following diagram commutes
\begin{equation}\label{cd65dswg}
\begin{array}{c}
\xymatrix{
\mathbb U_{\bar{n}}\mathbb S=\mathbb U_{n_1}...\mathbb U_{n_k}\mathbb S\ar[d]_{\theta_k(\bar{n})\mathbb S}
\ar[rr]^{\mathbb U_{n_1}...\mathbb U_{n_{k-1}}\lambda'_{n_k}}&& \mathbb U_{n_1}...\mathbb U_{n_{k-1}}\mathbb S\mathbb U_{n_k}
\ar[rr]^{\qquad\mathbb U_{n_1}...\mathbb U_{n_{k-2}}\lambda'_{n_{k-1}}\mathbb U_{n_k}} &&\dots  \ar[rr]^{\lambda_{n_1}\mathbb U_{n_2}...\mathbb U_{n_k}\qquad\qquad}& &\mathbb S\mathbb U_{n_1}...\mathbb U_{n_k}=
\mathbb S\mathbb U_{\bar{n}}\ar[d]^{\mathbb S\theta_k(\bar{n})}\\
\mathbb U_{\Sigma\bar{n}+(2-k)}\mathbb S\ar[rrrrrr]^{\lambda'_{\Sigma\bar{n}+(2-k)}} &&&&&& \mathbb S\mathbb U_{\Sigma\bar{n}+(2-k)}
}
\end{array}
\end{equation} We now recall (see, for instance, \cite[$\S$ 5.2]{Jac}) that the Kleisli category $Kl(\mathbb S)$ of the monad $(\mathbb S,\theta_{\mathbb S},\iota_{\mathbb S})$ is the category
whose objects are the same as those of $\mathcal C$, but a morphism $M\longrightarrow M'$ in $Kl(\mathbb S)$ corresponds to a morphism
$M\longrightarrow \mathbb S(M')$ in $\mathcal C$. The composition of $\alpha: M\longrightarrow \mathbb S(M')$ with
$\beta:M'\longrightarrow \mathbb S(M'')$ in $Kl(\mathbb S)$ is given by
\begin{equation}
M \xrightarrow{\alpha}\mathbb S(M') \xrightarrow{\mathbb S(\beta)}\mathbb S^2(M'')\xrightarrow{\theta_{\mathbb S}(M'')}\mathbb S(M'')
\end{equation} There is a canonical functor $J_{\mathbb S}:\mathcal C\longrightarrow Kl(\mathbb S)$   that is identity on objects and which takes the morphism $
\alpha:M\longrightarrow M'$ in $\mathcal C$ to $M\xrightarrow{\alpha}M'\xrightarrow{\iota_{\mathbb S}(M')}
\mathbb S(M')$ in $Kl(\mathbb S)$.  From \cite[Proposition 5.2.5]{Jac}, we know that a natural transformation $\lambda'_n:
\mathbb U_n\mathbb S\longrightarrow \mathbb S \mathbb U_n$ satisfying the conditions in \eqref{dist6tf2x} corresponds to an extension 
$\widehat{\mathbb U}_n$ of the endofunctor $\mathbb U_n$ to the Kleisli category $Kl(\mathbb S)$, i.e., $\widehat{\mathbb U}_n\circ J_{\mathbb S}
=J_{\mathbb S}\circ \mathbb U_n$. From \cite[Theorem 4.16]{Tan}, we know that the condition in \eqref{cd65dswg} ensures that the natural transformations 
$\theta_k(\bar{n}):\mathbb U_{\bar{n}}\longrightarrow \mathbb U_{\Sigma\bar{n}+(2-k)}$ extend to natural transformations $\widehat{\theta}_k(\bar{n}):\widehat{\mathbb U}_{\bar{n}}\longrightarrow \widehat{\mathbb U}_{\Sigma\bar{n}+(2-k)}$.  Proceeding as in the proof of Theorem \ref{T6.3qz}, we can now prove the following result.

\begin{Thm}\label{T6.7qz}
Let  $(\mathbb U,\Theta)$ be an $A_\infty$-monad and let $(\mathbb S,\theta_{\mathbb S},\iota_{\mathbb S})$ be a monad on $\mathcal C$. Then, there is a one one correspondence between the following

\smallskip
(a) Extensions of the $A_\infty$-monad $(\mathbb U,\Theta)$ to an $A_\infty$-monad $(\widehat{\mathbb U},\widehat{\Theta})$ on the category $Kl({\mathbb S})$.

\smallskip
(b) Distributive laws  of the $A_\infty$-monad $(\mathbb U,\Theta)$ over the monad $\mathbb S$.
\end{Thm}

\section{$A_\infty$-(co)algebras and $A_\infty$-(co)monads}

In this section, we will show how to give examples of $A_\infty$-monads and $A_\infty$-comonads on any $K$-linear Grothendieck category.  For basic definitions such as $A_\infty$-algebras, $A_\infty$-coalgebras, $A_\infty$-modules over $A_\infty$-algebras and $A_\infty$-comodules over $A_\infty$-coalgebras, we refer the reader to the appendix in Section \ref{Appendix}. 

\smallskip
Let $\mathcal C$ be a $K$-linear Grothendieck category.    If $R$ is a $K$-algebra, an $R$-object in $\mathcal C$ consists of $M\in \mathcal C$ along with a morphism $\xi_M:R\longrightarrow \mathcal C(M,M)$ of $K$-algebras. The category $\mathcal C_R$ of $R$-objects in $\mathcal C$ is a classical construction of Popescu \cite[p 108]{Pop}. This abstract module theory  has also been developed extensively by Artin and Zhang \cite{AZ}, who studied Hilbert Basis Theorem, completions, localizations and Grothendieck's theory of flat descent in this context. In particular, if $\mathcal C=Mod\text{-}S$, the category of right modules over a $K$-algebra $S$, then $\mathcal C_R=Mod\text{-}(S\otimes_KR)$.  Accordingly,  the abstract module category $\mathcal C_R$ is often understood as a noncommutative base change of $R$ by means of the category $\mathcal C$ (see Lowen and Van den Bergh \cite{LoV}). If $V$ is a right $R$-module, then there is a pair of adjoint functors (see \cite[$\S$ B3, B4]{AZ}) 
\begin{equation}\label{scu5.1}
\_\_\otimes V : \mathcal C\longrightarrow \mathcal C_R\qquad \underline{Hom}_R(V,\_\_):\mathcal C_R\longrightarrow \mathcal C
\end{equation} Therefore, for $\mathcal C=\mathcal C_K$, an $R$-module object in $\mathcal C$ is determined by $M\in \mathcal C$ and a structure map
$M\otimes R\longrightarrow M$ in $\mathcal C$ satisfying the usual associativity and unit conditions. We observe that $\_\_\otimes R:\mathcal C\longrightarrow \mathcal C$ is an ordinary monad on $\mathcal C$. Similarly, if $C'$ is a $K$-coalgebra, a $C'$-comodule object in $\mathcal C$ is determined by $P\in \mathcal C$ and a structure map $P\longrightarrow P\otimes C'$ satisfying coassociativity and counit conditions (see Brzezi\'{n}ski and Wisbauer \cite[$\S$ 39.1]{BWis}). Again, $\_\_\otimes C'$ is an ordinary comonad on $\mathcal C$. This suggests how we can create $A_\infty$-monads and $A_\infty$-comonads on $\mathcal C$.

\smallskip
We  let $Vect^{\mathbb Z}$ be the category of $\mathbb Z$-graded vector spaces. We know that $Vect^{\mathbb Z}$ is a closed symmetric monoidal category, with internal hom
$[W',W]$ of two objects given by
\begin{equation}\label{6.1tj}
[W',W]_p:=\underset{k\in\mathbb Z}{\prod}\textrm{ }Vect(W'_k,W_{k+p}) \qquad p\in \mathbb Z
\end{equation}   Additionally, we consider the endofunctor $T:Vect^{\mathbb Z}\longrightarrow Vect^{\mathbb Z}$ that inverts the grading, i.e., $(TW)_k:=W_{-k}$ for any $k\in \mathbb Z$ and $W\in
Vect^{\mathbb Z}$. 

\smallskip
We let  $(\mathbb F,\theta_{\mathbb F})$ be a non-unital monad on $\mathcal C$. In other words,  we have an endofunctor $\mathbb F\in End(\mathcal C)$ along with a ``multiplication'' 
$\theta_{\mathbb F}:\mathbb F\circ \mathbb F\longrightarrow \mathbb F$ satisfying the usual associativity conditions. We note that any monad on $\mathcal C$ is equipped in particular with the structure of a non-unital monad. Similarly, a non-counital comonad $(\mathbb G,\Delta_{\mathbb G})$ consists of an endofunctor $\mathbb G\in End(\mathcal C)$ along with a ``comultiplication'' 
$\Delta_{\mathbb G}:\mathbb G\longrightarrow \mathbb G\circ \mathbb G$ satisfying the usual coassociativity conditions.  The following result can be used to construct several examples of $A_\infty$-monads on $\mathcal C$. 

\begin{thm}\label{scuP5.1}
Let   $A=\underset{n\in \mathbb Z}{\bigoplus}A_n$ be an $A_\infty$-algebra and let $(\mathbb F,\theta_{\mathbb F})$ be a non-unital monad on $\mathcal C$.  Suppose that $\mathbb F$ preserves direct sums. Then,  the collection of functors
\begin{equation}\label{scue3.1e}
\mathbb U^{A,\mathbb F}_{\mathcal C}=\{\mathbb U_{\mathcal C,n}^{A,\mathbb F}:=A_n\otimes \mathbb F( \_\_):\mathcal C\longrightarrow \mathcal C\}_{n\in \mathbb Z}
\end{equation} is an $A_\infty$-monad on $\mathcal C$. In particular, the collection of functors $\mathbb U^A_{\mathcal C}=\{\mathbb U_{\mathcal C,n}^A:=A_n\otimes \_\_:\mathcal C\longrightarrow \mathcal C\}_{n\in \mathbb Z}$  is an $A_\infty$-monad on $\mathcal C$.
\end{thm}

\begin{proof}  Let $m_k:A^{\otimes k}\longrightarrow A$ of degree $(2-k)$ for $k\geq 1$ be the structure maps for $A$ as an $A_\infty$-algebra as in
Definition \ref{D3.1}. Since $\mathbb F$ preserves direct sums, we note that for $k\geq 1$ and $\bar{n}=(n_1,...,n_k)\in \mathbb Z^k$, we have
\begin{equation}
\mathbb U^{A,\mathbb F}_{
\mathcal C, n_1}\circ ...\circ \mathbb U^{A,\mathbb F}_{\mathcal C, n_k}=A_{n_1}\otimes ... \otimes A_{n_k}\otimes \mathbb F^k(\_\_)
\end{equation} We can now define natural transformations as in \eqref{2.1trf}
\begin{equation} \Theta_{\mathcal C}^{A,\mathbb F}=\{\mbox{$\theta^{A,\mathbb F}_{\mathcal C,k}(\bar{n}):\mathbb U^{A,\mathbb F}_{\mathcal C,\bar{n}}=\mathbb U^{A,\mathbb F}_{
\mathcal C, n_1}\circ ...\circ \mathbb U^{A,\mathbb F}_{\mathcal C, n_k}\longrightarrow \mathbb U^{A,\mathbb F}_{\mathcal C, n_1+...+n_k+(2-k)}=\mathbb U^{A,\mathbb F}_{\mathcal C, \Sigma \bar{n}+(2-k)}$} \}
\end{equation} induced by the structure maps 
$
A_{n_1}\otimes ... \otimes A_{n_k}\longrightarrow A_{n_1+...+n_k+(2-k)}
$ of the $A_\infty$-algebra $A$ and the $k$-fold multiplication $\theta_{\mathbb F}^{k-1}:\mathbb F^k\longrightarrow \mathbb F$ on the monad $\mathbb F$. The relations in \eqref{2.1e} now follow directly from the relations satisfied by the structure maps of $A$ (see  \eqref{Ap4.1}).

\end{proof}

We will say that an $A_\infty$-algebra $A$ is even if the structure maps $m_k:A^{\otimes k}\longrightarrow A$ vanish whenever $k$ is odd. For instance, this condition is satisfied by  $A_{2,p}$-algebras (for $p$ even), which have been extensively studied in the literature (see, for instance, \cite{DoVa}, \cite{HeLu}, \cite{Kelx}, \cite{LPWZ}). An $A_{2,p}$-algebra is an $A_\infty$-algebra with all the operations $m_k$  vanishing except for $k=2$ and $k=p$. The $A_{2,p}$-algebras were introduced by He and Lu in \cite{HeLu} and studied further by Dotsenko and Vallette in \cite{DoVa}. If $A$ is a $p$-Koszul algebra, then the Ext-algebra 
$Ext^*(K_A,K_A)$ carries the structure of an $A_{2,p}$-algebra (see \cite{HeLu}). Here, $K_A$ refers to the $A$-module structure on $K$ obtained from the fact that $A_0=K$. 
The authors in \cite{HeLu} also describe a general procedure for constructing an $A_{2,p}$-algebra starting from any positively graded associative algebra. Other motivating examples for $A_{2,p}$-algebras, such as the $A_\infty$-Koszul dual of $K[t]/(t^p)$  for $p\geq 3$, appear in  Lu,
Palmieri, Wu, and Zhang \cite{LPWZ} and in Keller \cite[$\S$ 2.6]{Kelx}.

\begin{thm}\label{scuP5.2}
Let  $A=\underset{n\in \mathbb Z}{\bigoplus}A_n$ be an $A_\infty$-algebra that is even  and let $(\mathbb F,\theta_{\mathbb F})$ be a non-unital monad on $\mathcal C$.  Suppose that  $\mathbb F$ has a right adjoint $\mathbb G$.  Then, the collection of functors 
\begin{equation}\label{7scue3.3e}
\mathbb V^{A,\mathbb G}_{\mathcal C}=\{\mathbb V_{\mathcal C,n}^{A,\mathbb G}:=\underline{Hom}(A_n, \mathbb G(\_\_)):\mathcal C\longrightarrow \mathcal C\}_{n\in \mathbb Z}
\end{equation} is an   $A_\infty$-comonad on $\mathcal C$.  In particular, the collection of functors
\begin{equation}\label{scue3.3e}
\mathbb V^A_{\mathcal C}=\{\mathbb V_{\mathcal C,n}^A:=\underline{Hom}(A_n, \_\_):\mathcal C\longrightarrow \mathcal C\}_{n\in \mathbb Z}
\end{equation} is an   $A_\infty$-comonad on $\mathcal C$. Further,  we have isomorphisms of categories  $EM_{\mathbb U^{A,\mathbb F}_{\mathcal C}}^{e}\cong EM^{\mathbb V^{A,\mathbb G}_{\mathcal C}}_{e}$ and $\widetilde{EM}_{\mathbb U^{A,\mathbb F}_{\mathcal C}}^{eo}\cong \widetilde{EM}^{\mathbb V^{A,\mathbb G}_{\mathcal C}}_{eo}$ . 
\end{thm}

\begin{proof}
Since $\mathbb F$ is a left adjoint (and hence it preserves direct sums), it follows by \eqref{scu5.1} that the functors $\underline{Hom}(A_n, \mathbb G(\_\_)):\mathcal C\longrightarrow \mathcal C$ are right adjoint to the functors $A_n\otimes \mathbb F(\_\_)=\mathbb F(A_n\otimes \_\_):\mathcal C\longrightarrow \mathcal C$. The result now follows from Theorem \ref{T2.3}, Proposition \ref{P2.7} and Theorem \ref{T2.10}.
\end{proof}

The following result can be used to give several examples of $A_\infty$-comonads.

\begin{thm}\label{scuP3.4} Let $C=\underset{n\in \mathbb Z}{\bigoplus}C_n$ be an $A_\infty$-coalgebra with structure maps $\{w_k: C\longrightarrow C^{\otimes k}\}_{k\geq 1}$   and let $(\mathbb G,\Delta_{\mathbb G})$ be a non-counital comonad on $\mathcal C$.  Suppose that $\mathbb G$ preserves direct sums. Then, the collection of functors
\begin{equation}\label{scue3.1ezz}
\mathbb U^{TC,\mathbb G}_{\mathcal C}=\{\mathbb U_{\mathcal C,n}^{TC,\mathbb G}:=C_{-n}\otimes \mathbb G(\_\_):\mathcal C\longrightarrow \mathcal C\}_{n\in \mathbb Z}
\end{equation} can be equipped with the structure of an $A_\infty$-comonad on $\mathcal C$. In particular, $\mathbb U^{TC}_{\mathcal C}=\{\mathbb U_{\mathcal C,n}^{TC}:=C_{-n}\otimes \_\_:\mathcal C\longrightarrow \mathcal C\}_{n\in \mathbb Z}
$ is an $A_\infty$-comonad on $\mathcal C$. 

\smallskip If $C$ is even, i.e., $w_k=0$ whenever $k$ is odd and $\mathbb G$ has a right adjoint $\mathbb J$,   the collection of functors
\begin{equation}\label{scue3.3ez}
\mathbb V^{TC,\mathbb J}_{\mathcal C}=\{\mathbb V_{\mathcal C,n}^{TC,\mathbb J}:=\underline{Hom}(C_{-n}, \mathbb J(\_\_)):\mathcal C\longrightarrow \mathcal C\}_{n\in \mathbb Z}
\end{equation} is an   $A_\infty$-monad on $\mathcal C$. In particular,  if $C$ is even, $\mathbb V^{TC}_{\mathcal C}=\{\mathbb V_{\mathcal C,n}^{TC}:=\underline{Hom}(C_{-n}, \_\_):\mathcal C\longrightarrow \mathcal C\}_{n\in \mathbb Z}
$ is an   $A_\infty$-monad on $\mathcal C$. 
\end{thm}

\begin{proof}
In a manner similar to the proof of Proposition \ref{scuP5.1}, it is clear from the definition of  an $A_\infty$-coalgebra in  Definition \ref{ApD3.4} and the fact that 
$\mathbb G$ preserves direct sums that  $\mathbb U^{TC,\mathbb G}_{\mathcal C}$ is an $A_\infty$-comonad on $\mathcal C$. If $C$ is even, so is $\mathbb U^{TC,\mathbb G}_{\mathcal C}$. By \eqref{scu5.1}, the functors
$\{\underline{Hom}(C_{-n},\mathbb J(\_\_))=\underline{Hom}((TC)_n,\mathbb J(\_\_))\}_{n\in \mathbb Z}$ are right adjoint to the functors $\{\mathbb G((TC)_n\otimes \_\_)=C_{-n}\otimes \mathbb G(\_\_)\}_{n\in \mathbb Z}$ and the result now follows directly from Theorem \ref{T2.3}. 
\end{proof}

\smallskip
For the rest of this section, we suppose that $\mathcal C=Vect$, the category of vector spaces over the field $K$.  In this paper, for an  $A_\infty$-algebra $A$, the category of left $A_\infty$-modules will always be denoted by $_A\mathbf M$ (see Definition \ref{D3.3}). For an $A_\infty$-coalgebra $C$, the category of left $A_\infty$-comodules will be denoted by
$^C\mathbf M$ (see Definition \ref{D3.5}).

\begin{thm}\label{P3.2} Let $A=\underset{n\in \mathbb Z}{\bigoplus}A_n$ be a $\mathbb Z$-graded vector space. Consider the collection of functors
\begin{equation}\label{e3.1e}
\mathbb U^A=\{\mathbb U_n^A:=A_n\otimes \_\_:Vect\longrightarrow Vect\}_{n\in \mathbb Z}
\end{equation} Then, the family $\{\mathbb U_n^A\}$ can be equipped with the structure of an $A_\infty$-monad if and only if $A=\underset{n\in \mathbb Z}{\bigoplus}A_n$ can be equipped
with the structure of an $A_\infty$-algebra.
\end{thm}

\begin{proof}
If $A$ is an $A_\infty$-algebra, it follows from Proposition \ref{scuP5.1} that $\{\mathbb U_n^A\}$ can be equipped with the structure of an $A_\infty$-monad.  Conversely, suppose that the collection $\mathbb U^A$ as defined in \eqref{e3.1e} is an $A_\infty$-monad. Then, by applying the natural transformations $\theta_k^A(\bar{n}):\mathbb U_{\bar{n}}^A=\mathbb U_{n_1}^A\circ ...\circ \mathbb U_{n_k}^A\longrightarrow \mathbb U_{n_1+...+n_k+(2-k)}^A=\mathbb U_{\Sigma \bar{n}+(2-k)}^A$ to  the vector space $K$, we obtain morphisms $A_{n_1}\otimes ... \otimes A_{n_k}\longrightarrow A_{n_1+...+n_k+(2-k)}
$. The relations in \eqref{2.1e} now ensure that these satisfy the conditions for being an $A_\infty$-algebra.
\end{proof}

\begin{cor}\label{C3.3} Let $A=\underset{n\in \mathbb Z}{\bigoplus}A_n$ be an $A_\infty$-algebra. Then,  we have

\smallskip
(a) The category $EM_{\mathbb U_A}$ of modules over the $A_\infty$-monad $\mathbb U^A=\{\mathbb U_n^A:=A_n\otimes \_\_:Vect\longrightarrow Vect\}_{n\in \mathbb Z}$ is isomorphic to the category $_A\mathbf M$ of $A_\infty$-modules over $A$. 

\smallskip 
(b) Suppose that $\mathbb U^A$ is even.  Then, the collection of functors
\begin{equation}\label{e3.3e}
\mathbb V^A=\{\mathbb V_n^A:=Hom(A_n, \_\_):Vect\longrightarrow Vect\}_{n\in \mathbb Z}
\end{equation} is an   $A_\infty$-comonad.

\smallskip
(c) Suppose that $\mathbb U^A$ is even.  Then,  we have isomorphisms of categories  $EM_{\mathbb U_A}^{e}\cong EM^{\mathbb V_A}_{e}$ and $\widetilde{EM}_{\mathbb U_A}^{eo}\cong \widetilde{EM}^{\mathbb V_A}_{eo}$ . 
\end{cor} 
\begin{proof}
By inspecting the conditions in Definition \ref{D2.4}, we can see directly that $\mathbb U^A$-modules correspond to $A_\infty$-modules over $A$ in the sense of Definition \ref{D3.3}. This proves (a). Since the functors
$\{Hom(A_n,\_\_)\}_{n\in \mathbb Z}$ are right adjoint to the functors $\{A_n\otimes\_\_\}_{n\in \mathbb Z}$, the result of (b)  follows from Theorem \ref{T2.3}. The result of (c) is also clear from Proposition \ref{P2.7} and Theorem \ref{T2.10}.
\end{proof}

\begin{thm}\label{P5.4fd} Let $C=\underset{n\in \mathbb Z}{\bigoplus}C_n$ be an $A_\infty$-coalgebra. Then, there is a canonical functor 
\begin{equation}\label{5.6yw}
\mathscr I^C: {^C}\mathbf M\longrightarrow EM^{\mathbb U^{TC}}
\end{equation} that embeds $A_\infty$-comodules as a full subcategory of the Eilenberg-Moore category of comodules over the $A_\infty$-comonad $\mathbb U^{TC}$. 
\end{thm}

\begin{proof}
Let $P$ be an $A_\infty$-comodule over $C$ in the sense of Definition \ref{D3.5}, equipped with structure maps $\{w_k^P:P\longrightarrow C^{\otimes k-1}\otimes P\}_{k\geq 1}$. We set $Q=TP$, i.e., $Q_n:=P_{-n}$ for $n\in \mathbb Z$. For any $k\geq 1$, $\bar{n}\in \mathbb Z^{k-1}$, $l\in \mathbb Z$, we now have
\begin{equation}\label{5.7yy}
\begin{CD}
\rho_k(\bar{n},l): Q_{\Sigma\bar{n}+l+(2-k)}=P_{-\Sigma\bar{n}-l-(2-k)} @>w^P_{k,-\Sigma\bar{n}-l-(2-k)}>> (C^{\otimes k-1}\otimes P)_{-\Sigma\bar{n}-l}\\
@. @VVV \\  @.  (TC)_{n_1}\otimes ...
\otimes (TC)_{n_{k-1}}\otimes (TP)_l =\mathbb U^{TC}_{\bar{n}}Q_l\\
\end{CD}
\end{equation} where the second map in \eqref{5.7yy} is the canonical projection. Due to the relations in Definition \ref{D3.5} satisfied by the structure maps of $P$, it is clear that the collection $\rho_k(\bar{n},l)$ satisfies the relations in \eqref{l2.14} for $Q$ to be a $\mathbb U^{TC}$-comodule. As such, setting 
$\mathscr I^C(P)=Q=TP$ defines a functor $\mathscr I^C: {^C}\mathbf M\longrightarrow EM^{\mathbb U^{TC}}$. 

\smallskip
It remains to show that $\mathscr I^C$ determines a full embedding. For this, we consider $P'\in {^C}\mathbf M$ along with structure maps   $\{w_k^{P'}:P'\longrightarrow C^{\otimes k-1}\otimes P'\}_{k\geq 1}$ as well as $Q'=TP'=\mathscr I^C(P')\in EM^{\mathbb U^{TC}}$ with structure maps $\rho'_k(\bar{n},l): Q'_{\Sigma\bar{n}+l+(2-k)}\longrightarrow \mathbb U^{TC}_{\bar{n}}Q'_l$ determined as in \eqref{5.7yy}. Let $h=\{h_n\}_{n\in \mathbb Z}\in  EM^{\mathbb U^{TC}}(\mathscr I^C(P),\mathscr I^C(P'))$. For each fixed $T\in \mathbb Z$
and $k\geq 1$, we now observe that the outer square as well as the lower square in the following diagram commute
\begin{equation}\label{5.8zz}
\small
\begin{CD}
Q_{T+(2-k)}=P_{-T-(2-k)} @>h_{T+(2-k)}>> Q'_{T+(2-k)}=P'_{-T-(2-k)} \\ @Vw^P_{k,-T-(2-k)}VV  @VVw^{P'}_{k,-T-(2-k)}V  \\ \underset{(\bar{n},l)\in \mathbb Z(T,k)}{\bigoplus}(C_{-n_1}\otimes ...\otimes C_{-n_{k-1}}\otimes P_{-l}) @> \underset{(\bar{n},l)\in \mathbb Z(T,k)}{\bigoplus}(C_{-n_1}\otimes ...\otimes C_{-n_{k-1}}\otimes h_l)>> \underset{(\bar{n},l)\in \mathbb Z(T,k)}{\bigoplus}(C_{-n_1}\otimes ...\otimes C_{-n_{k-1}}\otimes P'_{-l}) \\
@VVV @VVV \\  \underset{(\bar{n},l)\in \mathbb Z(T,k)}{\prod}(C_{-n_1}\otimes ...\otimes C_{-n_{k-1}}\otimes P_{-l}) @> \underset{(\bar{n},l)\in \mathbb Z(T,k)}{\prod}(C_{-n_1}\otimes ...\otimes C_{-n_{k-1}}\otimes h_l)>>  \underset{(\bar{n},l)\in \mathbb Z(T,k)}{\prod}(C_{-n_1}\otimes ...\otimes C_{-n_{k-1}}\otimes P'_{-l}) \\
\end{CD}
\end{equation} Since the direct sums in \eqref{5.8zz} embed into the direct products, it follows that the upper square in \eqref{5.8zz} also commutes. We now see that $h=\{h_n\}_{n\in \mathbb Z}$ also determines a morphism in $^C\mathbf M$. 
\end{proof}

\section{Rational pairings between $A_\infty$-algebras and $A_\infty$-coalgebras}

If $W\in Vect^{\mathbb Z}$ is a graded vector space, we note that  by \eqref{6.1tj}, its graded dual $W^*=[W,K]$ is given by $W^*_n=Vect(W_{-n},K)$ for $n\in\mathbb Z$. If $C$ is an $A_\infty$-coalgebra, equipped with structure maps $\{w_k: C\longrightarrow C^{\otimes k}\}_{k\geq 1}$, we know that its graded dual $C^*$ is an $A_\infty$-algebra with  $m_k: (C^*)^{\otimes k} \longrightarrow C^*$ determined by 
\begin{equation}\label{6.1sudt}m_k (\phi_1\otimes \cdots \otimes \phi_k)= (-1)^{k(|\phi_1|+\ldots +|\phi_k|+1)}(\mu\circ (\phi_1\otimes \cdots \otimes \phi_k)\circ w_k)
\end{equation} where $\mu$ denotes the multiplication on $K$. If $M\in \mathbf M^C$ is a right $C$-comodule, we know that $M$ may be treated as a left $C^*$-module with structure maps given by
\begin{equation}\label{6.2dnk}
\begin{CD}
(C^*)^{\otimes k-1}\otimes M @>(C^*)^{\otimes k-1}\otimes w^M_k>> (C^*)^{\otimes k-1}\otimes M\otimes C^{\otimes k-1} @>>> M \\
\end{CD}
\end{equation} where $\{w_k^M:M\longrightarrow M\otimes C^{\otimes k-1}\}_{k\geq 1}$ comes from the right $C$-comodule structure of $M$.

\begin{defn}\label{D3.7}
Let $C$ be an $A_\infty$-coalgebra and let $A$ be an $A_\infty$-algebra. A pairing between $C$ and $A$ is an $A_\infty$-algebra morphism $f= \{f_k: A^{\otimes k}\longrightarrow C^*\}_{k\geq 1}$ from $A$ to the graded dual $C^*$ of $C$.
\end{defn}

\begin{Thm}\label{T6.4} Let $C$ be an $A_\infty$-coalgebra and let $A$ be an $A_\infty$-algebra. Then, a pairing   $f= \{f_k: A^{\otimes k}\longrightarrow C^*\}_{k\geq 1}$  between $C$ and $A$ induces a functor $\iota(C,A):\mathbf M^{C}\longrightarrow {_A}\mathbf M$.
\end{Thm}

\begin{proof}
From \eqref{6.2dnk}, it follows that we have a functor $\mathbf M^C\longrightarrow {_{C^*}}\mathbf M$. Considering the morphism  $f= \{f_k: A^{\otimes k}\longrightarrow C^*\}_{k\geq 1}$ of $A_\infty$-algebras, we also have a functor ${_{C^*}}\mathbf M\longrightarrow {_A}\mathbf M$ given by restriction of scalars (see, for instance, \cite[$\S$ 6.2]{BK}). The result is now clear. 
\end{proof}

We now want to define a rational pairing between an $A_\infty$-coalgebra and an $A_\infty$-algebra. For this, we will need to refine the result of Theorem \ref{T6.4}. 
First, we consider a morphism
$f=\{f_k: A^{\otimes k}\longrightarrow B\}_{k\geq 1}$ of $A_\infty$-algebras as in Definition \ref{D3.2}. Let $M$ be a left $B$-module, equipped with structure maps
$\{m_k^{B,M}:B^{\otimes k-1}\otimes M\longrightarrow M\}_{k\geq 1}$. Then, the restriction of scalars makes $M$ into an $A$-module, determined by setting (see, for instance, \cite[$\S$ 6.2]{BK})
\begin{equation}\label{rscalar}
m_{k+1}^{A,M}:A^{\otimes k}\otimes M\longrightarrow M\qquad m_{k+1}^{A,M}=\sum (-1)^s m_{l+1}^{B,M}(f_{i_1}\otimes ...\otimes f_{i_l}\otimes M)
\end{equation} for $k\geq 0$, where the sum is taken over all decompositions $k=i_1+...+i_l$ and 
$s=(l-1)(i_1-1)+(l-2)(i_2-1)+\dots + 2(i_{l-2}-1)+(i_{l-1}-1)$. 
We also note that   $f=\{f_k: A^{\otimes k}\longrightarrow B\}_{k\geq 1}$   determines a morphism
\begin{equation}\label{tildef}
\tilde{f}: \left(\underset{k\geq 0}{\bigoplus} \textrm{ }A^{\otimes k} \right)\longrightarrow \left(\underset{l\geq 0}{\bigoplus} \textrm{ }B^{\otimes l} \right)
\end{equation}
of graded vector spaces, where the action of $\tilde{f}$ on the component $A^{\otimes k}$ is given by $\sum (-1)^s(f_{i_1}\otimes ...\otimes f_{i_l})$, where the signs are as in \eqref{rscalar}. Accordingly, for any $V\in Vect^{\mathbb Z}$, we have an induced morphism
\begin{equation}
[\tilde f,V] :\underset{l\geq 0}{\prod}[B^{\otimes l},V] =\left[ \left(\underset{l\geq 0}{\bigoplus} \textrm{ }B^{\otimes l} \right),V\right]\longrightarrow \left[ \left(\underset{k\geq 0}{\bigoplus} \textrm{ }A^{\otimes k} \right),V\right]=\underset{k\geq 0}{\prod}[A^{\otimes k},V] 
\end{equation}
In particular, we consider $f=\{f_k: A^{\otimes k}\longrightarrow C^*\}_{k\geq 1}$ determining the pairing between the $A_\infty$-coalgebra $C$ and the $A_\infty$-algebra $A$. For any $V\in Vect^{\mathbb Z}$, we note that the   closed symmetric structure on $Vect^{\mathbb Z}$ induces a canonical morphism
\begin{equation}\label{6.8uy}
\begin{CD}\alpha(C,A,V):  \underset{l\geq 0}{\prod} V\otimes C^{\otimes l}    @>>> \underset{l\geq 0}{\prod}[(C^*)^{\otimes l},V] @>[\tilde f,V]>>  \underset{k\geq 0}{\prod}[A^{\otimes k},V]  \end{CD}
\end{equation}

\begin{defn}\label{D6.5qh} Let $f= \{f_k: A^{\otimes k}\longrightarrow C^*\}_{k\geq 1}$ be a pairing of an $A_\infty$-coalgebra $C$ and an $A_\infty$-algebra $A$. We will say that the pairing is rational if   the canonical morphism $\alpha(C,A,V )$ as defined in \eqref{6.8uy} is a monomorphism in $Vect^{\mathbb Z}$ for any $V\in Vect^{\mathbb Z}$. 

\end{defn}

\begin{thm}\label{P6.6sg} Let $f= \{f_k: A^{\otimes k}\longrightarrow C^*\}_{k\geq 1}$ be a rational pairing of an $A_\infty$-coalgebra $C$ and an $A_\infty$-algebra $A$. Then, the functor $\iota(C,A):
\mathbf M^C\longrightarrow {_A}\mathbf M$ embeds $\mathbf M^C$ as a full subcategory of ${_A}\mathbf M$.  
\end{thm}

\begin{proof}
 We consider $M$, $N\in \mathbf M^C$ with structure maps $\{w^M_k:M\longrightarrow M\otimes C^{\otimes k-1}\}_{k\geq 1}$ and $\{w_k^N:N\longrightarrow 
 N\otimes C^{\otimes k-1}\}_{k\geq 1}$ respectively. If we treat $M$, $N$ as left $A$-modules via the functor $\iota(C,A)$, we have to show that  any  $A$-module morphism $h:M\longrightarrow N$ is also a morphism in $\mathbf M^C$. We consider therefore  the following diagram 
\begin{equation}\label{6.8cd}
\begin{CD}
M @>h>> N \\
@V\prod w^M_{l+1}VV @VV\prod w^N_{l+1}V \\
\prod M\otimes C^{\otimes l} @>\prod h\otimes C^{\otimes l}>> \prod N\otimes C^{\otimes l} \\
@V\alpha(C,A,M)VV @VV\alpha(C,A,N)V\\
\prod [A^{\otimes k},M] @>\prod [A^{\otimes k},h] >>\prod  [A^{\otimes k},N] \\
\end{CD}
\end{equation} Considering the explicit definition of restriction of scalars in \eqref{rscalar} and the construction of the functor $\iota(C,A)$ in Theorem \ref{T6.4}, we realize that the  outer square in \eqref{6.8cd} commutes because $h$ is an $A$-module morphism. The lower square in \eqref{6.8cd} commutes because  $h$ is a morphism of graded vector spaces. Accordingly, we have
\begin{equation}\label{6.9yd}
\alpha(C,A,N)\circ \left(\prod w^N_{l+1}\right )\circ h=\left(\prod [A^{\otimes k},h] \right)\circ \alpha(C,A,M)\circ \left(\prod w^M_{l+1}\right )=\alpha(C,A,N)\circ \left( \prod h\otimes C^{\otimes l}\right)  \circ \left(\prod w^M_{l+1}\right )
\end{equation} Since the pairing is rational by assumption, we know that $\alpha(C,A,N)$ is a monomorphism. Accordingly, it follows from \eqref{6.9yd} that the upper square in \eqref{6.8cd}
also commutes. This proves that $h$ is a morphism of right $C$-comodules. 
\end{proof}

 \begin{thm}\label{P6.7gk}  Let $f= \{f_k: A^{\otimes k}\longrightarrow C^*\}_{k\geq 1}$ be a rational pairing of an $A_\infty$-coalgebra $C$ and an $A_\infty$-algebra $A$.  Then, the category $\mathbf M^C$, treated as a full subcategory of ${_A}\mathbf M$, is closed under direct sums, subobjects and quotients.
 
 \end{thm} 
 
 \begin{proof} It is clear that  $\mathbf M^C$, treated as a full subcategory of ${_A}\mathbf M$, is closed under direct sums. Suppose we have a short exact sequence
 \begin{equation}
 \begin{CD}
 0@>>> M'@>g>> M@>h>> M''@>>> 0 
 \end{CD}
 \end{equation}
 in ${_A}\mathbf M $ such that $M\in \mathbf M^C$ with structure maps $\{w^M_{l+1}:M\longrightarrow M\otimes C^{\otimes l}\}_{l\geq 0}$. Since products of short exact sequences in
 $Vect^{\mathbb Z}$ are exact, we have  the following commutative diagram
 \begin{equation}\label{6.12cd}
 \begin{CD}
  0@>>> M'@>g>> M@>h>> M''@>>> 0 \\
 @.  @.  @VV\prod w^M_{\bullet+1}V @. @. \\
  0@>>> \prod (M'\otimes C^{ \bullet} )@>\prod (g\otimes C^{\bullet})>> \prod (M\otimes C^{\bullet} )@>\prod (h\otimes C^{\bullet})>> \prod (M''\otimes C^{\bullet})@>>> 0 \\
  @. @V\alpha(C,A,M')VV @V\alpha(C,A,M)VV @V\alpha(C,A,M'')VV @. \\
  0@>>> \prod [A^\bullet,M'] @>\prod [A^\bullet,g] >> \prod [A^\bullet,M] @>\prod [A^\bullet,h] >> \prod [A^\bullet,M''] @>>> 0\\
 \end{CD}
 \end{equation}
  Since $g$ is an $A$-module morphism, we note that the composition $\alpha(C,A,M)\circ \left(\prod w^M_{\bullet+1}\right)\circ g$ factors through $\prod [A^\bullet,g]$. Accordingly, we have
$ \left(\prod [A^\bullet,h]\right)\circ \alpha(C,A,M)\circ \left(\prod w^M_{\bullet+1}\right)\circ g=0$. It follows that
\begin{equation}\label{6.13do}
\alpha(C,A,M'')\circ \left(\prod (h\otimes C^{\bullet})\right)\circ \left(\prod w^M_{\bullet+1}\right)\circ g=\left(\prod [A^\bullet,h]\right)\circ \alpha(C,A,M)\circ \left(\prod w^M_{\bullet+1}\right)\circ g=0
\end{equation}
 Since $f$ is a rational pairing, we know that $\alpha(C,A,M'')$ is a monomorphism, whence it follows from \eqref{6.13do} that $ \left(\prod (h\otimes C^{\bullet})\right)\circ \left(\prod w^M_{\bullet+1}\right)\circ g=0$. Since the middle row in \eqref{6.12cd} is short exact, there are now  induced maps $M'\longrightarrow \prod (M'\otimes C^{ \bullet})$ and $M''\longrightarrow \prod (M''\otimes C^{ \bullet})$ making 
 the diagram commutative and giving $M'$, $M''$ the structure of $C$-comodules. This proves the result.
 \end{proof} 
 
Given the rational pairing $f= \{f_k: A^{\otimes k}\longrightarrow C^*\}_{k\geq 1}$, we now set for any   $M\in {_A}\mathbf M$
\begin{equation}\label{rsubr}
R_f(M):=\underset{g\in {_A}\mathbf M(\iota(C,A)(N),M),N\in \mathbf M^C}{\sum} Im(g)\subseteq M
\end{equation} We now have the main result on rational pairings of $A_\infty$-coalgebras and $A_\infty$-algebras.

\begin{Thm}\label{T6.8nh} Let $f= \{f_k: A^{\otimes k}\longrightarrow C^*\}_{k\geq 1}$ be a rational pairing of an $A_\infty$-coalgebra $C$ and an $A_\infty$-algebra $A$. Then, we have a functor
$R_f:{_A}\mathbf M\longrightarrow \mathbf M^C$ that is right adjoint to $\iota(C,A):\mathbf M^C\longrightarrow {_A}\mathbf M$, i.e., we have natural isomorphisms
\begin{equation}\label{6.15fv}
{_A}\mathbf M(\iota(C,A)(M'),M)\cong \mathbf M^C(M',R_f(M))
\end{equation} for any $M'\in \mathbf M^C$ and $M\in{_A}\mathbf M$.
\end{Thm}

\begin{proof} We consider some $g\in {_A}\mathbf M(\iota(C,A)(N),M)$, with $N\in \mathbf M^C$. Since $\mathbf M^C$ as a full subcategory of ${_A}\mathbf M$ is closed under quotients,
we see that $Im(g)\in \mathbf M^C$. Applying Proposition \ref{P6.7gk} again, since $\mathbf M^C$  is closed under direct sums and quotients in ${_A}\mathbf M$, we see that the sum $R_f(M)=\sum Im(g)$ defined as in \eqref{rsubr} must in $ \mathbf M^C$. 

\smallskip
Further, from the definition in \eqref{rsubr}, it follows that the image of any morphism $\iota(C,A)(M')\longrightarrow M$ in ${_A}\mathbf M$ must lie in $R_f(M)$. From this it is clear that we have an isomorphism ${_A}\mathbf M(\iota(C,A)(M'),M)\cong \mathbf M^C(M',R_f(M))$. 

\smallskip
It remains to show that $R_f$ is a functor. For this, we consider a morphism $h:M_1\longrightarrow M_2$ in ${_A}\mathbf M$. Since $R_f(M_1)$ lies in $\mathbf M^C$, it follows from the definition in \eqref{rsubr} that the image of the composition $R_f(M_1)\hookrightarrow M_1\xrightarrow{h}M_2$ lies in $R_f(M_2)$. This proves the result. 

\end{proof}

\section{$A_\infty$-contramodules}

Let $C=\underset{n\in \mathbb Z}{\bigoplus}C_n$ be an $A_\infty$-coalgebra, equipped with structure maps $\{w_k:C\longrightarrow C^{\otimes k}\}_{k\geq 1}$, where $w_k$ has degree
$2-k$  (see Definition \ref{ApD3.4}). For any $k\geq 1$ and $\bar{n}=(n_1,...,n_k)\in \mathbb Z^k$, let $w_k(\bar{n}):C_{-\Sigma\bar{n}-(2-k)}\longrightarrow C_{-n_1}\otimes ... \otimes C_{-n_k}$ be a component of the structure map $w_k:C\longrightarrow C^{\otimes k}$ of $C$. We define $w'_k:C\longrightarrow C^{\otimes k}$ whose components are given by
$
w'_k(\bar{n}):=(-1)^{k(\Sigma\bar{n}+1)}w_k(\bar{n}):C_{-\Sigma\bar{n}-(2-k)}\longrightarrow C_{-n_1}\otimes ... \otimes C_{-n_k}
$. We are now ready to introduce $A_\infty$-contramodules over $C$.

\begin{defn}\label{D7.1wc} Let $C$ be an $A_\infty$-coalgebra.  An $A_\infty$-contramodule  is a $\mathbb Z$-graded vector space $M$ equipped with structure maps
\begin{equation}\label{req7.3j}
t^M_k:[C^{\otimes k-1},M]\longrightarrow M \qquad k\geq 1
\end{equation} with $t_k^M$ of degree $(2-k)$ satisfying, for each $n\geq 1$, 
\begin{equation}\label{k7.3c}
\sum (-1)^{p+qr}t^M_{p+1+r}\circ [C^{\otimes p}\otimes w_q'\otimes C^{\otimes r-1},M]+\sum (-1)^a t^M_{a+1}\circ [C^{\otimes a},t^M_b]=0
\end{equation} where the first sum in \eqref{k7.3c} is taken over all $p\geq 0$,  $q,r\geq 1$ such that $p+q+r=n$ and the second sum in \eqref{k7.3c} is taken over all
$a\geq 0$, $b\geq 1$ with $a+b=n$. In \eqref{k7.3c}, we choose the isomorphism $[C^{\otimes a},[C^{\otimes b-1},M]]\cong [C^{\otimes a+b-1},M]$  making $M$ a right $A_\infty$-contramodule.

\smallskip
A morphism $h:M\longrightarrow M'$ of $A_\infty$-contramodules is a morphism of graded vector spaces that is compatible with the respective structure maps. We will denote by
$\mathbf M_{[C,\_\_]}$ the category of $A_\infty$-contramodules over $C$.
\end{defn}

We say that $M\in \mathbf M_{[C,\_\_]}$ is even if $t_k^M=0$ whenever $k$ is odd. We denote by $\mathbf M_{[C,\_\_]}^e$ the full subcategory of $\mathbf M_{[C,\_\_]}$ consisting of even objects. 
By Proposition \ref{scuP3.4}, we know that when $C$ is even, the collection of functors
\begin{equation}\label{e3.3eqi}
\mathbb V^{TC}=\{\mathbb V_n^{TC}:=Hom(C_{-n}, \_\_):Vect\longrightarrow Vect\}_{n\in \mathbb Z}
\end{equation} is an   $A_\infty$-monad. The following result is now the $A_\infty$-contramodule counterpart of Proposition \ref{P5.4fd}. 

\begin{thm}\label{P7.2fd} Let $C$ be an even $A_\infty$-coalgebra. Then, there is a canonical functor 
\begin{equation}\label{5.6ywzs}
\mathscr I_{[C,\_\_]}: \mathbf M_{[C,\_\_]}^e\longrightarrow EM_{\mathbb V^{TC}}^e
\end{equation} that embeds $\mathbf M_{[C,\_\_]}^e$ as a subcategory of  even modules over the $A_\infty$-monad $\mathbb V^{TC}$. If $C$ is such that $C_n\ne 0$ for only finitely many $n\in \mathbb Z$, then the categories $ \mathbf M_{[C,\_\_]}$ and $EM_{\mathbb V^{TC}}$ are isomorphic.
\end{thm}

\begin{proof}
Let $M\in \mathbf M_{[C,\_\_]}^e$. For $k\geq 1$, $\bar{n}\in \mathbb Z^{k-1}$, $l\in \mathbb Z$, we define
\begin{equation}\label{inc7}
\pi_k(\bar{n},l):\mathbb V^{TC}_{\bar{n}}M_l=[C_{-n_1}\otimes ...\otimes C_{-n_{k-1}},M_l]\hookrightarrow [C^{\otimes k-1},M]_{\Sigma\bar{n}+l}\xrightarrow{t^M_{k,\Sigma\bar{n}+l}}M_{\Sigma\bar{n}+l+(2-k)}
\end{equation} We denote by $\{\theta_k(\bar{n})\}$ the structure maps of the $A_\infty$-monad $\mathbb V^{TC}$. For $N\geq 0$, $\bar{z}=(\bar{n},\bar{n}',\bar{n}'')\in \mathbb Z^N$ with $|\bar{n}|=p$, $|\bar{n}'|=q$, $|\bar{n}''|=r-1$, we observe that the following diagram is commutative
\begin{equation}\label{7.6cd1}
\begin{CD}
\mathbb V^{TC}_{\bar{n}}\mathbb V^{TC}_{\bar{n}'}\mathbb V^{TC}_{\bar{n}''}M_l @>(\mathbb V^{TC}_{\bar{n}}\ast\theta_q(\bar{n}')\ast\mathbb V^{TC}_{\bar{n}''})M_l>> \mathbb V^{TC}_{\bar{n}}\mathbb V^{TC}_{\Sigma\bar{n}'+(2-q)}\mathbb V^{TC}_{\bar{n}''}M_l @>\pi_{p+1+r}(\bar{n},\Sigma\bar{n}'+(2-q),\bar{n}'',l)>> M_{\Sigma\bar{z}+l+(2-N)}\\
@VVV @VVV @VVV \\
[C^{\otimes N},M]_{\Sigma\bar{z}+l}@>[C^{\otimes p}\otimes w'_q\otimes C^{\otimes r-1},M]_{\Sigma\bar{z}+l}>> [C^{\otimes (p+r)},M]_{\Sigma\bar{z}+l+(2-q)}@>t^M_{p+1+r,\Sigma\bar{z}+l+(2-q)}>> M_{\Sigma\bar{z}+l+(2-N)}\\
\end{CD}
\end{equation}  Similarly, for  $\bar{z}=(\bar{m},\bar{m}')\in \mathbb Z^N$ with $|\bar{m}|=a$,  $|\bar{m}'|=b-1$, we observe that the following diagram is commutative
\begin{equation}\label{7.6cd2}
\begin{CD}
\mathbb V^{TC}_{\bar{m}}\mathbb V^{TC}_{\bar{m}'}M_l @>\mathbb V^{TC}_{\bar{m}}(\pi_b(\bar{m}',l))>> \mathbb V^{TC}_{\bar{m}} M_{\Sigma\bar{m}'+l+(2-b)} @>\pi_{a+1}(\bar{m},\Sigma\bar{m}'+l+(2-b))>> M_{\Sigma\bar{z}+l+(2-N)}\\
@VVV @VVV @VVV \\
[C^{\otimes a},[C^{\otimes b-1},M]]_{\Sigma\bar{z}+l}@>[C^{\otimes a},t^M_b]_{\Sigma\bar{z}+l}>> [C^{\otimes a},M]_{\Sigma\bar{z}+l+(2-b)}@>t^M_{a+1,\Sigma\bar{z}+l+(2-b)}>> M_{\Sigma\bar{z}+l+(2-N)}\\
\end{CD}
\end{equation} Using the fact that the expressions in \ref{k7.3c} are zero, it follows from \eqref{7.6cd1} and \eqref{7.6cd2} that the $\pi_k(\bar{n},l)$ defined in \eqref{inc7} make 
$M$ into a module over the $A_\infty$-monad $\mathbb V^{TC}$. We therefore have the functor $ \mathbf M_{[C,\_\_]}^e\longrightarrow EM_{\mathbb V^{TC}}^e$. 

\smallskip
Now suppose that $C_n\ne 0$ for only finitely many $n\in \mathbb Z$. Then, it follows that for any $T\in \mathbb Z$, we have
\begin{equation}\label{7.8qk}
[C^{\otimes k-1},M]_T=\underset{n_1+...+n_{k-1}+l=T}{\prod}[C_{-n_1}\otimes ...\otimes C_{-n_{k-1}},M_l] =\underset{n_1+...+n_{k-1}+l=T}{\bigoplus}\mathbb V^{TC}_{n_1}...\mathbb V_{n_{k-1}}^{TC}(M_l) 
\end{equation} In other words, the structure maps $t^M_k:[C^{\otimes k-1},M]\longrightarrow M$ of $M\in \mathbf M_{[C,\_\_]}^e$ are completely determined by the induced morphisms $\mathbb V^{TC}_{n_1}...\mathbb V_{n_{k-1}}^{TC}(M_l) \longrightarrow M_{n_1+...+n_{k-1}+l+(2-k)}$ on components. It follows that  the categories $ \mathbf M_{[C,\_\_]}^e$ and $EM_{\mathbb V^{TC}}^e$ are isomorphic.
\end{proof}

\begin{thm}\label{L7.2eh} Let $C$ be an $A_\infty$-coalgebra. Then, there exists a faithful functor $\kappa^C: \mathbf M_{[C,\_\_]}\longrightarrow \mathbf M_{C^*}$. If the $A_\infty$-coalgebra $C$ is such that the canonical map
$
W\otimes C^* \longrightarrow [C,W] 
$
is an epimorphism for each $W\in Vect^{\mathbb Z}$, then $\kappa^C$ is an isomorphism. 
\end{thm}

\begin{proof}
Let $M\in \mathbf M_{[C,\_\_]}$, with structure maps as in \eqref{req7.3j}. Then, for any $k\geq 1$, we consider the compositions
\begin{equation}\label{7.4fq} 
m_k^M: M\otimes (C^*)^{\otimes k-1}\longrightarrow [C^{\otimes k-1},M]\xrightarrow{t_k^M} M 
\end{equation}  where the morphism $ M\otimes (C^*)^{\otimes k-1} \longrightarrow [C^{\otimes k-1},M]$ in \eqref{7.4fq} is induced by the canonical morphism
$M\otimes (C^*)^{\otimes k-1}\otimes  C^{\otimes k-1} \longrightarrow M$. This makes $M$ into a $C^*$-module. It is also clear that the functor $\kappa^C$ is faithful.

\smallskip
In particular, suppose that the $A_\infty$-coalgebra $C$ is such that the canonical map
$
W\otimes C^* \longrightarrow [C,W] 
$
is an epimorphism for each $W\in Vect^{\mathbb Z}$. We note that for each $n\in \mathbb Z$, the canonical map
\begin{equation}\label{7.7go}
(W\otimes C^* )_n =\underset{k\in \mathbb Z}{\bigoplus}(W_{k+n}\otimes (C^*)_{-k}) \longrightarrow \underset{k\in \mathbb Z}{\bigoplus} [C_k,W_{k+n}]\hookrightarrow  \underset{k\in \mathbb Z}{\prod} [C_k,W_{k+n}]=[C,W]_n
\end{equation}
is always a monomorphism, i.e.  $
W\otimes C^*  \longrightarrow [C,W] 
$
is a monomorphism for each $W\in Vect^{\mathbb Z}$.  By assumption, the morphisms in \eqref{7.7go} are also epimorphisms.  Then, the maps $M\otimes (C^*)^{\otimes k-1} \longrightarrow [C^{\otimes k-1},M]$    which appear in  \eqref{7.4fq}
are all isomorphisms and it is clear that the functor $\kappa^C$ is an isomorphism.
\end{proof}

\begin{thm}\label{P7.3vax}  Let $C$ be an $A_\infty$-coalgebra.  Then, there is a functor $[C,\_\_]:Vect^{\mathbb Z}\longrightarrow \mathbf M_{[C,\_\_]}$.  
\end{thm}

\begin{proof} We consider $W\in Vect^{\mathbb Z}$ and put $N=[C,W]$.
By setting $t_k^N:=[w'_k,W]:[C^{\otimes k-1},N]=[C^{\otimes k-1},[C,W]]=[C^{\otimes k},W]\longrightarrow [C,W]=N$ where $w_k':C\longrightarrow C^{\otimes k}$ is as at the beginning of this section, we obtain 
$N=[C,W]\in \mathbf M_{[C,\_\_]}$.
\end{proof}

We will now construct functors between categories of $A_\infty$-comodules and $A_\infty$-contramodules. Let $(C,w^C_\bullet)$, $(D,w^D_\bullet)$ be $A_\infty$-coalgebras. By a $(C,D)$-space $(N,w^L_\bullet,w^R_\bullet)$, we will mean a graded vector space $N$ equipped with families of morphisms
\begin{equation}\label{8.1sh}
w^L_\bullet=\{w^L_p:N\longrightarrow C^{p-1}\otimes N\}_{p\geq 1}\qquad w^R_\bullet=\{w^R_q:N\longrightarrow N\otimes D^{\otimes q-1}\}_{q\geq 1}
\end{equation} such that $(N,w^L_\bullet)$ is a left $C$-comodule and $(N,w^R_\bullet)$ is a right $D$-comodule. In particular, any $(C,D)$-bicomodule   becomes a $(C,D)$-space in an obvious manner. Additionally, for $Q$, $Q'\in \mathbf M^D$, we set $[Q,Q']^D_p=\mathbf M^D(Q[-p],Q')$ for $p\in \mathbb Z$.

\begin{thm}\label{Pro8.1} Let $(C,w^C_\bullet)$, $(D,w^D_\bullet)$ be $A_\infty$-coalgebras and let $(N,w^L_{\bullet},w^R_\bullet)$ be a $(C,D)$-space. Then, for each $Q\in \mathbf M^D$ and each $k\geq 1$, there is an isomorphism
\begin{equation}\label{eq8.1qs}
[C^{\otimes k-1}, [N,Q]^D]\cong [ C^{\otimes k-1}\otimes N,Q]^D
\end{equation}

\end{thm}

\begin{proof}
We consider $f\in [C^{\otimes k-1}, [N,Q)]^D]$. Then, $f$ corresponds to   \begin{equation} f':C^{\otimes k-1}\otimes N\longrightarrow Q \qquad 
(\tilde{c}\otimes n)\mapsto f(\tilde{c})(n)
\end{equation} In order to prove that $f'\in  [ C^{\otimes k-1}\otimes N,Q]^D$, we have to show that the following diagram commutes for each $l\geq 1$
\begin{equation}\label{8.3cd}
\begin{CD}
C^{\otimes k-1}\otimes N @>C^{\otimes k-1}\otimes w^R_l>> C^{\otimes k-1}\otimes N\otimes D^{\otimes l-1}\\
@Vf'VV @VV(f'\otimes D^{\otimes l-1})V \\
Q @>w^Q_{l}>> Q\otimes D^{\otimes l-1} \\
\end{CD} \qquad \Rightarrow \qquad \begin{array}{l} f(\tilde{c})(w^R_l(n)_0)\otimes w^R_l(n)_1\\ \\ = w^Q_l(f(\tilde{c})(n))_0\otimes w^Q_l(f(\tilde{c})(n))_1  \\
\end{array}
\end{equation} where we have adapted the Sweedler notation in \eqref{8.3cd}. But since $f(\tilde{c})$ is a $D$-comodule morphism, we already have
\begin{equation}\label{8.4fq}
w^Q_l(f(\tilde{c})(n))_0\otimes w^Q_l(f(\tilde{c})(n))_1=f(\tilde{c})(w^R_l(n)_0)\otimes w^R_l(n)_1
\end{equation} From \eqref{8.3cd} and \eqref{8.4fq}, it follows that $f'\in \mathbf M^D(C^{\otimes k-1}\otimes N,Q)$. 
Conversely, we consider $g\in  [ C^{\otimes k-1}\otimes N,Q]^D$. Then, $g$ corresponds to 
\begin{equation}
g':C^{\otimes k-1}\longrightarrow [N,Q] \qquad g'(\tilde{c})(n):=g(\tilde{c}\otimes n)
\end{equation} In order to show that $g'$ takes values in $ [N,Q]^D$, we have to show that for each $l\geq 1$ we must have
\begin{equation}\label{86q}
\begin{array}{ll}
w^Q_l(g(\tilde{c}\otimes n))_0\otimes w^Q_l(g(\tilde{c}\otimes n))_1&=w^Q_l(g'(\tilde{c})(n))_0\otimes w^Q_l(g'(\tilde{c})(n))_1\\
&{=} g'(\tilde{c})(w^R_l(n)_0)\otimes w^R_l(n)_1=  g(\tilde{c}\otimes w^R_l(n)_0)\otimes w^R_l(n)_1 \\
\end{array}\end{equation} However since $g\in  [ C^{\otimes k-1}\otimes N,Q]^D $, we already know that
\begin{equation}\label{87q}
w^Q_l(g(\tilde{c}\otimes n))_0\otimes w^Q_l(g(\tilde{c}\otimes n))_1=g(\tilde{c}\otimes w^R_l(n)_0)\otimes w^R_l(n)_1
\end{equation} From \eqref{86q} and \eqref{87q} it is clear that $g'\in [C^{\otimes k-1}, [N,Q]^D]$. This proves the result.
\end{proof}

We will say that  $(N,w^L_\bullet,w^R_\bullet)$ is a commuting $(C,D)$-space if the $C$ and $D$-coactions commute with each other, i.e., for each $k\geq 1$, $l\geq 1$, we have a commutative diagram
\begin{equation}\label{8.9cds}
\begin{CD}
N @>w^R_l>> N\otimes D^{\otimes l-1}\\
@Vw^L_kVV @VVw^L_k\otimes D^{\otimes l-1}V\\
C^{\otimes k-1}\otimes N @>C^{\otimes k-1}\otimes w^R_l>> C^{\otimes k-1}\otimes N\otimes D^{\otimes l-1}\\
\end{CD}
\end{equation} We note in particular that this means that each $w^R_l:N\longrightarrow N\otimes D^{\otimes l-1}$ is a morphism of left $C$-comodules, while each 
$w^L_k:N\longrightarrow C^{\otimes k-1}\otimes N$ is a morphism of right $D$-comodules.  For the rest of this section, we will assume that both the $C$-comodule and $D$-comodule structures on $N$ are even.

\begin{thm}\label{Pro8.2e}  Let $(C,w^C_\bullet)$, $(D,w^D_\bullet)$ be even $A_\infty$-coalgebras. Then,   $(N,w^L_{\bullet},w^R_\bullet)$ induces a functor
$[N,\_\_]^D:\mathbf M^D\longrightarrow \mathbf M_{[C,\_\_]}$.
\end{thm}

\begin{proof}
We consider $Q\in \mathbf M^D$. We claim that $ [N,Q]^D\in \mathbf M_{[C,\_\_]}$. From Proposition \ref{Pro8.1}, we already have that $[C^{\otimes k-1}, [N,Q]^D]\cong [ C^{\otimes k-1}\otimes N,Q]^D$ for each $k\geq 1$. We now have structure maps 
\begin{equation}\label{8.8zd}
\begin{CD}
t_k:[C^{\otimes k-1}, [N,Q]^D]\cong [ C^{\otimes k-1}\otimes N,Q]^D@>[w^L_k,Q]^D>>[N,Q]^D\\
\end{CD}
\end{equation}  where we have used the fact that each morphism $w^L_k:N\longrightarrow C^{\otimes k-1}\otimes N$ is a morphism of right $D$-comodules. It may be verified that the morphisms
in \eqref{8.8zd} make $[N,Q]^D$ into an $A_\infty$-contramodule over $C$. 
\end{proof}

Let $(M,t_\bullet^M)$ be an $A_\infty$-contramodule over $C$. Then, for each $k\geq 1$, we have a pair of canonical maps
\begin{equation}\label{8.11tik}
 [C^{\otimes k-1},M]\otimes N\doublerightarrow{(ev_{[C^{\otimes k-1},M]}\otimes N)\circ ([C^{\otimes k-1},M]\otimes w^L_k)}{t^M_{k}\otimes N}  M\otimes N
\end{equation} where $ev_{[C^{\otimes k-1},M]}:  [C^{\otimes k-1},M]\otimes C^{\otimes k-1}\longrightarrow M$ is the canonical evaluation map.  We now define the contratensor product
$M\boxtimes_CN$ to be the coequalizer
\begin{equation}\label{ctrtens}
M\boxtimes_CN:=Coeq\left(\underset{k\geq 1}{\bigoplus}\textrm{ } [C^{\otimes k-1},M]\otimes N\doublerightarrow{\underset{k\geq 1}{\bigoplus}\textrm{ }(ev_{[C^{\otimes k-1},M]}\otimes N)\circ ([C^{\otimes k-1},M]\otimes w^L_k)}{t^M_{k}\otimes N}  M\otimes N\right)
\end{equation} We note that \eqref{ctrtens} extends the contratensor product on contramodules over coassociative coalgebras defined in \cite{semicontra}. We now have the following result.

\begin{thm}\label{Pro8.3v}   Let $(C,w^C_\bullet)$, $(D,w^D_\bullet)$ be even $A_\infty$-coalgebras. Then,   $(N,w^L_{\bullet},w^R_\bullet)$  induces a functor 
$\_\_\boxtimes_C N:  \mathbf M_{[C,\_\_]}\longrightarrow \mathbf M^D$.
\end{thm} 

\begin{proof} By \eqref{ctrtens}, we know that for any $M\in  \mathbf M_{[C,\_\_]}$, the contratensor product $M\boxtimes_CN$ is a quotient of $M\otimes N$. We claim that the maps
$M\otimes w^R_l: M\otimes N\longrightarrow M\otimes N\otimes D^{\otimes l-1}$ descend to the quotient $M\boxtimes_CN$. For this, we note that for any $k,l\geq 1$, we have the following commutative diagram
\begin{equation}\label{8.13cda}\small
\begin{CD}
 [C^{\otimes k-1},M]\otimes N @>([C^{\otimes k-1},M]\otimes w^L_k)>> \underset{\otimes C^{\otimes k-1}\otimes N}{\overset{ [C^{\otimes k-1},M]}{}} @>(ev_{[C^{\otimes k-1},M]}\otimes N)>>M\otimes N \\
@V [C^{\otimes k-1},M]\otimes w^R_lVV @V[C^{\otimes k-1},M]\otimes C^{\otimes k-1}\otimes w^R_lVV @VM\otimes w^R_lVV \\
[ C^{\otimes k-1},M]\otimes N\otimes D^{\otimes l-1} @>([C^{\otimes k-1},M]\otimes w^L_k)\otimes D^{\otimes l-1}>> \underset{\otimes C^{\otimes k-1}\otimes N\otimes D^{\otimes l-1}}{\overset{ [C^{\otimes k-1},M]}{}}@>(ev_{[C^{\otimes k-1},M]}\otimes N)\otimes D^{\otimes l-1}>> M\otimes N\otimes  D^{\otimes l-1}\\
\end{CD}
\end{equation} We note that the left side square in \eqref{8.13cda} commutes because $N$ is a commuting $(C,D)$-space. Also, for any $k,l\geq 1$, we have the commutative diagram
\begin{equation}\label{8.14cda}
\begin{CD}
[C^{\otimes k-1},M]\otimes N @>[C^{\otimes k-1},M]\otimes w^R_l>> [C^{\otimes k-1},M]\otimes N\otimes D^{\otimes l-1} \\
@Vt^M_{k}\otimes NVV @VV t^M_k\otimes N\otimes D^{\otimes l-1}V \\
M\otimes N @>M\otimes w^R_l>> M\otimes N\otimes  D^{\otimes l-1}\\
\end{CD} 
\end{equation}
Considering \eqref{8.13cda}, \eqref{8.14cda} and the definition in \eqref{ctrtens}, we have an induced map on the coequalizers $M\boxtimes_CN\longrightarrow (M\boxtimes_CN)\otimes D^{\otimes l-1}
$ for each $l\geq 1$, making $M\boxtimes_CN$ into an $A_\infty$-comodule over $D$.  
\end{proof} 

The following simple observation will be used in the proof of the next theorem. If $(Q,w^Q_\bullet)$ and $(Q',w^{Q'}_\bullet)$ are $D$-comodules,
then $\mathbf M^D(Q,Q')$ may be expressed as the equalizer
\begin{equation}\label{8.15pw}
\mathbf M^D(Q,Q')=Eq\left([Q,Q']\doublerightarrow{\qquad f\mapsto{\prod} w^{Q'}_\bullet\circ f\qquad }{f\mapsto {\prod}(f\otimes D^{\otimes \bullet-1})\circ 
w^Q_\bullet}{\prod}[Q,Q'\otimes D^{\otimes \bullet-1} ]\right)
\end{equation} Similarly, if $(M,t^M_\bullet)$ and $(M',t^{M'}_\bullet)$ are $C$-contramodules, then $\mathbf M_{[C,\_\_]}(M,M')$ may be expressed as the equalizer
\begin{equation}\label{8.16pw}
\mathbf M_{[C,\_\_]}(M,M')=Eq\left([M,M']\doublerightarrow{\qquad f\mapsto{\prod} t^{M'}_\bullet\circ [C^{\bullet-1},f]\qquad }{f\mapsto {\prod}f\circ t^M_\bullet}{\prod}[[C^{\otimes \bullet-1},M], M']\right)
\end{equation} We are now ready to prove the main result of this section. 

\begin{Thm}\label{T8.4co}
  Let $(C,w^C_\bullet)$, $(D,w^D_\bullet)$ be even $A_\infty$-coalgebras. Then, any commuting $(C,D)$-space  $(N,w^L_{\bullet},w^R_\bullet)$ that is even leads to an adjunction of functors
  \begin{equation}\label{8.17ha}
  \mathbf M^D(M\boxtimes_CN, Q)\cong \mathbf M_{[C,\_\_]}(M,[N,Q]^D)
  \end{equation} for $M\in \mathbf M_{[C,\_\_]}$ and $Q\in \mathbf M^D$. 
\end{Thm} 

\begin{proof}
Using the definition in \eqref{ctrtens} as well as the expressions in \eqref{8.15pw} and \eqref{8.16pw}, we can now compute directly
\begin{equation*}\small
\begin{array}{l}
  \mathbf M^D(M\boxtimes_CN, Q)\\
  =Eq\left([M\boxtimes_CN,Q]\doublerightarrow{\qquad \qquad }{}\underset{l\geq 1}{\prod}[M\boxtimes_CN,Q\otimes D^{\otimes l-1}]\right) \\
  =Eq\left(\left[Coeq\left(\underset{k\geq 1}{\oplus} [C^{\otimes k-1},M]\otimes N\doublerightarrow{ }{}M\otimes N\right),Q\right]\doublerightarrow{  }{}\underset{l\geq 1}{\prod} \left[Coeq\left(\underset{k\geq 1}{\oplus}[C^{\otimes k-1},M]\doublerightarrow{  }{}M\otimes N\right),Q\otimes D^{\otimes l-1}\right]\right)\\
   =Eq\left( Eq\left([M\otimes N,Q]\doublerightarrow{}{}\underset{k\geq 1}{\prod} [[C^{\otimes k-1},M]\otimes N,Q]\right)\doublerightarrow{}{}\underset{l\geq 1}{\prod}Eq\left([M\otimes N,Q\otimes D^{\otimes l-1}]\doublerightarrow{}{}\underset{k\geq 1}{\prod} [[C^{\otimes k-1},M]\otimes N,Q\otimes D^{\otimes l-1}]\right)\right)\\
   =Eq\left( Eq\left([M,[N,Q]]\doublerightarrow{}{}\underset{k\geq 1}{\prod} [[C^{\otimes k-1},M],[N,Q]]\right)\doublerightarrow{}{}\underset{l\geq 1}{\prod}Eq\left([M,[N,Q\otimes D^{\otimes l-1}]]\doublerightarrow{}{}\underset{k\geq 1}{\prod} [[C^{\otimes k-1},M],[N,Q\otimes D^{\otimes l-1}]]\right)\right)\\
    =Eq\left( Eq\left([M,[N,Q]]\doublerightarrow{}{}\underset{k\geq 1}{\prod} [[C^{\otimes k-1},M],[N,Q]]\right)\doublerightarrow{}{}Eq\left(\underset{l\geq 1}{\prod}[M,[N,Q\otimes D^{\otimes l-1}]]\doublerightarrow{}{}\underset{l\geq 1}{\prod}\underset{k\geq 1}{\prod} [[C^{\otimes k-1},M],[N,Q\otimes D^{\otimes l-1}]]\right)\right)\\
     =Eq\left( Eq\left([M,[N,Q]]\doublerightarrow{}{}\underset{l\geq 1}{\prod}[M,[N,Q\otimes D^{\otimes l-1}]]\right)\doublerightarrow{}{}Eq\left(\underset{k\geq 1}{\prod} [[C^{\otimes k-1},M],[N,Q]]\doublerightarrow{}{}\underset{k\geq 1}{\prod}\underset{l\geq 1}{\prod} [[C^{\otimes k-1},M],[N, Q\otimes D^{\otimes l-1}]]\right)\right)\\
       =Eq\left( \left[M,Eq\left([N,Q]\doublerightarrow{}{}\underset{l\geq 1}{\prod}[N,Q\otimes D^{\otimes l-1}]\right)\right]\doublerightarrow{}{}\underset{k\geq 1}{\prod}\left[[C^{\otimes k-1},M],Eq\left([N,Q]\doublerightarrow{}{}\underset{l\geq 1}{\prod}[N,Q\otimes D^{\otimes l-1}]\right)\right]\right)\\
       =Eq\left(\left[M,[N,Q]^D\right]\doublerightarrow{\qquad \qquad }{}\underset{k\geq 1}{\prod}\left[[C^{\otimes k-1},M],[N,Q]^D\right]\right)\\
      = \mathbf M_{[C,\_\_]}(M,[N,Q]^D)\\
\end{array}
\end{equation*}
\end{proof}

We conclude with the following result.

\begin{cor}\label{C8.5} Let $(C,w^C_\bullet)$, $(D,w^D_\bullet)$ be even $A_\infty$-coalgebras. Let $N_1$ be an even left $A_\infty$-comodule over $C$ and $N_2$ be an  even right $A_\infty$-comodule over $D$. 
Then, we have an induced adjunction of functors, i.e., natural isomorphisms
\begin{equation}\label{8.17ho}
  \mathbf M^D(M\boxtimes_C(N_1\otimes N_2), Q)\cong \mathbf M_{[C,\_\_]}(M,[(N_1\otimes N_2),Q]^D)
  \end{equation} for $M\in \mathbf M_{[C,\_\_]}$ and $Q\in \mathbf M^D$. 
\end{cor} 

\begin{proof}
If $N_1$ is a left $C$-comodule and $N_2$ is a  right $D$-comodule, then $N_1\otimes N_2$ carries  left $C$-coactions and right $D$-coactions in the obvious manner. It is clear that the left $C$-coactions commute with the right $D$-coactions, making $N_1\otimes N_2$ a commuting $(C,D)$-space. The result now follows from Theorem \ref{T8.4co}.
\end{proof}

\section{Appendix : Basic Definitions}\label{Appendix}

For the convenience of the reader, we recall in this section the well known definitions of $A_\infty$-algebras, $A_\infty$-coalgebras as well as modules and comodules over them respectively.

\begin{defn}\label{D3.1} An $A_\infty$-algebra over $K$ is a $\mathbb Z$-graded vector space $A=\underset{n \in \mathbb Z}{\bigoplus} A_n$ along with a family of graded $K$-linear maps
$\{m_k: A^{\otimes k}\longrightarrow A\}_{k\geq 1}$ with $m_k$ of degree
 $2-k$, satisfying for each $n\geq 1$:
\begin{equation}\label{Ap4.1} \sum(-1)^{p+qr} m_{p+1+r}(1^{\otimes p} \otimes m_q \otimes 1^{\otimes r})=0 \end{equation}
where $1$ denotes the identity map and the sum is over all triples $(p, q, r)$ such that $n=p+q+r$, $p, r \geq 0$, $q \geq 1$.
\end{defn}
We should mention that the  terms in \eqref{Ap4.1} will have additional signs when applied to elements of $A$, following the usual Koszul sign conventions. 

\begin{defn}\label{D3.2} Let $A$ and $B$ be $A_\infty$-algebras. A morphism of $A_\infty$-algebras $f:A\longrightarrow B$ is a family of graded $K$-linear maps $\{f_k: A^{\otimes k}\longrightarrow B\}_{k\geq 1}$ with $f_k$ of degree $1-k$, satisfying for each $n\geq 1$:
\begin{equation}\label{Ap4.2}\sum (-1)^{p +qr} f_{p+1+r}(1^{\otimes p} \otimes m_q \otimes 1^{\otimes r})= \sum (-1)^s m_t ( f_{i_1}\otimes f_{i_2}\otimes \cdots \otimes f_{i_t})\end{equation}
where the first sum is over all $(p, q, r)$ such that $n= p+q+r$ and the second sum is over all $1 \leq t \leq n$ and all decompositions $n= i_1 + i_2 + \cdots +i_t$, and the sign $s$ is given by 
$s= (t-1)(i_1-1) + (t-2)(i_2-1)+ \cdots + 2(i_{t-2}-1)+ (i_{t-1}-1).$
\end{defn}

\begin{defn}\label{D3.3} Let $A$ be an $A_\infty$-algebra over $K$. A (left) $A_\infty$-module over $A$ is a $\mathbb Z$-graded space $M$ with maps
$\{m_k^M: A^{\otimes k-1} \otimes M \longrightarrow M\}_{k\geq 1}$ with $m_k^M$ 
of degree $2-k$, satisfying for each $n\geq 1$:
$$\sum (-1)^{p+qr} m_{p+1+r}^M(1_A^{\otimes p} \otimes m_q \otimes 1_A^{\otimes r-1}\otimes 1_M) + \sum (-1)^a m_{a+1}^M ( 1^{\otimes a}_A \otimes m_b^M )=0, $$
where the first sum is over all $p, q, r$ such that $p+q+r=n$, $p\geq 0$, $q, r \geq 1$   and the second sum is over all $a,b$ such that $a+b=n$, $a\geq 0$ and $b\geq 1$.

\smallskip
A morphism of left $A_\infty$-modules $f: M\longrightarrow N$ is a family of maps $f=\{f_l: M_l \longrightarrow N_l\}_{l\in \mathbb Z}$ that is well behaved with respect to the structure maps of $M$ and $N$. The category of left $A_\infty$-modules over $A$ will be denoted by $_A\mathbf M$. Similarly, we may define the category $\mathbf M_A$ of right $A$-modules.
\end{defn}

\begin{defn}\label{D3.4} Let $M$, $N$ be left $A_\infty$-modules over the $A_\infty$-algebra $A$, with respective structure maps 
$\{m_k^M: A^{\otimes k-1}\otimes M \longrightarrow M\}_{k\geq 1}$ and $\{m_k^N: A^{\otimes k-1}\otimes N \longrightarrow N\}_{k\geq 1}$. 
An $A_\infty$-morphism of (left) $A_\infty$-modules $f: M\longrightarrow N$ is a family of maps
$$f=\{f_k: A^{\otimes k-1} \otimes M \longrightarrow N\}_{k\geq 1}$$ with $f_k$ of degree $1-k$, satisfying for each $n\geq 1$:
$$\sum (-1)^{a} f_{a+1}(1^{\otimes a}_A\otimes m_b^M)+\sum (-1)^{a+bc} f_{a+1+c}(1^{\otimes a}_A \otimes m_b \otimes 1^{\otimes c-1}_A\otimes 1_M) = \sum m_{a+1}^N( 1^{\otimes a}_A \otimes f_b ), $$ where:

\smallskip
(1)  the first sum on the left side is taken over all $a,b$ such that $n=a+b$, $a\geq 0$, $b\geq 1$ 

\smallskip
(2)  the second sum on the left side is taken over all $a,b,c$ such that $n=a+b+c$, $a\geq 0$, $b,c\geq 1$ 

\smallskip
(3) the sum on the right side is taken over all $a,b$ such that $n=a+b$, $a\geq 0$ and $b \geq 1$.
\end{defn}

\begin{defn}\label{ApD3.4} An $A_\infty$-coalgebra over $K$ is a $\mathbb Z$-graded vector space $C=\underset{n \in \mathbb Z}{\bigoplus}C_n$ along with a  family of graded $K$-linear maps
$\{w_k: C\longrightarrow C^{\otimes k}\}_{k\geq 1}$ with $w_k$  of degree $2-k$, 
satisfying the following two conditions:

\smallskip
(a) For each $n\geq 1$, we have
$$\sum(-1)^{pq+r} (1^{\otimes p} \otimes w_q \otimes 1^{\otimes r})w_{p+1+r}=0, ~\forall \textrm{ }k \geq 1,$$
where $1$ denotes the identity map, and the sum is over all $p, q, r$ such that $n=p+q+r$, $p, r \geq 0$, $q \geq 1$.  

\smallskip
(b) The induced morphism $\underset{k\geq 1}{\prod} w_k: C\longrightarrow \underset{k\geq 1}{\prod}C^{\otimes k}$ factors through the direct sum $\underset{k\geq 1}{\bigoplus}C^{\otimes k}$.
\end{defn}

\begin{defn}\label{D3.5}  Let $C$ be an $A_\infty$-coalgebra over $K$ with structure maps $\{w_k: C\longrightarrow C^{\otimes k}\}_{k\geq 1}$. A (left) $A_\infty$-comodule over $C$ is a $\mathbb Z$-graded space $P$ with maps
$\{w_k^P: P \longrightarrow C^{\otimes k-1}\otimes P\}_{k\geq 1}$ with $w_k^P$
of degree $2-k$, satisfying, for each $n\geq 1$, we have
$$\sum(-1)^{pq+r} (1_C^{\otimes p} \otimes w_q\otimes 1^{\otimes r-1}_C \otimes 1_P) w_{p+1+r}^P+ \sum (-1)^{ab} (1^{\otimes a} \otimes w_b^P) w_{a+1}^P=0,$$
where the first sum is over all $p,q,r$ such that $n=p+q+r$, $p\geq 0$, $q,r\geq 1$ and the second sum is over all $a,b$ such that $a+b=n$, $a\geq 0$, $b\geq	 1$.

\smallskip
A morphism of left $A_\infty$-comodules $g: P\longrightarrow Q$ is a family of maps $g=\{g_l: P_l \longrightarrow Q_l\}_{l\in \mathbb Z}$ that is well behaved with respect to the structure maps of $P$ and $Q$. The category of left $A_\infty$-comodules over $C$ will be denoted by $^C\mathbf M$. Additionally, we will say that $P\in {^C}\mathbf M$ is strongly finite if 
the induced morphism $\underset{k\geq 1}{\prod} w_k^P: P\longrightarrow \underset{k\geq 1}{\prod}(C^{\otimes k-1}\otimes P)$ factors through the direct sum $\underset{k\geq 1}{\bigoplus}(C^{\otimes k-1}\otimes P)$. Similarly, we may define the category $\mathbf M^C$ of right $C$-comodules.
\end{defn}

\begin{defn}\label{D3.6} Let $C$ be an $A_\infty$-coalgebra. Let $P,Q$ be left $A_\infty$-comodules over $C$ with respective structure maps $\{
w_k^P: P \longrightarrow C^{\otimes k-1} \otimes P\}_{k\geq 1}$ and $\{
w_k^Q: Q \longrightarrow C^{\otimes k-1} \otimes Q\}_{k\geq 1}$. 
An $A_\infty$-morphism of comodules $g: P\longrightarrow Q$ is a sequence of graded morphisms $\{g_k: P \longrightarrow C^{\otimes k-1} \otimes Q\}_{k\geq 1}$, with $g_k$ of  degree $1-k$, satisfying for each $n\geq 1$:
$$\sum(-1)^{ab} (1_C^{\otimes a}\otimes w_b^Q) g_{a+1}+\sum (-1)^{ab+c} ( 1_C^{\otimes a} \otimes w_b \otimes 1_C^{\otimes c-1}\otimes 1_Q) g_{a+1+c}= \sum (1_C^{\otimes a} \otimes g_b) w_{a+1}^P,$$ where 

\smallskip
(1) the first sum on the left is over all $a,b$ such that $n=a+b$, $a\geq 0$, $b\geq 1$. 

\smallskip
(2) the second sum on the left is over all $a,b,c$ such that $n=a+b+c$, $a \geq 0$, $b,c \geq 1$ 

\smallskip
(3) the sum on the right side is over all $a,b$ such that $n= a+b$, $a\geq 0$, $b \geq 1$.
\end{defn}

\end{document}